\documentclass[12pt,a4paper,leqno]{amsart}

\usepackage[utf8]{inputenc}
\usepackage[T1]{fontenc}
\usepackage[english]{babel}
\usepackage{amsthm}
\usepackage{amscd}
\usepackage{amsfonts}         
\usepackage{amsmath}
\usepackage{amssymb}
\usepackage{hyperref}
\usepackage{enumerate}
\usepackage{tikz}
\usepackage{tikz-cd}
\usepackage{mathrsfs}
\usepackage{eucal}  
\usepackage{todonotes}

\newcommand{\Ab}{\mathbb{A}}

\newcommand{\Cb}{\mathbb{C}}

\newcommand{\Eb}{\mathbb{E}}
\newcommand{\Fb}{\mathbb{F}}
\newcommand{\Gb}{\mathbb{G}}
\newcommand{\Hb}{\mathbb{H}}

\newcommand{\Kb}{\mathbb{K}}

\newcommand{\Nb}{\mathbb{N}}

\newcommand{\Pb}{\mathbb{P}}
\newcommand{\Qb}{\mathbb{Q}}

\newcommand{\Tb}{\mathbb{T}}

\newcommand{\Vb}{\mathbb{V}}

\newcommand{\Zb}{\mathbb{Z}}

\newcommand{\Ac}{\mathcal{A}}

\newcommand{\Cc}{\mathcal{C}}
\newcommand{\Dc}{\mathcal{D}}
\newcommand{\Ec}{\mathcal{E}}

\newcommand{\Kc}{\mathcal{K}}

\newcommand{\Nc}{\mathcal{N}}
\newcommand{\Oc}{\mathcal{O}}
\newcommand{\Pc}{\mathcal{P}}

\newcommand{\Ls}{\mathscr{L}}

\newcommand{\Xfr}{\mathfrak{X}}

\newcommand{\pfr}{\mathfrak{p}}

\newcommand{\Lbf}{\mathbf{L}}

\newcommand{\PCob}{{\underline\Omega}}

\newcommand{\op}{\mathrm{op}}

\newcommand{\CH}{\mathrm{CH}}

\newcommand{\Spec}{\mathrm{Spec}}
\newcommand{\Gr}{\mathrm{Gr}}

\newcommand{\Td}{\mathrm{Td}}

\newcommand{\colim}{\mathrm{colim}}

\newcommand{\Fun}{\mathrm{Fun}}
\newcommand{\hook}{\hookrightarrow}

\newcommand{\Id}{\mathrm{Id}}
\newcommand{\Th}{\mathrm{Th}}

\newcommand{\red}{\mathrm{red}}

\newcommand{\xto}{\xrightarrow}
\newcommand{\dash}{{\text -}}

\newcommand{\dSch}{\mathrm{dSch}}

\newcommand{\MGL}{\mathrm{MGL}}
\newcommand{\tto}{\twoheadrightarrow}

\newcommand{\MS}{\mathrm{MS}}
\newcommand{\SH}{\mathrm{SH}}
\newcommand{\syn}{\mathrm{syn}}
\newcommand{\Mod}{\mathrm{Mod}}
\newcommand{\Sm}{\mathrm{Sm}}

\newcommand{\Ext}{\mathrm{ext}}
\newcommand{\Sh}{\mathrm{Sh}}

\newcommand{\Ani}{\mathrm{Ani}}
\newcommand{\Cat}{\mathrm{Cat}}

\newcommand{\maps}{\mathrm{maps}}

\newcommand{\cdh}{\mathrm{cdh}}

\newcommand{\afp}{\mathrm{afp}}
\newcommand{\fp}{\mathrm{fp}}

\newcommand{\mot}{\mathrm{mot}}
\newcommand{\dbe}{\mathrm{dbe}}
\newcommand{\sbe}{\mathrm{sbe}}
\newcommand{\Sch}{\mathrm{Sch}}
\newcommand{\Nis}{\mathrm{Nis}}
\newcommand{\Sp}{\mathrm{Sp}}

\newcommand{\Zar}{\mathrm{Zar}}

\newcommand{\Fin}{\mathrm{Fin}}
\newcommand{\SSeq}{\mathrm{SSeq}}
\newcommand{\lax}{\mathrm{lax}}

\newcommand{\CAlg}{\mathrm{CAlg}}
\newcommand{\pt}{\mathrm{pt}}
\newcommand{\procdh}{\mathrm{procdh}}
\newcommand{\st}{\mathrm{st}}
\newcommand{\Sq}{\mathrm{Sq}}
\newcommand{\Q}{\mathrm{Q}}
\newcommand{\ebe}{\mathrm{ebe}}
\newcommand{\gys}{\mathrm{gys}}
\newcommand{\tot}{\mathrm{tot}}
\newcommand{\h}{\mathrm{h}}

\newcommand{\dR}{\mathrm{dR}}

\newtheorem{theo}{Tplottin ubuntuheorem}
\theoremstyle{plain}
\newtheorem{thm}[theo]{Theorem}
\newtheorem{lem}[theo]{Lemma}
\newtheorem{prop}[theo]{Proposition}
\newtheorem{cor}[theo]{Corollary}
\newtheorem{con}[theo]{Conjecture}
\newtheorem{obs}[theo]{Observation}

\newtheorem*{thm*}{Theorem}
\newtheorem*{lem*}{Lemma}
\newtheorem*{prop*}{Proposition}
\newtheorem*{cor*}{Corollary}
\newtheorem*{var*}{Variant}
\newcommand{\comp}{{{\kern -.5pt}\wedge}}
\newcommand{\et}{\text{\'{e}t}}

\theoremstyle{definition}
\newtheorem{defn}[theo]{Definition}
\newtheorem{ex}[theo]{Example}
\newtheorem{cons}[theo]{Construction}
\newtheorem{rem}[theo]{Remark}
\newtheorem{var}[theo]{Variant}

\newtheorem{quest}[theo]{Question}
\renewcommand{\P}{\mathrm{P}}
\newcommand{\B}{\mathrm{B}}
\numberwithin{theo}{section}

\title{Motivic Steenrod operations at the characteristic via infinite ramification}
\author{Toni Annala}
\address{Department of Mathematics, University of Toronto, Toronto.}
\email{toni.annala@utoronto.ca}
\author{Elden Elmanto}
\address{Department of Mathematics, University of Toronto, Toronto.}
\email{elden.elmanto@utoronto.ca}

\date{\today}

\begin{document}

\maketitle

\begin{abstract}
We construct motivic power operations on the mod-$p$ motivic cohomology of $\Fb_p$-schemes using a motivic refinement of Nizio{\l}'s theorem. The key input is a purity theorem for motivic cohomology established by Levine. Our operations satisfy the expected properties  (naturality, Adem relations, and the Cartan formula) for all bidegrees, generalizing previous results of Primozic which were only know along the ``Chow diagonal.'' We offer geometric applications of our construction: 1) an example of non-(quasi-)smoothable algebraic cycle at the characteristic, 2) an answer to the motivic Steenrod problem at the characteristic, 3) a counterexample to the integral version of a crystalline Tate conjecture. 
\end{abstract}

\tableofcontents

\section{Operations? Again?} Algebraic topologists have made algebraic geometers (perhaps begrudgingly) contend with symbols like $\Sq^i, \P^j$ and the likes floating around their cohomology theories, at least whenever there is some torsion. One of the first instances of this was the work of Atiyah and Hirzebruch \cite{atiyah-hirzebruch} in 1964 where a counterexample to the integral Hodge conjecture was produced via the Steenrod operations. In a similar thread, Hartshorne, Rees and Thomas \cite{hartshorne-rees-thomas} proved non-smoothability results for integral algebraic cycles.

At this point, the algebraic geometer might be rightly annoyed --- it seems that all the topologist does is to spoil algebro-geometric dreams using these hieroglyphic-like symbols. Then comes Voevodsky. The motivic Steenrod operations on mod-$\ell$ motivic cohomology in characteristic zero \cite{voevodsky:2003a} lie at the heart of the proof of the Bloch-Kato conjectures \cite{voevodsky:2003b,voevodsky:2011}. In fact, the shape of these operations is at the heart of many issues of a good theory of motives and motivic homotopy. Another notable result is Feng's thesis \cite{feng-steenrod} where he used Steenrod operations on \'etale cohomology to settle a conjecture of Tate's on the Artin-Tate pairing on Brauer groups away from the residue characteristic.

\subsection{Operations and the statement of the main theorem} \label{sec:operations-main} Let us now explain the format of motivic cohomology operations. Let $p$ be a prime, $F$ a field. The formation of mod-$p$ motivic cohomology, assembles into a functor valued in bigraded commutative rings (with the cup product $\cup$ as the multiplicative structure):
\[
H^{\star}_{\mot}(-,\Fb_p(\star)) \colon \Sm_F^{\mathrm{op}} \to \mathrm{BiGrCAlg}.
\]
There is a canonical class
\[
c \in H_{\mot}^2(\mathbb{P}_F^1,\Fb_p(1)) (\cong \mathrm{Pic}(\mathbb{P}_F^1)/p),
\]
classifying the mod-$p$ reduction of the first chern class of $\mathcal{O}(1)$. A \emph{bistable cohomology operation} (of degree $(r,s)$) \cite[(2.13) \& Proposition 2.16]{voevodsky:2003a} is a collection of natural transformation 
\[
P_{ij}: H^{i}_{\mathrm{mot}}(-,\Fb_p(j)) \rightarrow H^{i+r}_{\mathrm{mot}}(-,\Fb_p(j+s)),
\] which commutes with multiplication by $c$:
\[
P_{ij}(x \cup c) = P_{i+2,j+1}(x) \cup c. 
\]

This definition is a direct analog of the definition of stable cohomology operations for singular cohomology. The main goal of this paper is to construct certain canonical, non-trivial bistable operations when $F$ is a field of characteristic $p > 0$. These operations are even constructed on the extension of mod-$p$ motivic cohomology to all noetherian $\Fb_p$-schemes, defined and studied by the second author and Morrow \cite{elmanto-morrow}, as well as Kelly and Saito \cite{kelly-saito}. The following is our main theorem. 

\begin{thm}\label{thm:main1} Let $p$ be a prime and $F$ be a prime field (in other words, either $\Qb$ or $\Fb_p$). There exists natural transformations of presheaves of motivic cohomology groups on Noetherian $\Fb_p$-schemes:
\[
\P^n_{ij}: H^i_{\mot}(-,\Fb_p(j)) \to H^{i+2n(p-1)}_{\mot}(-,\Fb_p(j+n(p-1)))
\]
and
\[
B_{ij}^n \colon H_{\mot}^i(-; \Fb_p(j)) \to H_{\mot}^{i + 2n(p-1) + 1}(-; \Fb_p(j + n(p-1)).
\]
Moreover, these operations satisfy the expected properties, namely:
\begin{enumerate}
\item $\P^0 = \Id$ and $\B^i = \beta \P^i$;
\item the motivic Adem relations hold;
\item the motivic Cartan formulas hold;
\item if $x \in H_{\mot}^{2i}(X; \Fb_p(i))$, then $\P^i(x) = x^p$;
\item if $y \in H^{j}_{\mot}(X; \Fb_p(k))$ is such that $j-k<i$ and $k \leq i$, then $\P^i(y) = 0$;
\item the Bockstein $\beta$ is a graded derivation with respect to the first grading.
\end{enumerate}

\end{thm}

We refer the reader to Theorem~\ref{thm:ModPSteenrodSpectrumLevel} for the motivic Adem and Cartain relations. 

\subsection{Operations and endomorphisms} We now explain our theorem in the context of motivic homotopy theory. Let $H\Fb_{p,F}$ be the motivic spectrum representing mod-$p$ motivic cohomology in Morel--Voevodsky's category $\SH_F$ of motivic spectra over $F$; this motivic spectrum could be built as Bloch's cycle complex \cite{bloch:1986} or via Voevodsky's theory of finite correspondences \cite{VSF}. As explained in the introduction to \cite{HKO}, there are three related, but \emph{a priori} different algebra of bigraded operations in motivic homotopy theory summarized in the following zig-zag:
\[
\mathcal{A}_F^{\star,\star} \hookrightarrow \mathcal{M}_F^{\star,\star} \twoheadleftarrow  H\Fb_{p,F}^{\star,\star}H\Fb_{p,F}. 
\]
Here:
\begin{enumerate}
\item $\mathcal{M}_F^{\star,\star}$ is the algebra of bistable operations in motivic cohomology as explained above in \S\ref{sec:operations-main}.
\item $H\Fb_{p,F}^{\star,\star}H\Fb_{p,F}$ is the bigraded ($\SH_F$-linear) endomorphism ring of $H\Fb_{p,F}$.
\item Finally, $\mathcal{A}_F^{\star,\star}$ is, prior to this paper, far \emph{only defined if $p$ is invertible in $F$}. It is the subalgebra of $\mathcal{M}_F^{\star,\star}$ generated by power operations as constructed by Voevodsky in \cite{voevodsky:2003a}\footnote{Of note are also Brosnan's Steenrod operations for Chow groups \cite{brosnan:2003} built in a similar manner to Voevodsky.} ; nowadays we know that these operations arise from the norm structure on motivic cohomology in the sense of \cite{bachmann-hoyois}; see \cite[Example 7.25]{bachmann-hoyois} for details. 
\end{enumerate}

The main result of \cite{HKO} proves that all three objects are isomorphic when $p$ is invertible in $F$; generalizing the case when $F$ is characteristic zero by Voevodsky in \cite{voevodsky:2003a}. The main result of this paper is that $\mathcal{A}^{\star,\star}$ also makes sense mod-$p$:

\begin{thm}\label{thm:main1} There is a canonical map of bigraded, associative $\Fb_p$-algebras
\begin{equation}\label{eq:map}
\mathcal{A}_{\mathbb{Q}}^{\star,\star} \rightarrow H\Fb_{p,\Fb_p}^{\star,\star}H\Fb_{p,\Fb_p}. 
\end{equation}
The image of the Bockstein $\beta$ and the power operations $\mathrm{P}^i$ remain linearly independent and its image has a $H^{\star}_{\mot}(\Fb_p;\Fb_p(\star))$-linear basis given by the admissible monomials in the Bocksteins and the power operations. 
\end{thm}

We refer the reader to Corollary~\ref{cor:admiss} for a precise statement on the image of $\mathcal{A}_{\mathbb{Q}}^{\star,\star}$ in $H\Fb_{p,\Fb_p}^{\star,\star}H\Fb_{p,\Fb_p}$; we denote this image by 
\[
\mathcal{A}^{\star,\star}_{\Fb_p} \subset H\Fb_{p,\Fb_p}^{\star,\star}H\Fb_{p,\Fb_p}.
\]

\begin{rem}[Power operations and norms]\label{rem:power} Ultimately, the construction of power operations rely on the classifying stack $B\Sigma_n$ where $\Sigma_n$ is the symmetric group on $n$-letters. In effect, Voevodsky in \cite{voevodsky:2003a} constructs the $n$-th total power operation as the dotted arrow in the following diagram of spectra\footnote{We employ the following notation: for any scheme $S$ we have an adjunction $\sigma^{\infty}: \Sh_{\Nis, \Ab^1}(\Sm_S; \Sp)\rightleftarrows \SH_S:\omega^{\infty}$. The domain of $\sigma^{\infty}$ is the $\infty$-category of $\Ab^1$-invariant Nisnevich sheaves, also commonly referred to as ``$S^1$-spectra'' in the literature. We will also use the further adjunction
\[
\Sigma^{\infty}_{\Tb}: \Sh_{\Nis, \Ab^1}(\Sm_S; \mathrm{Ani})_{\star} \rightleftarrows \SH_S: \Omega^{\infty}_{\Tb},
\]
between pointed $\Ab^1$-invariant Nisnevich sheaves of anima and $\SH_S$.}
\[
\begin{tikzcd}
\omega^{\infty}H\Fb_p(X) \ar[dashed]{r}{\P_n}  \ar[swap]{dr}{x \mapsto x^n} & \omega^{\infty}H\Fb_p(X \times B\Sigma_n)\ar{d}\\
& \omega^{\infty}H\Fb_p(X).
\end{tikzcd}
\]
The problem is that $BC_p$ is $\Ab^1$-contractible in characteristic $p > 0$, hence one cannot proceed in this manner \cite[Proposition~5.5]{BEHKSY:2021}.\footnote{Unless otherwise specified, all quotient stacks are taken in the fppf-topology.} In principle, one needs to replace $BC_p$ by the stack $B\mu_p$ and one wants a ``total $\mu_p$-power operation'': $\omega^{\infty}H\Fb_p(X) \rightarrow \omega^{\infty}H\Fb_p(X \times B\mu_p)$ out of which one can extract elements of $\mathcal{A}^{\star,\star}$. One approach is to endow a ``flat norm'' structure on $H\Fb_p$, extending the \'etale norms of \cite{bachmann-hoyois}. We do not yet have access to this technology. 
\end{rem}

\begin{rem}[An element called $\tau$] We comment on the kernel of the map~\eqref{eq:map}. Let us recall that if $F$ is a field of characteristic not $p$ and furthermore admits a single $p$-th root of unity, then we have an element $\tau_p \in H_{\mot}^0(F; \Fb_p(1)) = \mu_p(F)$ classifying a primitive $p$-th root of unity. Noting the Nesterenko--Suslin isomorphism: $K_{j}^M(F) \cong H_{\mot}^{j}(F, \Zb(j))$, the resolution of the Bloch--Kato conjecture implies that we have an isomorphsim of bigraded rings:
\[
K_{\star}^M(F)/p[\tau_p] \cong H_{\mot}^{\star}(F, \Fb_p(\star)). 
\]
Combining this with the structure of the $\mathcal{A}^{\star,\star}_F$ (see, for example, \cite[Theorem 1.1(1)]{HKO}) we can write down a $\Fb_p$-linear basis for $\mathcal{A}^{\star,\star}_F$ which constitutes $\tau$-multiples of the Bockstein and power operations. For each prime $p$, there exists an integer $r > 0$ and an element $\widetilde{\tau}_p \in H_{\mot}^{0}(\mathbb{Q}, \Fb_p(r))$ which maps to the $r$-th power of $\tau_p$; see \cite[Section 6.1]{elso} for the exact numerology. The map~\eqref{eq:map} eliminates $\tilde{\tau}_p$. This is explained by the mod-$p$ counterpart of the Bloch--Kato conjectures which is a theorem of Geisser--Levine: for all fields of characteristic $p > 0$, we have that:
\[
K_{\star}^M(F)/p \cong H_{\mot}^{\star}(F, \Fb_p(\star)). 
\]
In particular, $H_{\mot}^{0}(\Fb_p, \Fb_p(\star)) = 0$ unless $\star = 0$. Part of the content of Theorem~\ref{thm:main1} is that even though $\widetilde{\tau}_p$-multiples of the Bockstein and power operations are killed by~\eqref{eq:map}, those without such multiples survive and retains their independence relations.
\end{rem}


\begin{rem}[Other fields of characteristic $p$] Let $\Fb_p \subset F$ be an extension of fields of characteristic $p > 0$. We define the bigraded Steenrod algebra over $F$ to be the base change:
\[
\mathcal{A}^{\star,\star}_{F} := \mathcal{A}^{\star,\star}_{\Fb_p} \otimes_{H^{\star}_{\mot}(\Fb_p;\Fb_p(\star))} H^{\star}_{\mot}(F;\Fb_p(\star))\]
By design, $\mathcal{A}^{\star,\star}_{F}$ is spanned by  $H^{\star}_{\mot}(F;\Fb_p(\star)) = K^M_{\star}(F)/p$-linear combinations of Bockstein and Power operations. These act naturally on the mod-$p$ motivic cohomology of smooth $F$-schemes, governed by the Adem and Cartan relations. We also refer the reader to the discussion in \S\ref{subsect:allPops} for an extension of the motivic Steenrod operations acting on the $\Ab^1$-invariant motivic cohomology of an arbitrary quasicompact, quasiseparated $\Fb_p$-schemes. 
\end{rem}

\subsection{Relations to \cite{primozic:2020} and \cite{SpitzweckFrankland}}

 The goal of this paper is to construct power operations on mod-$p$ motivic cohomology of $\Fb_p$-schemes, satisfying the analogs of the relations in topology and mod-$\ell$ motivic cohomology. Our approach is to directly construct a graded-$\mathbb{E}_1$-map
\[
\mathrm{end}_{\SH_{\Kb}}(H\Fb_p, \Tb^{\otimes \star} \otimes H\Fb_p) \rightarrow \mathrm{end}_{\SH_{\Fb_p}}(H\Fb_p, \Tb^{\otimes \star} \otimes H\Fb_p)
\]
between endomorphism spectra, where $\Kb = \Qb_p[p^{1/p^\infty}]$ is an extension of $\Qb_p$ that is infinitely ramified at $p$. Here $\Tb$ denotes the \emph{Tate motive}, also known as the motive of the reduced projective line in $\SH$. Throughout the article, we will denote by $\Oc = \Zb_p[p^{1/p^\infty}]$.

The work of  Primozic \cite{primozic:2020}, relying on the work of Frankland--Spitzweck \cite{SpitzweckFrankland} constructs power operations $\P^i: H_{\mot}^{m}(-;\Fb_p(n)) \rightarrow H_{\mot}^{m+2i(p-1)}(-;\Fb_p(n+i(p-1))$ which satisfy the Adem and Cartan formulas on the Chow line diagonal. Let us explain how this is done. 

Chronologically, the Steenrod algebra in topology came first before Milnor discovered the extremely deep fact that the dual Steenrod algebra is \emph{free}. In the mod-$p$ motivic story, the roles are reversed. Learning from the lessons of history, we can try to study the dual motivic Steenrod algebra by guessing a basis for the dual Steenrod algebra, and then dualize to get the power operations. Indeed, let $\mathcal{A}_{\star,\star} := \pi_{\star,\star}H\Fb_p \otimes H\Fb_p$ then there are elements in $\mathcal{A}_{\star,\star}$ with the following names and bigradings:
\[
\tau_i, |\tau_i| = (2p^i-1, p^i-1) \qquad \zeta_j, |\zeta_j|=(2p^j-2, p^j-1). 
\]
They are produced by the coaction on $B\mu_p$ as in \cite[Corollary 10.25]{spitzweck:2018}; this is supposed to be ``dual'' to the hypothetical norm alluded to in Remark~\ref{rem:power}. Suppose that we are given a sequence $\alpha:=(\epsilon_0, r_1, \epsilon_1, r_2, \cdots)$ where $\epsilon_i \in \{0, 1\}$ and $r_j \geq 0$. Then we have a monomial $\omega(\alpha) = \tau_0^{\epsilon_0}\zeta_1^{r_1}\tau_1^{\epsilon_1}\cdots$ in $\mathcal{A}_{\star,\star}$. These monomials then specify a map
\begin{equation}\label{eq:psi}
\Psi: \bigoplus_{\alpha} \Tb^{\otimes q_{\alpha}} \otimes H\Fb_p[p_{\alpha} - 2q_{\alpha}] \rightarrow H\Fb_p \otimes H\Fb_p. 
\end{equation}
Over fields where $p$ is invertible, this map is an equivalence by \cite{voevodsky:2003a,HKO}, but this remains unknown over characteristic $p > 0$ fields. Since we do not know how exactly the dual Steenrod algebra looks like, we cannot really define the power operations by (re)taking duals. Nonetheless, the main result of \cite{SpitzweckFrankland} shows that $\Psi$ is a \emph{split monomorphism} over any field. This suggests that there is some hope in producing power operations by taking duals.

This strategy was executed by Primozic. Let $i: \Spec(\Fb_p) \hookrightarrow \Spec(\Zb_p)$ and $j: \Spec(\Qb_p) \hookrightarrow  \Spec(\Zb_p)$. Then we have a splitting in $\SH(\Fb_p)$\footnote{Actually everything works over any discrete valuation ring of mixed characterisitc $(0, p)$.}: 
\begin{equation}\label{eq:splitting}
i^*j_*H\Fb_p \simeq H\Fb_p \oplus \Tb^{\otimes -1}H\Fb_p[1],
\end{equation}
which is also ultimately responsible for the splitting of $\Psi$ discussed above. This result is a consequence of absolute purity for the pair $(\Spec(\Zb_p), \Spec(\Fb_p))$; the latter is an ingredient we also use in our approach. Writing $\pi: i^*j_*H\Fb_p  \rightarrow H\Fb_p$ for the projection onto the first coordinate, Primozic sets up a map of graded abelian groups: $\Phi: H\Fb_{p,\Zb_p}^{\star,\star}H\Fb_{p,\Zb_p} \rightarrow H\Fb_{p,\Fb_p}^{\star,\star}H\Fb_{p,\Fb_p}$ by taking an element $f \in H\Fb_{p,\Zb_p}^{\star,\star}H\Fb_{p,\Zb_p}$ and sending it to the composite:
\begin{equation}\label{eq:phi(f)}
\begin{tikzcd}
\Phi(f): H\Fb_p \ar{r}{i^*\mathrm{unit}} &  i^*j_*H\Fb_p \ar{r}{i^*j_*(f)} & \Sigma^{a,b}i^*j_*H\Fb_p \ar{r}{\pi} & \Sigma^{a,b}H\Fb_p.
\end{tikzcd}
\end{equation}
So, for example, taking $f = \P^n$ constructs $\Phi(\P^n)$ which he declares to be the $n$-th power operation. He then verifies the usual Adem and Cartan relations, but can only do so \emph{up to error terms} which disappears on classes along the Chow line. These error terms has to do with the fact that $\Psi(f)$ involves projecting $i^*j_*H\Fb_p$ onto the summand without the shift via $\pi$. 

The goal of this paper and the punchline of our story is that: \emph{we can delete the orthogonal complement by going very deeply ramified}. The arithmetically-minded audience is surely unsurprised by this, by now, common technique in arithmetic geometry. 

\subsection{Applications} We now discuss some applications of our Steenrod operations. 

\subsubsection{Syntomic realizations}

Beyond verifying the expected Adem and Cartan relations outside of the Chow line, let us highlight one immediate virtue in our approach to motivic Steenrod operations: it also produces syntomic cohomology operations. An environment to do motivic homotopy theory over base scheme $S$ without $\Ab^1$-invariance was introduced by the first author and Iwasa \cite{annala-iwasa:MotSp} and later developed by the first author, Hoyois and Iwasa \cite{AHI}. In particular, we have a fully faithful embedding $\SH_S \subset \MS_S$. This environment is christened \emph{motivic spectra} and is denoted by $\MS_S$. One immediate advantage of $\MS_S$ is that in characteristic $p > 0$, its \'etale sheafification is not $p$-adically zero, this is in stark contrast to the $\Ab^1$-invariant situation as explained in \cite[Appendix A]{bachmann-hoyois-et} (see also \cite[Lemme 3.10]{ayoub-realisation} and \cite[Proposition A.3.1]{cisinski-deglise-etale}). Indeed, $\MS^{\text{\'{e}t}}_S$, the \'etale-sheafified version of motivic spectra, seems to provide a good theory of \'etale motivic spectra in all characteristics and provides a setting where Clausen's Selmer $K$-theory is representable \cite[Section 5.4]{annala-iwasa:MotSp}.

Étale sheafification in $\MS$ produces a map of graded $\Eb_1$-algebras
\[
\mathrm{end}_{\SH_{\Fb_p}}(H\Fb_p, \Tb^{\otimes \star} \otimes H\Fb_p) \simeq \mathrm{end}_{\MS_{\Fb_p}}(H\Fb_p, \Tb^{\otimes \star} \otimes H\Fb_p) \rightarrow \mathrm{end}_{\MS_{\Fb_p}}(H\Fb_p^{\syn}, \Tb^{\otimes \star} \otimes H\Fb_p^{\syn}),
\]
where $H\Fb_p^{\syn}$ denotes the motivic spectrum that represents the mod-$p$ syntomic cohomology (see Proposition~\ref{prop:etale}). From this, we obtain mod-$p$ power operations $\P^i$ acting on the mod-$p$ syntomic cohomology, and in particular an algebra $\Ac^{\star, \star}_\syn$ acting on $H\Fb_p^{\syn}$ that is generated by the $\P^i$ and the Bockstein $\beta$ modulo the Adem relations. This, \emph{syntomic Steenrod algebra} has been discovered independently in an upcoming work of Shachar Carmeli and Tony Feng \cite{carmeli-feng}, where they use the syntomic power operations to prove the last remaining case of the Tate conjectures on the Artin-Tate pairing on Brauer groups.

\subsubsection{Geometric applications in positive characteristics} A general type of problems that Steenrod operations have been used to address takes the following form:

\begin{quest}\label{quest:geometric} If $H^n(X)$ is a cohomology theory on $k$-variety, is there an ample supply of geometric classes that control $H^n(X)$?
\end{quest}

We outline answers to the above question in three different, but not unrelated settings.
\begin{enumerate}
\item If $X$ is a smooth $k$-scheme and $H^n(X) = \mathrm{CH}^n(X) \otimes \Lambda$ is the Chow group of codimension $n$-cycles with coefficients in a ring $\Lambda$, one asks if a cycle can be represented by a linear combination of cycles whose components are \emph{nonsingular}; this is an algebro-gemetric counterpart to the topological question of Borel--Haefliger \cite{borel-haefliger}. This was resolved in the negative by \cite{hartshorne-rees-thomas} over the complex numbers and with integral coefficients. The problem has received renewed attention in recent years when $\Lambda = \Qb$ due to the seminal work of Kollar and Voisin \cite{kollar-voisin} which proves a positive result for cycles below the middle dimension. 
\item The oldest of these is Steenrod's realization problem \cite{eilenberg}: here $X$ is possibly singular complex variety and $H^n = H^n_{\mathrm{sing}}(-;\Lambda)$ is singular cohomology with some coefficient ring $\Lambda$. A variant of the question asks if every class can be represented as the pushforward of the fundamental class of a manifold mapping into $X$. With mod-$p$ coefficients, this was settled in the negative in Thom's thesis \cite{thom-thesis}.
\item Let $X$ be a smooth projective variety over the complex numbers or a finitely generated field. Over the complex numbers we have the Hodge conjecture, in which case $H^{n} = H^{n,n}$ the group of Hodge cycles. Otherwise, over certain fields, we have the Tate conjecture in which case $H^n$ is the group of Tate cycles, defined as the subgroup of geometric $\ell$-adic cohomology which are fixed by the Galois group. These deep conjectures ask that these groups are spanned by algebraic cycles. There is an industry of counterexamples to the integral versions of these conjectures, beginning with Atiyah--Hirzebruch's counterexamples \cite{atiyah-hirzebruch}; we also note Schoen's formulation of the integral version of the Tate conjectures in \cite{schoen} and the $\ell$-adic analog of Atiyah--Hirzebruch's counterexamples by Colliot-Th\'{e}l\`ene--Szamuely \cite{ct-szamuely}.
\end{enumerate}

In \S\ref{sec:appl} we investigate mod-$p$ variants of all of the above questions. We provide integral counterexamples to optimistic versions of these problems based on the construction of the Steenrod operations. We remark that our counterexamples are \emph{somewhat different} from the classical/$\ell$-adic ones (namely, the ones constructed in \cite{hartshorne-rees-thomas, thom-thesis, atiyah-hirzebruch, ct-szamuely}), though they do build on the geometric insights of these works. In particular our counterexamples highlight some of the new phenomena that occurs when working with $p$-adic coefficients.

\subsubsection{Wu's formula} One of the key results about classical power operations is the \emph{Wu formula}: writing $\P := \sum \P^0 + \P^1 + \cdots + \P^n + \cdots$ for the total power operations on cohomology it measures the difference between $\P$ and the pushforward $f_!$ defined for an appropriate map of topological spaces. Indeed, the characteristic class known as \emph{Wu classes} measures the difference between $\P(f_!)$ and $f_!(\P)$. Primozic proved the Wu formula was for the Chow groups in \cite[Proposition 7.1]{primozic:2020}. Our approach upgrades this to all of motivic and syntomic cohomology.

\begin{thm}\label{thm:wu-intro} Let $X \rightarrow Y$ be projective, quasi-smooth morphism over a common base $\Fb_p$-scheme $S$. the formula
\[
\P(f_!(x)) = f_!(w(\Lbf_{X/Y}) \cdot \P(x)).
\]
holds in both mod-$p$ syntomic and motivic cohomology; here $w$ denotes the total \emph{total Wu class} of a virtual vector bundle/perfect complex.
\end{thm}

Theorem~\ref{thm:wu-intro} is proved as Corollary~\ref{cor:total-power-ops} after an appropriate definition of the Wu class in Remark~\ref{rem:todd-vs-wu}. In a sequel to this paper, this result will be key in investigating the difference between various notions of the coniveau filtration on crystalline cohomology.

\subsection{A comment on the Hopkins--Morel isomorphism}

We finish off this introduction with a few words about our original motivation in writing this paper (which we, unfortunately, did not manage to do).  Many specialists in motivic homotopy theory, one way or another, are trying to study the following conjecture:
\begin{con}\label{conj:steenrod-span} For all fields, $F$, the injection $\mathcal{A}^{\star,\star}_{F} \subset H\Fb_{p,F}^{\star,\star}H\Fb_{p,F}$ is in fact an equality. 
\end{con}
This is related to the following central question in motivic stable homotopy theory.

\begin{con}[The Hopkins--Morel isomorphism aka the motivic Quillen theorem]\label{conj:Hopkins--Morel} Let $F$ be a field. Then the map in $\SH_F$:
\[
\colim_n \MGL/(a_1, \cdots, a_n) \rightarrow H\Zb,
\]
is an equivalence. Here the symbols $a_i$ are generators of the Lazard ring. 
\end{con}

Because $H\Zb$ is stable under base change, it suffices to prove Conjecture~\ref{conj:Hopkins--Morel} over prime fields. Conjecture~\ref{conj:Hopkins--Morel} has been resolved after inverting the exponential characteristic of $k$ thanks to work of Hoyois \cite{hoyois:2013}; in characteristic zero this was proved in unpublished work of Hopkins and Morel. Among other things, Conjecture~\ref{conj:Hopkins--Morel} builds motivic filtrations on algebraic cobordism and other related theories (those which are Landweber exact) \cite{spitzweck1,spitzweck2}; work of Levine \cite{levine-comparison} also determines the slices of the motivic sphere spectrum up to this conjecture. 

The method of proof in \cite{hoyois:2013} essentially shows that Conjecture~\ref{conj:steenrod-span} implies Conjecture~\ref{conj:Hopkins--Morel}. We explain how the methods of the present paper makes precise the relationship between these conjectures and absolute purity for the absolute motivic spectra:
\[
E = H \Fb_p \otimes H\Fb_p \qquad \text{or} \qquad \colim_n \MGL/(p, a_1, \cdots, a_n).
\]
We are very grateful to Jacob Lurie for clarifying our thoughts on these matters.

In the notation of \S\ref{sec:purity}, if absolute purity for either of these two theories can be established just for $i: X/\pi \hookrightarrow X$ where $X$ is smooth over $D$, a mixed characteristic $(0,p)$ discrete valuation ring with uniformizer $\pi$ then Theorem~\ref{thm:Detroit} to proves an equivalence
\[
E/p \xrightarrow{\simeq} j_*E_{\Kb}/p
\]
for $j \colon \Spec(\Zb_p[p^{1/p^\infty}, p^{-1}]= \Kb) \hook \Spec(\Zb_p[p^{1/p^\infty}] )$. Since $\Kb$ is a characteristic zero field, we know exactly how $E/p_{\Kb}$ looks either by \cite{voevodsky:2003a} (which proves that the map~\eqref{eq:psi} is an equivalence) or \cite{hoyois:2013} (which identifies the colimit with $H\Fb_p$). Since all the objects in sight are stable under base change, we can then establish Conjecture~\ref{conj:steenrod-span} or Conjecture~\ref{conj:Hopkins--Morel} over $\Fb_p$. Unfortunately, the authors of the present paper have no strategy to prove absolute purity in either of these cases.
%
%
%

\subsection{Notation} In the introduction, we have denoted by $H \Zb, H\Fb_p$, the $\Ab^1$-invariant motivic cohomology extracted from $\Ab^1$-invariant motivic stable homotopy theory. To avoid confusion with the non-$\Ab^1$-invariant theories which will appear in \S\ref{sec:appl} we will denote these $\Ab^1$-invariant motivic spectra by $H\Zb^{\Ab^1}, H\Fb_p^{\Ab^1}$ and so on, starting from \S\ref{sec:background}.

In $\SH_X$ we denote by $\Tb_X \in \SH_X$ the \emph{Tate motive}, i.e., the motive of reduced $\Pb^1$. We write out Tate twists and shifts in full, following \cite{bem}. In usual motivic homotopy theory notation:
\[
\Tb^{\otimes q}_X \otimes E[p-2q] = \Sigma^{p,q}E. 
\]
We often omit the subscript $X$ from the notation if $X$ is clear from the context. We make free use of the Thom spectrum construction
\[
\mathrm{Th}_X: K(X) \rightarrow \mathrm{Pic}(\SH_X) \subset \SH_X;
\]
noting that if $\Ec$ is a locally free sheaf of finite rank on $X$, then $\mathrm{Th}_X(\Ec)$ is the cofibre in $\SH(X)$ of the map induced by the hyperplane inclusion $\Pb_X(\Ec) \rightarrow \Pb_X(\Ec \oplus \Oc)$. 

In \S\ref{sec:appl} we use the theory of non-$\Ab^1$-invariant motivic spectra $\MS_X$ as developed in \cite{annala-iwasa:MotSp, AHI}. In this situation, for a qcqs equicharacteristic scheme $X$  we write $H \Zb, H\Fb_p \in \MS_X$ for the non-$\Ab^1$-invariant motivic cohomology constructed by the second author and Morrow over fields \cite{elmanto-morrow}. Independently, Kelly--Saito \cite{kelly-saito} developed a related theory for noetherian $\Fb_p$-schemes of finite Krull dimension, defined via the procdh-local left Kan extension of $\Ab^1$-invariant motivic cohomology from smooth schemes, and showed that the two theories agree in this generality. Both perspectives will play a role in what follows. Singular schemes will appear only starting in \S\ref{sec:againstmgl}. Note, however, that $X$ is furthermore assumed to be a regular noetherian scheme (e.g. $S$ is a field) then $H \Zb \in \SH_S \subset \MS_X$ and coincides with $H\Zb^{\Ab^1}$ by \cite[Theorem 6.1]{elmanto-morrow}. We also have the Thom spectrum construction in $\MS_X$:
\[
\mathrm{Th}_X: K(X) \rightarrow \mathrm{Pic}(\MS_X) \subset \MS_X,
\]
whose values on an locally free sheaf is the same one described in $\SH_X$. Furthermore, $\SH_X$ and $\MS_X$ are naturally enriched over spectra and we write $\mathrm{maps}(E, F)$ (resp. $\mathrm{end}(E)$) for the spectrum of maps (resp. endomorphisms) between two objects.

Lastly, all schemes in this paper are quasicompact and quasiseparated and we denote by $\Sch$ the category of qcqs schemes. 

\subsection*{Acknowledgments}

This project was initiated when both authors were in residence at the Institute for Advanced Study, Princeton. TA was supported by supported by the National Science Foundation Grant No. DMS-1926686. EE was supported by an Erik Ellentuck fellowship from the Institute of Advanced Study, and the NSERC Discovery grant RGPIN-2025-07114. 

We would like to express gratitude for the Institute for providing excellent working conditions. Special thanks go to the staff at Rubenstein Commons --- Mike and Brendan --- who provided a friendly atmosphere to work late into the Princeton evenings and a good supply of libations for optimism and inspiration. Ideas around a motivic Nizio{\l} theorem came to EE somewhere between a sandwich run in Tatte and a biology lab hosted at the ``EXP'' at Northeastern. EE thanks ER for facilitating the environment and whimsey for this insight to occur. 

Lastly, we would also like to thank Tony Feng for discussions about his work on Steenrod operation and Tate's conjectures on the Brauer group, Jacob Lurie for bringing to our attention the relevance of infinite/deep ramification, Matthew Morrow for his inspiring vision on motivic cohomology which led to the techniques of this paper, Alexander Petrov for discussions around Steenrod operations in algebraic geometry, the Tate conjectures and help with parts of Theorem~\ref{thm:counterexample}, and Burt Totaro for discussions around the integral Tate conjecture. We also thanks Ben Williams and Marco D'Addezio for useful comments on the previous version of the draft.

\section{Preliminaries} \label{sec:background}

\subsection{$\Ab^1$-invariant motivic spectra}\label{ssec:a1-inv}

An \emph{absolute $\Ab^1$-invariant motivic spectrum} is an $\Ab^1$-invariant motivic spectrum $E \in \SH_{\Zb}$; such an object is equivalent to the data of a cartesian section of the cartesian fibration over schemes classified by the functor $\SH: \Sch^{\op} \rightarrow \Cat_{\infty}$. Let $f: X \rightarrow \Spec(\Zb)$ be the structure morphism of a scheme $X$, then we set $E_X:=f^*E$. Sometimes, instead of $\SH_{\Zb}$, we work with $\SH_B$ for another Dedekind domain $B$; for example $\Zb[\tfrac{1}{p}]$ for a prime $p$. We will also call $E \in \SH_B$, somewhat abusively, an absolute $\Ab^1$-invariant motivic spectra. The context will make this clear.

If $X$ is a scheme, then the \emph{$E$-cohomology} of $X$ is given by 
\[
E(X):=\maps_{\SH_X}(1_X, E_X);
\] 
here $\maps$ refer to the spectrum of maps. We also consider its \emph{Thom twists}: for $v \in K_0(X)$ we have the \emph{twisted $E$-cohomology}
\[
E(X, v) := \maps_{\SH_X}(1_X, \mathrm{Th}(v) \otimes f^*E).
\]
We reserve the following ``Tate twist'' notation for the trivial twist (with an appropriate shift):
\[
E(n)(X):= \maps_{\SH_X}(1_X, \mathrm{Th}(\mathcal{O}^n) \otimes f^*E[-2n]) \simeq \maps_{\SH_X}(1_X, \Tb_X^{\otimes n} \otimes f^*E[-2n]).
\]
If $E$ is a homotopy commutative ring spectrum we will also consider the bigraded $E$-cohomology ring:
\begin{equation}\label{eq:e-cohomology}
E^{\star,\star}(X) := \bigoplus_{p,q} \pi_0\maps_{\SH_X}(1_X, \Tb_X^{q} \otimes f^*E[p-q]).
\end{equation}
If $E$ is furthermore oriented, then choosing an orientation in the sense of  \cite{deglise-orientations}, we obtain functorial identifications $E(X, v) \simeq E(X, \mathcal{O}^{\mathrm{rk}(v)})$, whence we are allowed to ignore the Thom twists.

The formation of $E$-cohomology defines a functor
\[
E(n): \Sch^{\op} \rightarrow \Sp \qquad \forall n \in \Zb.
\]
To proceed, we need to review some aspects of motivic homotopy theory around functoriality of $E(n)$; we refer to \cite[Section 2]{EHKSY2} for more details. 

\subsubsection{The purity transformation}\label{sec:purity} Recall that if $f: X \rightarrow Y$ is a morphism of schemes, then we say that $f$ is \emph{smoothable} if it factors as $X \xrightarrow{i} V \xrightarrow{\pi} Y$ where $i$ is a closed immersion and $\pi$ is a smooth morphism. The work of D\'eglise, Jin and Khan \cite{DJKFundamental} constructs the \emph{purity transformaton} for any smoothable, lci morphism $f: X \rightarrow Y$ 
\[
\mathfrak{p}_f: E_X \otimes \mathrm{Th}(T_f) \rightarrow f^!E_Y,
\]
where $T_f$ is the virtual tangent bundle of $f$. We say that $E$ is \emph{absolutely pure} if $\mathfrak{p}_f$ is invertible whenever both $X$ and $Y$ are regular; to verify that $E$ is absolute pure it is in fact sufficient to verify this for $f$ a closed immersion \cite[Remark 4.3.12(i)]{DJKFundamental}. We will also employ the following terminology: we say that $E$ is \emph{absolutely pure} at $f$ (or, at the pair $(X, Y)$ if the map is clear) if $\mathfrak{p}_f$ is invertible. If $E$ is absolutely pure for all smoothable morphisms between regular schemes, we say that it is absolutely pure. 

\begin{ex}[Algebraic and Hermitian $K$-theory] The absolute motivic spectrum $\mathrm{KGL}$ representing Weibel's homotopy $K$-theory, which in the context of purity is equivalent to Quillen's $K$-theory, is the first example of an absolutely pure motivic spectrum. As explained in \cite[Theorem 13.6.3]{cisinski:2019}, this follows from Quillen's Devissage theorem. 

Algebraic $K$-theory is quadratically enhanced by theory of hermitian $K$-theory; these come in various flavors \cite{GW-II}. There is a motivic spectrum, denoted by $\mathrm{KQ}$, which enjoys absolute purity by a theorem of Calmes-Harpaz-Nardin \cite[Theorem 8.4.2]{calmes-harpaz-nardin}. It represents the $\Ab^1$-invariant version of homotopy symmetric Grothendieck-Witt theory.
\end{ex}

\begin{ex}[$\ell$-adic cohomology] For the next example, we work with $\Zb[\tfrac{1}{\ell}]$-schemes. Let $H_{et}\Zb_{\ell}$ be the absolute motivic spectrum representing $\ell$-adic \'etale cohomology. Then, results of Thomason (up to denominators) \cite{thomason-purity} and Gabber (without restrictions) \cite{fujiwara-gabber} proves absolute purity for this motivic spectrum. 
\end{ex}


The next two examples are rational in nature. 

\begin{ex}[Beilinson motivic cohomology] Let $H\Qb$ be the absolute motivic spectrum representing Beilinson motivic cohomology \cite[Chapter 15]{cisinski:2019}, it is a model for the rational part of an $\mathbb{A}^1$-invariant theory of motivic cohomology and is constructed out of the Adams operations on $\mathrm{KGL}$ by Riou \cite{riou-rr}. Using the results from $K$-theory, Cisinski and D\'eglise has deduced absolute purity for $H\Qb$ in \cite[Theorem 14.4.1]{cisinski:2019}.
\end{ex}

\begin{ex}[Hermitian $K$-theory and the rational motivic sphere] One of the main results of \cite{dfjk} is that the rationalized motivic sphere spectrum is absolutely pure. This, in turn, relies on absolute purity results for the rationalized Grothendieck-Witt theory over $\Zb[\tfrac{1}{2}]$-schemes. 
\end{ex}

\subsubsection{Borel--Moore homology and compactly supported cohomology} If $X$ is a qcqs scheme, $U$ a quasicompact open with reduced complement $Z$:
\[
Z \xrightarrow{i} X \xleftarrow{j} U
\]
we have a fibre sequence of functors
\begin{equation}\label{eq:i-j}
i_*i^! \rightarrow \mathrm{id} \rightarrow j_*j^*.
\end{equation}
The localization sequence will be key to our results. Because the appearance of $i^!$ in the first term of the fibre sequence, we need to discuss Borel--Moore/compactly supported cohomology. 


The endofunctor $i_*i^!$ instantiates \emph{Borel--Moore homology}. It is important to note that this theory is a relative in nature. If $v \in K_0(X)$, $E \in \SH_S$ and $f: X \rightarrow S$ is a separated, finite type morphism then the \emph{Borel--Moore homology (twisted by $v$)} is given by
\[
E^{\mathrm{BM}}(X/S,v):= \maps_{\SH_S}(1_S, f_*(f^!E \otimes \mathrm{Th}(-v))).
\]
The formation of Borel-Moore homology enjoys some features that we will use:
\begin{enumerate}
\item First, it has covariant functoriality along proper maps in the following sense: if $f: X \rightarrow Y$ is a proper morphism of $S$-schemes, each of which is separated and finite type over $S$ then we have 
\[
f_*: E^{\mathrm{BM}}(X/S,f^*v) \rightarrow E^{\mathrm{BM}}(Y/S,v).
\]
\item Second, if we have a cartesian square which we will call $\square$:
\[
\begin{tikzcd}
Y \ar{r} \ar{d} &  T \ar{d} \\
 X \ar{r} & S,
\end{tikzcd}
\]
then we have the base-change map
\[
\square^*: E^{\mathrm{BM}}(X/S,v) \rightarrow E^{\mathrm{BM}}(Y/T,f^*v).
\]
\item Lastly, if $E$ is an absolute motivic spectrum then by applying~\eqref{eq:i-j}, we get a fibre sequence of spectra
\begin{equation}\label{eq:fiber}
E^{\mathrm{BM}}(Z/X, -\mathcal{O}^n)[-2n] \rightarrow  E(n)(X) \rightarrow E(n)(U).
\end{equation}
\end{enumerate}

Both $f_*$ and $\square^*$ can be expressed in terms of the six functor formalism; we refer to \cite{DJKFundamental} and \cite{EHKSY2} for more details and use them as needed in the present paper. 

\subsection{Motivic cohomology}\label{sec:hz}

The following simple definition of motivic cohomology appears in \cite{bem} and is based on Voevodsky's slice filtration \cite{Voevodsky2002, Voevodsky2002a}. We use the following notation: for each $j \in \Zb$ we write $\mathrm{Fil}^j_{\mathrm{slice}}: \SH_X \rightarrow \SH_X$ for the formation of the $j$-th slice cover and define the $j$-th slice to be
\[
s^j := \mathrm{cofib}(\mathrm{Fil}^{j+1}_{\mathrm{slice}} \rightarrow \mathrm{Fil}^j_{\mathrm{slice}}).
\]

\begin{defn}\label{def:hz} The absolute motivic spectrum representing \emph{$\mathbb{A}^1$-invariant motivic cohomology} is defined to be
\[
H\Zb^{\Ab^1} := s^0\mathrm{KGL} \in \SH_{\Zb}.
\]
In the notation of~\S\ref{ssec:a1-inv}, we obtain functors
\[
\maps_{\SH_{(-)}}(1_{(-)}, \Tb_X^{\otimes n} \otimes H\Zb^{\Ab^1}_{(-)})[-2n] = \Zb(n)^{\Ab^1}:\Sch^{\op} \rightarrow D(\Zb),
\]
after noting that the presheaves of spectra promote to one landing in $D(\Zb)$ by \cite[Theorem 3.43(3)]{bem}. If $X \in \Sch$ we write
\[
H^i_{\Ab^1}(X; \Zb(n)) := H^i(\Zb(n)^{\Ab^1}(X)) \qquad \forall i \in \Zb.
\]
\end{defn}

\begin{rem}[Coefficients]\label{def:coefficients} Let $R$ be a coefficient ring, we denote by $HR^{\Ab^1} := H\Zb^{\Ab^1} \otimes R$ where $R$ is regarded as an Eilenberg--Maclane spectrum and the tensor product is given by the action of spectra on $\SH_{\Zb}$ in presentable, stable $\infty$-categories. Concretely, we have equivalences, natural in $X$:
\[
R(n)^{\Ab^1}(X) \simeq \Zb(n)^{\Ab^1}(X) \otimes R,
\]
where the tensor product is taken in $D(\Zb)$. 

\end{rem}

\begin{rem}[Periodization]\label{def:periodization} As explained in \cite{bem}, we can assemble $H\Zb^{\Ab^1}$ to be a graded $\Eb_{\infty}$-ring in $\SH(\Zb)$. Briefly, the slice filtration $f^{\star}\mathrm{KGL}$ is multiplicative by formal reasons and thus its graded pieces (dictated by Bott periodicity for $\mathrm{KGL}$):
\[
\mathrm{gr}^{\star}\mathrm{KGL} \simeq \Tb^{\otimes \star} \otimes H\Zb^{\Ab^1},
\]
assemble into a graded $\Eb_{\infty}$-ring in motivic spectra. In particular, this gives $H^{\star}_{\Ab^1}(X; \Zb(\star))$ naturally the structure of a bigraded ring. 
\end{rem}

\begin{rem}[Comparison with $s^01$ and stability under pullbacks] One of the key results of \cite{bem} is that, for any qcqs scheme $X$, the unit map $1_X \rightarrow \mathrm{KGL}_X$ induces a diagram of equivalences
\[
\begin{tikzcd}
(s^01)_X \ar{r}{\simeq} \ar[swap]{d}{\simeq} & s^0(1_X) \ar{d}{\simeq}\\
(s^0\mathrm{KGL})_X \ar[swap]{r}{\simeq} & s^0(\mathrm{KGL}_X).
\end{tikzcd}
\]
This verifies that all possible slice-theoretic definitions of $\Ab^1$-invariant motivic cohomology agree. More precisely, the main theorem of \cite{bem} shows that $\Zb^{\Ab^1}(n)$ is calculated as the $\Ab^1$-localization of the cdh-local left Kan extension of the theory on smooth $\Zb$-algebras.
\end{rem}

\begin{rem}[Comparison with Spitzweck]\label{rem:spitzweck} In \cite{spitzweck:2018}, Spitzweck constructs an absolute motivic cohomology spectrum $H\Zb^{\mathrm{Spi}}$ which admits a canonical $\mathbb{E}_{\infty}$-ring structure. A key feature of his construction is that his motivic cohomology spectrum represents Bloch--Levine's higher Chow groups (see the remark below for more). The results of \cite{bem} establish an equivalence between $H\Zb^{\mathrm{Spi}}$ and $H\Zb^{\Ab^1}$, though the main results of that paper does not require Spitzweck's theory as input or even Bloch--Levine's cycle complex over Dedekind domains (though we do need it over fields). We use freely this idenitification in this paper. 
\end{rem}


\begin{rem}[The Bloch--Levine cycle complex] There is a more explicit model for motivic cohomology, which works nicely for smooth schemes over Dedekind domains and over fields. We will use this complex to verify absolute purity for motivic cohomology in what follows.


For simplicity, we will assume that $X$ is a finite type scheme over a Dedekind domain or a field, which is equidimensional. Otherwise one needs to replace codimension with dimension in definition of Bloch's cycle complex.
The \textit{algebraic $n$-simplex} is defined as 
\[
\Delta^n := \Spec\left(\Zb[t_0,\dots,t_n] / \left(1 - \sum_i t_i \right) \right).
\]
For $i_0 < i_1 < \cdots < i_r$, the equations $t_{i_0} = t_{i_1} = \cdots = t_{i_r} = 0$ define a codimension $r$ face of $\Delta^n$. We denote by $z^n(X,i)$ the free abelian group that is generated by those codimension $n$ irreducible integral subschemes of $\Delta^i \times X$ whose intersection with all the faces of $\Delta^i \times X$ has the expected codimension (i.e., the intersections are \textit{proper}). For example, $z^n(X,0)$ is the group of algebraic cycles in the usual sense. By construction, there are well-defined and functorial pullback maps
\[
z^n(X,j) \to z^n(X,i)
\]
along inclusions of faces $\Delta^i \times X \hook \Delta^j \times X$. The \emph{cycle complex} $z^n(X,*)$ to be the complex of abelian groups whose degree $i$ piece is $z^n(X,i)$, and whose differentials $\partial \colon z^n(X,i) \to z^n(X,i-1)$ are given by the alternating sum of pullbacks along the codimension one faces in the standard fashion. Bloch's higher Chow groups are defined as the Zariski hypercohomology groups
\begin{align*}
\CH^i(X ; j) &= \Hb_{\Zar}^i(X; z^j(-,*)) (H^i(\simeq L_{\Zar}z^j(-,*)(X)))
\end{align*}
of the presheaf that associates to an open subset $U \subset X$ the motivic complex $z^j(U,*) \in D(\Zb)$  (in other words, the Zariski sheafification of the presheaf of complexes on $X$). These hypercohomology groups coincide with the cohomology groups of the complex up to a shift (see e.g. \cite{geisser:2005}), whenever $X$ is smooth over a Dedekind domain or a field:
\[
L_{\Zar}z^j(-,*)(X)[-2j] \simeq \Zb(j)^{\mot}(X).
\]
\end{rem}

There are functorial pullback maps
\[
f^* \colon z^i(X,\star) \to z^i(X',\star)
\]
along flat maps $f \colon X' \to X$, which are defined on the cycle level similarly to Chow groups \cite{fulton:1998}. Similarly, there are functorial pushforward maps
\[
p_* \colon z^i(Z,\star) \to z^{i-r}(X,\star)
\]
along proper maps $p \colon Z \to X$ of relative dimension $r$. Furthermore, if
\[
\begin{tikzcd}
Z' \arrow[r]{}{p'} \arrow[d]{}{f'} & X' \arrow[d]{}{f} \\
Z \arrow[r]{}{p} & X
\end{tikzcd}
\]
is a Cartesian square with $f$ flat and $p$ proper of relative dimension $r$, then
\begin{equation}\label{eq:PushPull}
p'_* f'^* = f^* p_* \colon z^i(Z,\star) \to z^{i-r}(X',\star)
\end{equation}
as maps of complexes (this is a consequence of \cite[Proposition~1.7]{fulton:1998}). Levine has proven the following localization theorem for higher Chow groups  \cite[Theorem~1.7]{levine:2001}. 

\begin{thm}\label{thm:DetroitPurity}
Let $X$ be a finite-type scheme over a discrete valuation ring $A$, let $i \colon Z \hook X$ be an equicodimensional closed embedding of codimension $r$, and let $j \colon U \hook X$ be the open complement of $Z$. Then, the null-sequence of complexes
\begin{equation}\label{eq:loc}
z^{n-r}(Z,\star) \xto{i_*} z^{n}(X,\star) \xto{j^*} z^{n}(U,\star)
\end{equation}
represents a cofibre sequence in $D(\Zb)$.
\end{thm}

In particular, we obtain a localization long exact sequence relating the motivic cohomology of a smooth $A$-scheme, the motivic cohomology of its special fibre, and the motivic cohomology of its generic fibre. This is a special case of the conjectural \textit{absolute purity sequence} for motivic cohomology, which would relate the motivic cohomologies of a regular scheme $X$, a regular closed subscheme $Z$, and its open complement $U$.  

Towards this end, we have following which is considered a folklore conjecture in the community.

\begin{con}[Motivic absolute purity] The motivic spectrum $H\Zb^{\Ab^1}$ is absolutely pure. 
\end{con}

We now explain how the work of Levine in \cite{levine:2001}, namely Theorem~\ref{thm:DetroitPurity}, essentially establishes this result in cases that we need for this paper. The only point to address is to compare the pushforward map involved in~\eqref{eq:loc} versus the purity map. The key calculation is performed in Appendix~\ref{sec:djk-higher}, but also follows from \cite[Corollary~8.4]{SpitzweckFrankland}. We also remark that \cite[Theorem 4.43]{bem} provides a different perspective on the localization sequence of Levine via a prismatic approach.

\begin{thm}[Levine]\label{thm:abs-pur} Let $X$ be smooth scheme over a discrete valuation ring $A$ of mixed characteristics.  Let $i: X/\pi \hookrightarrow X$ be the inclusion of the characteristic $p > 0$ fibre into $X$. Then, the purity transformation
\[
\mathfrak{p}_i: H\Zb_{X/\pi}^{\Ab^1} \otimes \Tb^{-1} \rightarrow i^!H\Zb^{\Ab^1}_X,
\]
is an equivalence. 
\end{thm}

\begin{proof} 


To proceed, we make some reductions. Let $B = \Spec(A)$ and $Z = \Spec(A/\pi)$ and let $\iota: Z \hookrightarrow B$ be the closed immersion and $j: X \setminus Z \hookrightarrow X$ be the complementary open immersion. By tor-independent base change for the purity transformation \cite[Proposition 3.2.8]{DJKFundamental}, it suffices to prove the result for $X = B$ itself. Therefore we are reduced to checking that $\mathfrak{p}_{\iota}$ is an equivalence. Now, Levine's Theorem~\ref{thm:DetroitPurity} implies\footnote{Strictly speaking, Theorem~\ref{thm:DetroitPurity} is a statement about graded sheaves of spectra, not about motivic spectra. However, as the sequence of Eq.~\eqref{eq:loc} is linear over the motivic cohomology of $X$, it can be enhanced into a cofibre sequence in $\SH_X$ using e.g. the orientation of the motivic cohomology of $X$. See \cite[Construction~6.5]{AHI:atiyah} for an argument outlining how graded sheaves with orientations can be enhanced into motivic spectra.} the existence of an equivalence $\iota^!H\Zb_B^{\Ab^1} \simeq H\Zb^{\Ab^1}_{Z} \otimes \Tb^{\otimes -1}$; for example combine \cite[Theorem 7.4 \& 7.18]{spitzweck:2018} and Remark~\ref{rem:spitzweck}. Therefore, the purity transformation is then a self-map: $\mathfrak{p}_{\iota}: H\Zb^{\Ab^1}_{Z} \otimes \Tb^{\otimes -1} \rightarrow H\Zb^{\Ab^1}_{Z}  \otimes \Tb^{\otimes -1}$. But by the calculation in Proposition~\ref{prop:DJKFundamentalClassForHigherCH} in the appendix this map classifies the fundamental class of $Z$, and therefore by linearity $\pfr_\iota$  is given by the identity on the level of motivic spectra.
\end{proof}

%
%
%
%

\subsection{Motivic Cartan formula}\label{ssect:Cartan}

Classically, the Cartan formula describes how the power operations interact with the cup product in mod-$p$ cohomology. Here, we describe the Cartan formula in motivic cohomology as an equivalence between maps from $H \Fb^{\Ab^1}_\ell \otimes H\Fb^{\Ab^1}_\ell$ to $H\Fb^{\Ab^1}_\ell$, and interpret it in terms of the spectrum-level coproduct of the motivic Steenrod algebra. The advantage of this spectrum-level description is that it makes it apparent that the motivic Cartan formula is stable under pullbacks. Throughout this subsection, we work with field $k$ and $\ell$ a prime invertible in $k$.

We denote by $\Ac$ the internal mapping object $\underline{\maps}(H\Fb^{\Ab^1}_\ell, H\Fb^{\Ab^1}_\ell)$ in $\SH_k$ whose bigraded homotopy groups form the motivic Steenrod algebra \cite[Theorem~1.1(2)]{HKO}. We consider it as a left $H\Fb^{\Ab^1}_\ell$-module with the module structure given by the target $H\Fb^{\Ab^1}_\ell$. Composing tensor product with the multiplication $\mu \colon H\Fb^{\Ab^1}_\ell \otimes H\Fb^{\Ab^1}_\ell \to H\Fb^{\Ab^1}_\ell$ defines the \textit{external product}
\begin{equation}\label{eq:externalProdEquiv}
\times \colon \Ac \otimes_{H\Fb^{\Ab^1}_\ell} \Ac \xto{\sim} \underline{\maps}(H\Fb^{\Ab^1}_\ell \otimes H\Fb^{\Ab^1}_\ell, H\Fb^{\Ab^1}_\ell)
\end{equation}
which is an equivalence by the results of \cite[\S 5.1]{HKO}. 

\begin{rem}\label{rem:sum} In fact, more is true. Let $\Psi$ be the map~\eqref{eq:psi} from the introduction, classifying admissible monomials in the dual mod-$\ell$ motivic Steenrod algebra
\[
\Psi: \bigoplus_{\alpha} \Tb^{\otimes q_{\alpha}} \otimes H\Fb^{\Ab^1}_{\ell}[p_{\alpha}-q_{\alpha}] \xrightarrow{\simeq} H\Fb^{\Ab^1}_{\ell} \otimes H\Fb^{\Ab^1}_{\ell};
\]
as stated there, this map is an equivalence when $\ell$ is invertible in $k$. Thus, the motivic spectrum $\Ac \otimes_{H\Fb^{\Ab^1}_\ell} \Ac$, and therefore the motivic spectrum $\underline{\maps}(H\Fb^{\Ab^1}_\ell \otimes H\Fb^{\Ab^1}_\ell, H\Fb^{\Ab^1}_\ell)$, can be written as a sum of shifts and twists of $H\Fb^{\Ab^1}_\ell$.
\end{rem}

\begin{rem}[Characterization of the external product]\label{rem:external}
If $x,y \in \Ac^{\star, \star}$ are $H\Fb^{\Ab^1}_\ell$-cohomology operations in $\SH_k$ (in other words, maps of the form $H\Fb_{\ell}^{\Ab^1} \rightarrow \Tb^{\otimes j} \otimes H\Fb_{\ell}^{\Ab^1}[i]$), then $x \times y$ induces the cohomology operation that is uniquely characterized by the property that the composition
\[
\left(H\Fb^{\Ab^1}_\ell \right)^{\star,\star}(X) \otimes \left(H\Fb^{\Ab^1}_\ell \right)^{\star,\star}(Y) \to
\left( H\Fb^{\Ab^1}_\ell \otimes H\Fb^{\Ab^1}_\ell \right)^{\star, \star}(X \times Y) \xto{x \times y} \left(H\Fb^{\Ab^1}_\ell \right)^{\star,\star}(X \times Y)
\]
sends $\alpha \otimes \beta$ to $x(\alpha) \times y(\beta)$ where $X, Y $ are smooth $k$-schemes. Above, the first map is the structure map given by the Day convolution of presheaves, on which the symmetric monoidal structure of $\SH_k$ is based upon. 
\end{rem}

Using this identification, we next enhance the coproduct $\psi^* \colon \Ac^{\star,\star} \to \Ac^{\star,\star} \otimes_{(H\Fb^{\Ab^1}_\ell)^{\star,\star}} \Ac^{\star,\star}$ that was defined in \cite[Lemma~11.6]{voevodsky:2003a} to a spectrum-level coproduct map. Subsequently, we will interpret the Cartan formula in terms of this map.

\begin{lem}\label{lem:coprod-psi}
The coproduct map $\psi^*$ is the map induced on the bigraded homotopy groups by precomposition by the multiplication map $\mu$:
\[
\mu^*\colon \underline{\maps}(H\Fb^{\Ab^1}_\ell, H\Fb^{\Ab^1}_\ell) \to \underline{\maps}(H\Fb^{\Ab^1}_\ell \otimes H\Fb^{\Ab^1}_\ell, H\Fb^{\Ab^1}_\ell) \xto{\times^{-1}, \simeq}  \Ac \otimes_{H\Fb^{\Ab^1}_\ell} \Ac.
\]
\end{lem}
\begin{proof}
Indeed, if $\mu^*(x) = \sum_i x'_i \times x''_i$, then $\sum_i x'_i \times x''_i \in \Ac^{\star\star} \otimes_{H\Fb_\ell^{\star\star}} \Ac^{\star\star}$ is the unique element satisfying the condition of \cite[Lemma~11.6]{voevodsky:2003a}. This then defines the spectrum-level map after the characterization of the external product in Remark~\ref{rem:external}.
\end{proof}

From now on, we will denote the coproduct on $\Ac$ by $\mu^*$ to highlight its relationship with the product structure on $H\Fb^{\Ab^1}_\ell$. Combining Lemma~\ref{lem:coprod-psi} with \cite[Lemma~11.6]{voevodsky:2003a}, motivic Cartan formula \cite[Proposition~9.7]{voevodsky:2003a} may be understood as the computation of the action of $\mu^*$ on the motivic power operations (natural on a motivic space $\Xfr$ over $k$) 
\[
\P^i \colon H^{\star}_{\Ab^1}(\Xfr; \Fb_{\ell}(\star)) \to H^{\star + 2i(p-1)}_{\Ab^1}(\Xfr; \Fb_{\ell}(\star + i(p-1)))
\]
and
\[
\B^i \colon H^{\star}_{\Ab^1}(\Xfr; \Fb_{\ell}(\star)) \to H^{\star + 2i(p-1) + 1}_{\Ab^1}(\Xfr; \Fb_{\ell}(\star + i(p-1));
\]
here  $\B^i = \beta \P^i$ and we refer the reader to \cite[Page 33]{voevodsky:2003a} or \cite[\S 2.4]{HKO} for details on the construction of these operations. We state the motivic Cartan formula here for the convenience of the reader.

\begin{thm}[Motivic Cartan formula]\label{thm:MotivicCartanSH}
Let $k$ be a field of characteristic $p \not = \ell$. Then, the following homotopies exist between maps to $\underline{\maps}(H\Fb^{\Ab^1}_\ell \otimes H\Fb^{\Ab^1}_\ell, H\Fb^{\Ab^1}_\ell) \in \SH_k$:
\begin{enumerate}
\item if $\ell \not = 2$, then
\begin{align*}
\mu^*(\P^i) \simeq& \sum_{r = 0}^i \P^r \times \P^{i-r} \\
\mu^*(\B^i) \simeq& \sum_{r = 0}^i \B^r \times \P^{i-r} + \P^r \times \B^{i-r} 
\end{align*}

\item if $\ell = 2$, then
\begin{align*}
\mu^*(\Sq^{2i}) \simeq& \sum_{r = 0}^i \Sq^{2r} \times \Sq^{2i-2r} + \tau \sum_{s = 0}^{i-1} \Sq^{2s + 1} \times \Sq^{2i-2s - 1} \\
\mu^*(\Sq^{2i  + 1}) \simeq& \sum_{r = 0}^i (\Sq^{2r + 1} \times \Sq^{2i-2r} + \Sq^{2r} \times \Sq^{2i-2r+1}) \\
& + \rho \sum_{s=0}^{i-1} \Sq^{2s+1} \times \Sq^{2i-2s-1},
\end{align*}
where $\tau$ is the non-zero element of $H_{\Ab^1}^0(k; \Fb_2(1)) = \mu_2(k)$, and $\rho$ is the image of $-1$ in $H_{\Ab^1}^{1}(k; \Fb_2(1)) = k^\times / (k^\times)^2$.
\end{enumerate}
\end{thm}

\begin{rem}[Spectrum-level Cartan formula]\label{rem:spectrumCartan} We further clarify the formulas of Theorem~\ref{thm:MotivicCartanSH}. Let $x: H\Fb_{\ell}^{\Ab^1} \rightarrow \Tb^{\otimes j} \otimes H\Fb_{\ell}^{\Ab^1}[i]$ be one of the power operations $\P^i, B^i$ or $\Sq^i$. Then the theorem asserts that we have the following commutative diagram in $\SH_k$:
\begin{equation}\label{eq:spectrumCartan}
\begin{tikzcd}
 H\Fb_{\ell}^{\Ab^1} \ar{r}{x} & \Tb^{\otimes n} \otimes H\Fb_{\ell}^{\Ab^1}[m]\\
  H\Fb_{\ell}^{\Ab^1} \otimes  H\Fb_{\ell}^{\Ab^1} \ar{u}{\mu} \ar{r}{\Psi^{-1},\simeq} & \bigoplus_{\alpha} \Tb^{\otimes q_{\alpha}} \otimes H\Fb^{\Ab^1}_{\ell}[p_{\alpha}-q_{\alpha}] \ar[swap]{u}{\sum x_n \otimes x_m},\\
\end{tikzcd}
\end{equation}
where the sum of the right vertical map classifies the expression on right hand side of the formulas in Theorem~\ref{thm:MotivicCartanSH}. In other words, these are lifts of the motivic Cartan formulas to a spectrum level-statement. That we can do this ultimately rests on the very deep fact that $\Psi$ is an equivalence. 
\end{rem}

\section{Mod-$p$ motivic power operations of $\Fb_p$-schemes}\label{sect:Pops}

We now construct the motivic power operations acting on the mod-$p$ motivic cohomology of $\Fb_p$-schemes. We begin by constructing the operations on smooth $\Fb_p$-varieties in Section~\ref{subsect:smoothPops}.


\subsection{A motivic Nizio{\l} theorem}\label{subsect:smoothPops}

To construct the motivic mod-$p$ power operations on smooth $\Fb_p$-varieties, we study motivic cohomology over $\Oc = \Zb_p[p^{1/p^\infty}]$, the infinitely ramified extension of $\Zb_p$ obtained by adjoining all the $p$th power roots of $p$. The idea is to use Theorem~\ref{thm:DetroitPurity} to show that the motivic cohomology of smooth $\Oc$-schemes is completely determined by the generic fibre. Our argument is similar to the one that used in \cite[Section~3]{AMM:2022} in order to obtain a similar result for algebraic $K$-theory. In turn, \cite{AMM:2022} was inspired by a phenomenon first observed by Nizio{\l}  in her proof of the crystalline comparison theorems \cite{niziol-crys}.

We denote by 
\[
j \colon \Spec(\Kb) \hook \Spec(\Oc)
\]
the open embedding obtained by inverting $p$. In particular $\Kb = \Oc[p^{-1}]$. The precise result we are going to prove is the following.

\begin{thm}\label{thm:Detroit}
Let $E$ be an absolute oriented motivic spectrum with absolute purity for pairs $(X, X_p)$, where $X$ is smooth over a discrete valuation ring $A$ of mixed characteristic $(0,p)$, and $X_p \hook X$ is the special fibre. Then, pulling back along the natural transformation $j_! j^* \to \Id$ induces an equivalence for all $m \geq 1$:
\[
\maps_{\SH_{\Oc}}(\Xfr, E/p^m) \rightarrow \maps_{\SH_{\Oc}}(j_! j^*(\Xfr), E/p^m).
\]
for all $\Xfr \in \SH_{\Oc}$. In particular, the unit map $E \to j_* E$ induces equivalences
\[
E/p^m \xrightarrow{\simeq} j_* E/p^m \qquad \forall m \geq 1,
\]
and
\[
(E)^\comp_p \xrightarrow{\simeq} ( j_* E)^\comp_p.
\]
\end{thm}

We will postpone this proof until after the key Lemma~\ref{lem:root-canal}. The following Corollary is the key point of our construction of motivic power operations. 

\begin{cor}\label{cor:main}
The pullback along the natural transformation $j_! j^* \to \Id$ induces an equivalence 
\[
\maps_{\SH_{\Oc}}(\Xfr, H\Fb_p^{\Ab^1}) \rightarrow \maps_{\SH_{\Oc}}(j_! j^*(\Xfr), H\Fb_p^{\Ab^1}).
\]
In particular, the unit map $H\Fb^{\Ab^1}_p \to j_* H\Fb^{\Ab^1}_p$ is an equivalence.
\end{cor}

\begin{proof} Motivic cohomology has the desired absolute purity by Theorem~\ref{thm:abs-pur}. The ``in particular'' part is a consequence of the first claim because it implies that mapping to the unit map induces an equivalence for all $\Xfr \in \SH_{\Oc}$.
\end{proof}

We now prove key Lemma~\ref{lem:root-canal}. To set it up, we need some notation:
\begin{enumerate}
\item $A$ is a discrete valuation ring of mixed characteristic $(0, p)$ with uniformizer $\pi$;
\item we let $A':=A[\sqrt[p^d]{\pi}] \cong A[t]/(t^{p^d} - \pi)$ for some integer $d \geq 1$;
\item we let $X$ be a smooth $A$-scheme and set $X' := X_{A'}$ to be the base change and $Y$ to be its special fibre;

\item we will also contemplate the cartesian square $\square$:
\begin{equation}\label{eq:extend}
\begin{tikzcd}
Y' \ar{d}{q} \ar{r}{k} & X' \ar{d}{p}\\
Y \ar{r}{i} & X.
\end{tikzcd}
\end{equation}
Note that the reduction $Y'_\red$ is isomorphic to $Y$. 

\end{enumerate}
%

\begin{lem}[Extracting roots]\label{lem:root-canal} In the notation of the previous paragraph, assume that $E$ is an oriented ring spectrum that is absolutely pure for the pair $(X, Y)$ and $(X', Y')$. Then the image of the base change map
\[
\square^*:E^{\mathrm{BM}}(Y/X, \mathcal{O}^n) \rightarrow E^{\mathrm{BM}}(Y'/X',\mathcal{O}^n)
\]
is $d$-divisible for all $n \in \Zb$.
\end{lem}
\begin{proof} 
Denote the open complements of $i:Y \rightarrow X$ and $k:Y' \rightarrow X'$ respectively by
\[
j:U \rightarrow X \qquad \ell: U' \rightarrow X',
\]
so that we also have the induced, finite morphism $r: U' \rightarrow U$.
Consider the commutative diagram
\begin{equation}\label{eq:extend}
\begin{tikzcd}
E(X)(n) \arrow[d]{}{p^*} \arrow[r]{}{j^*} & E(U,n) \arrow[d]{}{r^*} \\
E(X')(n) \arrow[d]{}{p_!} \arrow[r]{}{j'^*} & E(U')(n) \arrow[d]{}{r_{!}} \\
E(X)(n) \arrow[r]{}{j^*} & E(U)(n).
\end{tikzcd}
\end{equation}
The above diagram is considered as a diagram of $E^X$-modules in $\SH_X$, where $r$ is the restriction of $p$ over $U$. Projection formula implies that the map $p_! p^*$ is given by multiplication by $p_!(1_{X'}) \in E^X$.

Taking the horizontal fibres in \eqref{eq:extend}, and considering the induced mapping spectra whose sources are powers of $\Tb$, we obtain sequences of maps
\[
E^{\mathrm{BM}}(Y/X, -\mathcal{O}^n)[-2n] \xto{\square^*} E^{\mathrm{BM}}(Y'/X',-\mathcal{O}^n)[-2n] \xto{\square_!} E^{\mathrm{BM}}(Y/X,-\mathcal{O}^n)[-2n].
\] 
To prove the claim, it suffices to show  $p_!(1_{X'}) \simeq d$, which would imply that the above composition is multiplication by $d$, and that $\square_!$ is an isomorphism. We will do this in the next two paragraphs. 

We now show that $p_!(1_{X'}) \simeq d$. By assumption, $p$ factors as $X' \stackrel \iota \hook \Pb^1_X \xto{\pi} X$, where $\iota$ is a divisor for the line bundle $\Oc(d)$, and $\pi$ is the projection. As the second power of $c_1(\Oc(1))$ vanishes in $E$-cohomology, we can use the formal group law of $E$ to compute that $\iota_!(1_{X'}) = d c_1(\Oc(1))$. As $\pi_!(c_1(\Oc(1))) = 1_X$, it follows that $p_!(1_{X'}) = d$ in $E$-cohomology, as desired.

Next, we show that $\square_!$ is an isomorphism. Consider the following diagram
\[
\begin{tikzcd}
p_*p^* \ar{d}{\mathfrak{p}_p}& p_*k_*k^!p^* \ar{l}{\eta}  \ar{d}{\mathfrak{p}_p} & \Sigma^{-2,-1}p_*k_*k^*p^* \ar{l}{\mathfrak{p}_k, \simeq} \ar{dl}{\mathfrak{p}_{p \circ k}} \ar{dd}{\simeq}\\
p_*p^! \ar{dd}{\eta}& p_*k_*k^!p^! \ar{l}{\eta} \ar{d}{\simeq} & \\
 & i_*q_*q^!i^! \ar{d}{\eta,\simeq} &\ar{l}{\mathfrak{p}_{i \circ q}} i_*q_*q^*i^*  \\
\mathrm{id} & i_*i^!  \ar{ur}{\simeq} \ar{l}{\eta} & 
\end{tikzcd}
\]
which commutes by the naturality of the maps involved. The leftmost vertical composite calculates $p_!$ from Eq.~\eqref{eq:extend}, and the middle vertical composite calculates $\square_!$. The fact that $\square_!$ is an equivalence follows from the commutativity of the above diagram, and the labeled equivalences.
\end{proof}

\begin{proof}[Proof of Theorem~\ref{thm:Detroit}]
As both $j^*$ and $j_!$ commute with colimits, it suffices to prove the claim for $\Xfr$ that is representable by Tate twists of an affine smooth $\Oc$-scheme $X$. Hence it suffices to prove that the map:
\[
E(n)(X)/p^m \rightarrow E(n)(X[\tfrac{1}{p}])/p^m \qquad n \in \Zb, r \geq 1
\]
is an equivalence. 

Define a sequence or rings $A_r := \Zb_p[\sqrt[p^r]{p}]$. It converges to $\Oc$ in the sense that $A_\infty = \colim_r A_r = \Oc$. Since smooth $\Oc$ schemes are finitely presented over $\Oc$ by definition, $X = X_\infty$ is base changed from some $X_r$ defined over $A_r$; we set $U_r := X_r[\tfrac{1}{p}]$ and $Y_r$ the base change of $X_r$ to the closed point of $\Spec(A_r)$. Thanks to~\eqref{eq:fiber} we have a fibre sequence, compatible in $r$:
\[
E(Y_r/X_r,-\mathcal{O}^n)/p^m[-2n] \rightarrow E(n)(X_r) / p^m \rightarrow E(n)(U_r) / p^m 
\]
Applying Lemma~\ref{lem:root-canal} we conclude that
\[
\colim_{i\geq r} E(n)(X_i) / p^m \to \colim_{i \geq r} E(n)(U_i)/p^m
\]
is an equivalence as the fibre is contractible. As any absolute motivic spectrum is finitary \cite[Proposition~C.12(4)]{hoyois:2014}, this implies that the pullback map $E(n)(X_\infty)/p^r \to E(n)(U_\infty)/p^r$ is a quasi-isomorphism, as desired.
\end{proof}

\begin{var}\label{var:cyc} Consider the ring $\Zb_p[\zeta_{p^{\infty}}]$ where we have adjointed to $\Zb_p$ all its $p$-power roots of unity; its $p$-completion is the perfectoid ring that appears in many places in $p$-adic Hodge theory and is usually denoted by $\Zb_p^{\mathrm{cyc}}$. Consider
\[
j: \Spec((\Zb_p[\zeta_{p^{\infty}}])[\tfrac{1}{p}] =: \Qb_p[\zeta_{p^{\infty}}]) \hookrightarrow \Spec(\Zb_p^{\mathrm{cyc}}).
\]
We sketch how, the the same argument as in Theorem~\ref{thm:Detroit}, verifies that for any $E$ as in hypothesis of that result satisfies \[
(E)^\comp_p \xrightarrow{\simeq} (j_*E)^\comp_p. 
\]
In particular, we get the equivalence:
\[
H\Fb^{\Ab^1}_p \xrightarrow{\simeq} j_*H\Fb^{\Ab^1}_p.
\]
The key point is to modify Lemma~\ref{lem:root-canal} slightly. Noting that $\Zb_p[\zeta_{p^{\infty}}] \simeq \colim \Zb_p[\zeta_{p^{m}}]$ we set $R_m := \Zb_p[\zeta_{p^{m}}]$. Then for any smooth $R_m$-scheme $X_m$ and for any $m' > m$ we set $X_{mm'}$ to be the base change from $R_m$ to $R_{m'}$. We consider the square $\square_{mm'}$:
\[
\begin{tikzcd}
Y'_{mm'} \ar{d}{q} \ar{r}{k} & X_{mm'} \ar{d}{p}\\
Y_m \ar{r}{i} & X_m.
\end{tikzcd}
\]
In this case, we also note that $(Y'_{mm'})_{\mathrm{red}} \cong Y_m$ is a nilpotent thickening. Then the same arguments as in Lemma~\ref{lem:root-canal} shows that
\[
\square^*:E^{\mathrm{BM}}(Y_{m}/X_m, -\mathcal{O}^n)[-2n] \rightarrow E^{\mathrm{BM}}(Y_{mm'}/X_{mm'},-\mathcal{O}^n)[-2n]
\]
is divisible by some power of $p$, which is the crucial step for the proof of Theorem~\ref{thm:Detroit}.

\end{var}

\begin{rem}[Arbitrary perfectoid rings and the work of Bouis--Kundu] Variant~\ref{var:cyc} and Corollary~\ref{cor:main} are motivic refinements of \cite[Corollary 3.3]{AMM:2022} and our methods are inspired by theirs. Unfortunately since absolute purity is not known more generally for motivic cohomology, we are unable to use the same method to prove a more general statement. Nevertheless, the work of Bouis--Kundu \cite{bouis-kundu}, however, has established an isomorphism:
\[
H^n_{\Ab^1}(X;\Zb/p^m(n)) \xrightarrow{\cong} H^n_{\Ab^1}(X[\tfrac{1}{p}];\Zb/p^m(n)) \qquad \forall n, m \in \Zb
\]
whenever $X$ is smooth over a perfectoid valuation ring $V$ of mixed characteristic $(0, p)$ via a completely different method. 
\end{rem}

%
%

%
%
%
%
%

\subsection{Construction of power operations} \label{sec:construction}
We now reap the benefits of our work. First, consider the maps of $\Eb_1$-algebra ($\Eb_\infty$-coalgebras) spectra of endomorphisms (induced by the multiplication as in Lemma~\ref{lem:coprod-psi}); the (co)multiplicativity is evident because all the functors are strongly monoidal (and because the coalgebra structure is given by pullback along multiplication on $H\Fb^{\Ab^1}_p$):
\[
\mathrm{map}_{\SH_{\Kb}}(H\Fb_p^{\Ab^1}, H\Fb_p^{\Ab^1}) \xleftarrow{j^*,\simeq}  \mathrm{map}_{\SH_{\Oc}}(H\Fb_p^{\Ab^1}, H\Fb_p^{\Ab^1}) \xrightarrow{i^*} \mathrm{map}_{\SH_{\Fb_p}}(H\Fb_p^{\Ab^1},H\Fb_p^{\Ab^1}),
\]
where, the equivalence is Corollary~\ref{cor:main}. In fact, we can do better: we have morphisms and equivalences of graded $\Eb_1$-rings:
\[
\mathrm{map}_{\SH_{\Kb}}(H\Fb_p^{\Ab^1}, \Tb^{\otimes \star} \otimes H\Fb_p^{\Ab^1})\xleftarrow{j^*,\simeq} \mathrm{map}_{\SH_{\Oc}}(H\Fb_p^{\Ab^1}, \Tb^{\otimes \star} \otimes H\Fb_p^{\Ab^1}) \xrightarrow{i^*} \mathrm{map}_{\SH_{\Fb_p}}(H\Fb_p^{\Ab^1}, \Tb^{\otimes \star}\otimes H\Fb_p^{\Ab^1}).
\]
Passing to homotopy groups, Corollary~\ref{cor:main} implies that we have constructed a map of bigraded rings
\begin{align*}
\Ext_{\SH_{\Oc}}^{\star,\star} (H\Fb_p^{\Ab^1}, H\Fb_p^{\Ab^1}) &\stackrel{\sim}{\to} \Ext_{\SH_{\Oc}}^{\star,\star} (j_! j^* H\Fb^{\Ab^1}_p, H\Fb^{\Ab^1}_p) \\
&= \Ext_{\SH_{\Kb}}^{\star,\star} ( H\Fb^{\Ab^1}_p, H\Fb^{\Ab^1}_p).
\end{align*}
Above we have used the fact that motivic cohomology is stable under pullbacks \cite{spitzweck:2018}. As $\Kb$ is a field of characteristic 0, the structure of the final algebra in the above chain of equalities is completely understood by work of Voevodsky \cite{voevodsky:2010} and Hoyois--Kelly--Ostvaer \cite[Theorem~1.1(2)]{HKO}.
 
Thus, we have proven the following result.
\begin{thm}\label{thm:MotivicSteenrodAlgebraOverZpcyc}
The bigraded endomorphism algebra $\Ext_{\SH_{\Oc}}^{\star,\star} (H \Fb^{\Ab^1}_p, H\Fb_p^{\Ab^1})$ has the expected form. In particular:
\begin{enumerate}
\item The \emph{admissible monomials}
\[
\{\beta^{\epsilon_r} \P^{i_r} \cdots \beta^{\epsilon_1} \P^{i_1} \beta^{\epsilon_0} \vert r \geq 0, i_j > 0, \epsilon_i \in \{0,1\}, i_{j+1} > p i_j + \epsilon_j \}
\]
form a basis for $\Ext_{\SH_{\Oc}}^{\star,\star}$ as a left $H_{\Ab^1}^{\star}(\Oc; \Fb_p(\star))$-module;
\item the \emph{motivic power operations} $\P^i$ satisfy the motivic Adem relations;
\item the motivic Cartan formula holds for motivic power operations acting on products. 
\end{enumerate}
\end{thm}

\begin{proof} The motivic Cartan formula deserves some comment. As explained in Remark~\ref{rem:spectrumCartan} the Cartan formula is witnessed by the commutativity of the diagram~\eqref{eq:spectrumCartan} which holds over $\Kb$ since it is a field of characteristic zero (and uses the colgebra structure on the spectrum of endomorphsims). All the terms in the diagram are stable under pullbacks and therefore the corresponding diagram also commutes over $\Oc$, proving the Cartan formula. 
\end{proof}

Next, we construct the motivic mod-$p$ power operations of smooth $\Fb_p$-varieties. Let 
\[
i \colon \Spec(\Fb_p) \hook \Spec(\Oc)
\]
be the closed embedding obtained by killing all the roots of $p$ in $\Oc$. As the pullback functor $i^* \colon \SH_{\Oc} \to \SH_{\Fb_p}$ takes $H\Fb^{\Ab^1}_p$ to $H\Fb^{\Ab^1}_p$ (see \cite[Chapter 8]{spitzweck:2018} or \cite[Theorem 5.2]{bem} for a different approach), we obtain the following result on the level of motivic spectra.

\begin{thm}\label{thm:ModPSteenrodSpectrumLevel}
There exist reduced mod-$p$ power operation endomorphisms
\[
\P^i \colon H\Fb^{\Ab^1}_p \to  \Tb^{\otimes i(p-1)} \otimes H\Fb^{\Ab^1}_p 
\]
and
\[
\B^i \colon H\Fb^{\Ab^1}_p \to  \Tb^{\otimes i(p-1)} \otimes H\Fb^{\Ab^1}_p[1]
\]
in $\SH_{\Fb_p}$ that are pullbacks of similarly named endomorphisms in $\SH_{\Oc}$, and where
\begin{align*}
\B^i &= \beta \P^i \\
\P^0 &= \Id.
\end{align*}
Moreover, these operations satisfy the following properties
\begin{enumerate}
\item the motivic Adem relations hold: if $0 < a < pb$, then
\[
\P^a \P^b = \sum_{t = 0}^{\lfloor a / p \rfloor} (-1)^{a+t} {(p-1)(b - t) - 1 \choose a - pt} \P^{a + b - t} \P^t,
\]
and if $0 < a \leq pb$, then
\begin{align*}
\P^a \B^b &= \sum_{t=0}^{\lfloor a / p \rfloor} (-1)^{a + t} {(p-1)(b-t) \choose a - pt} \B^{a+p-t}\P^t \\
&+ \sum_{t=0}^{\lfloor (a-1)/p \rfloor} (-1)^{a + t - 1} {(p-1)(b-t) - 1 \choose a - pt - 1} \P^{a + b - t} \B^t;
\end{align*}

\item the motivic Cartan formula holds\footnote{Note that unlike in Section~\ref{ssect:Cartan}, we do not know that the external product induces an equivalence as in Eq.~\eqref{eq:externalProdEquiv}. Nonetheless, as all the maps appearing in these formulas are stable under pullbacks, the equations here are consequence of the space-level equations of Theorem~\ref{thm:MotivicCartanSH}.} 
\[
\mu^*(\P^i) = \sum_{r = 0}^i \P^r \times \P^{i-r}
\]
and
\[
\mu^*(\B^i) = \sum_{r = 0}^i (\B^r \times \P^{i-r} + \P^r \times \B^{i-r})
\]
where $\mu^*$ is the pullback on motivic cohomology induced by the multiplication map $\mu \colon H\Fb^{\Ab^1}_p \otimes H\Fb^{\Ab^1}_p \to H\Fb^{\Ab^1}_p$;

\item the induced map
\begin{equation}\label{eq:PowerVsPOp}
\Omega^{\infty - i}_{\Tb} \P^i \colon \Omega_{\Tb}^{\infty-i} H\Fb^{\Ab^1}_p \to \Omega_{\Tb}^{\infty-pi} H\Fb^{\Ab^1}_p
\end{equation}
in $\Sh_{\Nis, \Ab^1 } (\Sm_{\Fb_p})$ coincides with the $p$th power map.

\end{enumerate}
\end{thm}
\begin{proof}
Claims 1 and 2 follow immediately from the corresponding claims in characteristic 0 and Theorem~\ref{thm:MotivicSteenrodAlgebraOverZpcyc} after realizing that $\rho$ and $\tau$ vanish in characterstic 2.

We prove claim 3 next. Note that for the mod-$p$ motivic cohomology of smooth $\Oc$-schemes the desired identification follows from \cite[Lemma~5.12]{voevodsky:2003a} because, by Corollary~\ref{cor:main}, the relevant mapping anima can be computed after restricting to the generic fibre. Moreover, in \cite[Theorem 8.1]{bem}, the authors show that the $\Zb(j)^{\Ab^1}$, together with the graded multiplicative structure, are stable under pullbacks of $\Ab^1$-invariant cdh-sheaves of anima\footnote{Left Kan extension from smooth schemes commutes with the forgetful functor $\mathrm{Forg} \colon D(\Zb) \to \Ani$ on the target category because, point wise, it is computed by sifted colimits. This is because for a (derived) scheme $X$, the category $(\Sm_{\Zb}^\op)_{/X}$, the opposite category of smooth $\Zb$-schemes under $X$, admits finite coproducts, and any such $\infty$-category is sifted.}. In particular, the pullback of $H\Fb_p^{\Ab^1}$ from $\Sh_{\cdh, \Ab^1}(\Sch_{\Oc})$ as a $\Tb$-spectrum to $\Sh_{\cdh, \Ab^1}(\Sch_{\Fb_p})$ may be computed as a lax $\Tb$-spectrum\footnote{See \cite[\S 1]{annala-iwasa:MotSp} for the definition of a lax spectrum object.} simply by applying the pullback of $\Ab^1$-invariant cdh-sheaves. Thus, the identification of maps in Eq.~\eqref{eq:PowerVsPOp} can be pulled from $\Sh_{\Nis, \Ab^1}(\Sm_{\Oc})$ to $\Sh_{\Nis, \Ab^1}(\Sm_{\Fb_p})$, finishing the proof.
%
%
%
%
%
%
%
\end{proof}

The above spectrum-level result immediately implies the following result for cohomology of motivic spaces.

\begin{cor}\label{cor:ModPSteenrod}
Let $S$ be a smooth $\Fb_p$-scheme. Then the endomorphisms $\P^i$ and $B^i$ of Theorem~\ref{thm:ModPSteenrodSpectrumLevel} induce mod-$p$ motivic power operations 
\[
\P^i \colon H^{\star}_{\Ab^1}(\Xfr; \Fb_p(\star)) \to H^{\star + 2i(p-1)}_{\Ab^1}(\Xfr; \Fb_p(\star + i(p-1)))
\]
and
\[
B^i \colon H^{\star}_{\Ab^1}(\Xfr; \Fb_p(\star)) \to H^{\star + 2i(p-1) + 1}_{\Ab^1}(\Xfr; \Fb_p(\star + i(p-1))
\]
acting on the $\Ab^1$-local motivic cohomology of motivic spaces $\Xfr \in \Pc_{\Nis,\Ab^1}(\Sm_S)$. We have that $\P^0 = \Id$ and $\B^i = \beta \P^i$. Moreover, these operations satisfy the expected properties, namely:
\begin{enumerate}
\item the operations $\P^i$ and  $B^i$ commute with pullbacks;
\item the motivic Adem relations hold;
\item the motivic Cartan formulas hold;
\item if $x \in H_{\Ab^1}^{2i}(\Xfr; \Fb_p(i))$, then $\P^i(x) = x^p$;
\item if $y \in H^{j}_{\Ab^1}(\Xfr; \Fb_p(k))$ is such that $j-k<i$ and $k \leq i$, then $\P^i(y) = 0$;
\item the Bockstein $\beta$ is a graded derivation with respect to the first grading.
\end{enumerate}
\end{cor}
\begin{proof}
The first four claims are immediate from Theorem~\ref{thm:ModPSteenrodSpectrumLevel}. Claim 5 follows from claim 4 (see \cite[Lemma~9.9]{voevodsky:2003a}). Claim 6 is formal.
\end{proof}

Furthermore, the following result is proven similarly to \cite[Proposition~11.4]{voevodsky:2003a} which only uses the action of motivic power operations on the mod-$p$ motivic cohomology of products of $\rm{B}\mu_p$, which in turn may be computed using Corollary~\ref{cor:ModPSteenrod}.

\begin{cor}\label{cor:admiss}
Admissible monomials are $H_{\Ab^1}^{\star}(\Fb_p; \Fb_p(\star))$-linearly independent in 
\[
\Ext_{\SH_{\Fb_p}}^{\star,\star}(H\Fb^{\Ab^1}_p,H\Fb^{\Ab^1}_p).
\] In other words, there is an inclusion 
\[
\Ac_{\Fb_p}^{\star,\star} \hook \Ext_{\SH_{\Fb_p}}^{\star,\star}(H\Fb^{\Ab^1}_p,H\Fb^{\Ab^1}_p),
\]
where $\Ac_{\Fb_p}^{\star,\star}$ is the $H\Fb_p^{\star,\star}$-algebra generated by the Bockstein $\beta$ and the motivic power operations $\P^i$, modulo the motivic Adem relations.
\end{cor}

\subsection{Power operations on $\Fb_p$-schemes}\label{subsect:allPops} We now construct power operations on the $\Ab^1$-invariant motivic cohomology of (quasicompact, quasiseprated) $\Fb_p$-schemes. Given $X \in \Sch_{\Fb_p}$ with structure map $\pi: X \rightarrow \Spec(\Fb_p)$, we have that $H\Zb^{\Ab^1}_X \simeq \pi^*H\Zb^{\Ab^1}_{\Fb_p}$ (as $\Eb_{\infty}$-algebras in $\SH_{\Fb_p}$) thanks to \cite[Theorem 8.19]{bem}.  Therefore we have a maps of $\Eb_1$-algebras (and $\Eb_{\infty}$-coalgebras):
\[
\mathrm{end}_{\SH_{\Fb_p}}(H\Fb_p^{\Ab^1}) \rightarrow \mathrm{end}_{\SH_{X}}(H\Fb_p^{\Ab^1}).
\]
On the level of homotopy groups, we then obtain a map (in the notation of Corollary~\ref{cor:admiss}
\[
\Ac_{\Fb_p}^{\star,\star} \rightarrow \Ext_{\SH_{\Fb_p}}^{\star,\star}(H\Fb^{\Ab^1}_p,H\Fb^{\Ab^1}_p) \rightarrow \Ext_{\SH_{X}}^{\star,\star}(H\Fb^{\Ab^1}_p,H\Fb^{\Ab^1}_p);
\]
which we base change along the $\Ab^1$-invariant motivic cohomology ring of $X$ to get a map
\[
\Ac_{X}^{\star,\star}:=\Ac_{\Fb_p}^{\star,\star} \otimes_{H_{\Ab^1}^{\star}(\Fb_p; \Fb_p(\star))} H_{\Ab^1}^{\star}(X; \Fb_p(\star)) \rightarrow \Ext_{\SH_{X}}^{\star,\star}(H\Fb^{\Ab^1}_p,H\Fb^{\Ab^1}_p).
\]
In this way, $\Ac_{X}^{\star,\star}$ acts naturally on the $\Ab^1$-invariant motivic cohomology of smooth $X$-schemes (more generally, on motivic spaces over $X$).

\begin{cor}\label{cor:extension-a1}
Let $X \in \Sch_{\Fb_p}$. There exist reduced mod-$p$ power operation endomorphisms acting on $H\Fb^{\Ab^1}_p \in \SH_X$ that satisfy all the properties stated in Theorem~\ref{thm:ModPSteenrodSpectrumLevel}.
\end{cor}

\begin{proof}
Indeed, the first properties follow from functoriality and symmetric monoidality of $\pi^*$, respectively. The third property follows from Theorem~\ref{thm:ModPSteenrodSpectrumLevel}(3) using the fact $\Zb(j)^{\Ab^1}$ on $\Sm_X$ is calculated via pullback of $\Ab^1$-invariant cdh-sheaves using  \cite[Theorem 8.1]{bem} again. 
\end{proof}

In Remark~\ref{rem:nona1-inv} we discuss the extension of the mod-$p$ power operations to non-$\Ab^1$-invariant motivic cohomology. 

%

\section{Applications}\label{sec:appl}

We now come to some geometric applications of our power operations. Throughout the section, we denote by $k$ a field of characteristic $p > 0$, unless stated otherwise. 

\subsection{Non smoothable cycles at the characteristic}\label{ssec:nonsmoothable}

In this section, we give examples of algebraic cycles that cannot be smoothed modulo $p$, where $p$ is the characteristic of the base field $k$. The case of mod-$l$ coefficients was treated in \cite{hartshorne-rees-thomas} using étale/singular cohomology. 

Let $R$ be ring of coefficients. The problem addressed in \emph{op. cit.} was posed by Borel--Haefliger (for singular cohomology) \cite{borel-haefliger}: when can write one cycle $\alpha \in R \otimes \CH^n(X)$ as a  $R$-linear combination of nonsingular subvarieties of $X$? If $R  = \Zb$ then \cite[Theorem 1]{hartshorne-rees-thomas} gives a counterexample.

Our goal in this section is to give a counterexample to an even weaker version of the question in characteristic $p$ (so it has a better chance of being true): instead of asking that the subvariety is nonsingular we can ask that it is regularly immersed in $X$\footnote{Note that any smooth subvariety is regularly embedded.}. In fact, the language of derived algebraic geometry\footnote{We use very little of this theory; the formalism used in  \cite{khan-rydh} more than suffices.} provides a reasonably large class of immersions which we could hope to generate the group $R \otimes \CH^n(X)$.

\begin{defn}\cite[2.3.6, Proposition 2.3.8]{khan-rydh}\label{def:qsmooth} A closed immersion\footnote{Which means that it is a closed immersion on underlying classical schemes.} $i: Z \hookrightarrow X$ of derived schemes is \emph{quasi-smooth} if the cotangent complex $\Lbf_{Z/X}[-1]$ is a locally free $\mathcal{O}_Z$-module of finite rank. 
\end{defn}

\begin{ex}\label{exam:regular} Assume, as we do in this section, that $X$ is a smooth $k$-scheme and $Z$ is classical. Then $i$ is quasi-smooth if and only if $i$ is a regular immersion in the sense of being Koszul regular (see \cite[2.1.1, 2.3.6]{khan-rydh} for a review). In particular, if $Z$ is also a smooth $k$-scheme then $i$ is quasi-smooth (see \cite[Tag 0E9J]{stacks} and note that smooth schemes over a field are regular).
\end{ex}

From the rest of this subsection, we fix $X$ to be a smooth $k$-scheme. If $Z \subset X$ is a closed subvariety such that $Z$ is smooth then we can associate to $Z$, its fundamental class $[Z] \in \CH^{\mathrm{codim}(Z)}(X)$. The work of Khan extends this to quasi-smooth closed immersions. In what follows, we use Khan's extension of $\SH$ to derived schemes \cite{khan-mv,KhanThesis}; note that by \cite[Corollary 3.2.7]{khan-mv} we have that $\SH_X \simeq \SH_{X^{\mathrm{cl}}}$ where $X^{\mathrm{cl}} \hookrightarrow X$ is the classical reduction of $X$. On $\SH_X$, the motivic spectrum $H\Zb^{\Ab^1}$ is defined by pullback from $\Spec(\Zb)$. 

\begin{cons}[Khan] Suppose that $i \colon Z \hook X$ is a quasi-smooth closed immersion with normal sheaf $\Nc_i \simeq (\Lbf_{Z/X}[-1])^{\vee}$, then the trace associated to the purity transformation \cite[2.5.3]{DJKFundamental} \cite[Remark~3.8, Variant~3.11]{khan:virtualI},
\[
\mathrm{tr}_i \colon i_* \Sigma^{-\Nc_{i}} i^* \to \Id
\]
induces for all absolute $\Ab^1$-invariant motivic spectra $E$ the map
\begin{align}\label{eq:ThomGysin}
\mathrm{tr}_{i*} \colon& E(\Th_Z(\Nc_i),-n) =  \maps_{\SH_X}(\Sigma^n_{\Tb} 1_X, i_* \Sigma^{-\Nc_{i}}i^*E) \\
\to& E(X,-n) =  \maps_{\SH_X}(\Sigma^n_{\Tb} 1_X, E) \nonumber
\end{align}
which is an $\Ab^1$-homotopy theoretic analogue to the pullback along the Thom collapse map in topology. In Eq.~\eqref{eq:ThomGysin}, we have used the assumption that $E$ is absolute in the first identification. 

If $c$ is the virtual codimension of $Z$ in $X$ \cite[2.3.11]{khan-rydh}, then we have the Thom class\footnote{Whenever $E \in \SH_Z$ is oriented, the Thom class can be defined in terms of Chern classes, see e.g. Lemma~\ref{lem:ChernClassFormula}. Multiplying with $t(\Nc_{i})$ is induces the \textit{Thom isomorphism} $t(\Nc_i) \cdot \dash \colon E(\Sigma^{c}_\Tb 1_X) \xto{\sim} E(\Th_Z(\Nc_i))$.} in
\[
t(\Nc_{i}) \in H_{\Ab^1}^{2c}(\Th_Z(\Nc_i); \Zb(c)) = \pi_0\maps_{\SH_X}(\Sigma^{-c}_{\Tb} 1_X, i_* \Sigma^{-\Nc_{i}} H\Zb^{\Ab^1}) 
\]
If we assume now that $X$ is smooth over $k$, then $H^{2c}_{\Ab^1}(X; \Zb(c)) \cong \mathrm{CH}^c(X)$ and the \emph{virtual fundamental class} of $Z$ is defined as the image
\[
[Z] := \mathrm{tr}_{i*}(t(\Nc_{i})) \in \mathrm{CH}^c(X). 
\]
For any coefficient ring $R$, the image of $[Z]$ under the map $\mathrm{CH}^c(X) \rightarrow R \otimes \mathrm{CH}^c(X)$ will still be denoted as $[Z]$. 
\end{cons}

\begin{rem} By \cite[\S 3.3]{khan:virtualI}, the classes $[Z] \in \Qb \otimes \CH^\star(X)$ coincide with the pushforwards of the virtual fundamental classes defined by Behrend--Fantechi \cite{behrend-fantechi}. Therefore, one can think of the construction in the above paragraph as an integral refinement of virtual fundamental classes. Moreover, $[Z]$ coincides with the usual fundamental class of $Z$ whenever $Z$ is regularly immersed
\end{rem}

\begin{rem}[Compatibility with operations] \label{rem:compatible} We remark on the compatibility of $\mathrm{tr}_{i*}$ and power operations. A motivic cohomology operation on $H\Fb_p^{\Ab^1}$ is a morphism in $\SH_X$ given by $H\Fb_p^{\Ab^1} \xrightarrow{o} \Sigma^n_{\Tb}H\Fb_p^{\Ab^1}[m]$. Since the formation of $\mathrm{tr}_i$ upgrades to a transformation
\[
\SH_X \rightarrow \SH_X^{\Delta^1} \qquad E \mapsto \mathrm{tr}_i \colon i_* \Sigma^{-\Nc_{i}} i^* E \to E;
\]
we obtain a commutative diagram
\[
\begin{tikzcd}
i_* \Sigma^{-\Nc_{i}} i^*H\Fb_p^{\Ab^1} \ar{r}{o} \ar{d}{ \mathrm{tr}_i} &  i_* \Sigma^{-\Nc_{i}} i^*\Sigma^n_{\Tb}H\Fb_p^{\Ab^1}[m] \ar{d}{ \mathrm{tr}_i}\\
H\Fb_p^{\Ab^1} \ar{r}{o} & \Sigma^n_{\Tb}H\Fb_p^{\Ab^1}[m] \\
\end{tikzcd}
\]
This is the sense in which operations commute with the $\mathrm{tr}_{i*}$. However, we remark that the power operations \emph{do not} commute with the Thom isomorphism; the Riemann-Roch fomula discussed in~\S\ref{sec:riemann-roch} measures the extent to which this is true.
\end{rem}

We are now ready to prove the main result of this subsection.

\begin{thm}\label{thm:non-qs}
For every prime $p>0$, there exists $m,n \in \Nb$ such that the second Chern class of  the canonical sub bundle $\Ec$,
\[
c_2(\Ec) \in \CH^2(\Gr_k(m,n)),
\]
is not quasi-smoothable with $\Fb_p$-coefficients, where $k$ is any field of characteristic $p$.
\end{thm}
\begin{proof} 
For conciseness, we will denote $\Gr_k(m,n)$ by $X$. We will give different but related arguments for the cases $p=2$ and $p>2$. 
\begin{enumerate}
\item \textit{Case $p=2$}: We take $m = 3$. Let us denote the Chern classes $c_1(\Ec), c_2(\Ec)$, and $c_3(\Ec)$, by $x,y$, and $z$, respectively. A general degree-two element in the cohomology of $X$ is of form $\alpha x^2 +  \beta y$, where $\alpha,\beta \in \Fb_2$. We will show that for a class of form $[Z]$, $\beta = 0$, and therefore $y$ is not quasi-smoothable with $\Fb_2$-coefficients.

Let $i \colon Z \hook X$ be a derived regular embedding of codimension 2, and let $t \in H\Fb_2^{\Ab^1}(\Th_Z(\Nc_i), 2)$ be the Thom class. The equality $\Th_Z(\Nc_i) = \Pb_Z(\Nc_i \oplus \Oc) / \Pb_Z(\Nc_i)$ identifies the $H\Fb^{\Ab^1}_2$-cohomology of $\Th_Z(\Nc_i)$ with the ideal generated by 
\[
t = c_2(\Nc_i) + c_1(\Nc_i) c_1(\Oc(1)) + c_1(\Oc(1))^2 \in H\Fb_2^{\Ab^1}(\Pb_Z(\Nc_i \oplus \Oc), 2).
\]
The action of the motivic Steenrod squares on Chern classes can be computed in the usual way using splitting principle and the Cartan formula. In particular, $\Sq^2(c_1) = c_1^2$ and $\Sq^2(c_2) = c_1 c_2 + c_3$. Using these identities, together with the Cartan formula and the fact that $c_3(\Nc_i) = 0$, we compute that 
\begin{equation}
\Sq^2(t) = c_1(\Nc_i) t.
\end{equation}
As the map $\mathrm{tr}_{i*}$ of Eq.~\eqref{eq:ThomGysin} induces a map of non-unital rings on cohomology, we see that
\begin{align*}
\Sq^2([Z])^2 
&= \mathrm{tr}_{i*} (c_1(\Nc_i)^2 t \cdot t) \\
&= \mathrm{tr}_{i*} (c_1(\Nc_i)^2 t) \cdot [Z],
\end{align*}
that is, $[Z]$ divides $\Sq^2([Z])^2$. We will show that this is impossible unless $\beta = 0$.

Using splitting principle, we compute that $\Sq^2(x^2) = 0$ and $\Sq^2(y) = xy+z$. Thus,
\[
\Sq^2(\alpha x^2 + \beta y)^2 = \beta^2 (x^2 y^2 + z^2). 
\]
In conclusion, if $\beta \not = 0$, then $x^2 y^2 + z^2$ should be divisible by $\alpha x^2 + \beta y$, which is impossible if $n \geq 9$, because the lowest-degree algebraic relation between $x,y$, and $z$ in the $H\Fb_2^{\Ab^1}$-cohomology of $\Gr_k(3,9)$ has degree 7 \cite[\S 2]{hartshorne-rees-thomas}. 

\item \textit{Case $p>2$}: Again, a general degree-two element in the cohomology of $X$ is of form $\alpha c_1^2 + \beta c_2$, where $\alpha, \beta \in \Fb_p$, and $c_i$ are the Chern classes of the canonical sub bundle $\Ec$. We will show that a class of form $[Z]$ will have $\beta = 0$, and therefore $c_2$ is not quasi-smoothable.

As the Thom class $t$ of $\Th_Z(\Nc_i)$ generates the cohomology of the Thom space as a module over the cohomology of $Z$, 
\begin{equation}
\P^1(t) =  ct
\end{equation}
for some cohomology class $c$. Thus, if $[Z] = \mathrm{tr}_{i*}(t)$, then $\P^1([Z])^2$ is divisible by $[Z]$. We show that this is impossible unless $\beta = 0$ under certain assumptions on $m$ and $n$ that we are going to explain next.

Let $m \geq 4$ be an even integer such that $-(m-2)/2$ is a non-square unit in $\Fb_p$, and let $n \geq m + 2p + 2$. Under the assumption on $n$, the cohomology ring of $X$ coincides with the polynomial ring generated by $c_1, \dots, c_m$ in degrees less than equal to $2p + 2$ \cite[\S 2]{hartshorne-rees-thomas}. Thus, it suffices to show that $\alpha c_1^2 + \beta c_2$ does not divide $P^1(\alpha c_1^2 + \beta c_2)^2$ in $\Fb_p[c_1, \dots , c_m]$. Furthermore, we may use the splitting principle to check this indivisibility in the ring $\Fb_p[a_1,\dots, a_m]$, where $a_i$ are the Chern roots of $\Ec$. Lastly, we may assume that $\beta = 1$.

As
\begin{equation*}
\P^1(c_1^2) = 2(a_1 + \cdots + a_m)^{p+1} \in \Fb_p[a_1,\dots, a_m] 
\end{equation*}
and
\begin{equation*}
\P^1(c_2) = \sum_{i < j} (a_i^p a_j + a_i a_j^p) \in \Fb_p[a_1,\dots, a_m]
\end{equation*}
we want to show that
\begin{equation*}
f := \alpha (a_1 + \cdot + a_m)^2 + \sum_{i < j} a_i a_j 
\end{equation*}
does not divide the square of 
\begin{equation*}
g =  2 \alpha (a_1 + \cdots + a_m)^{p+1} + \sum_{i < j} (a_i^p a_j + a_i a_j^p).
\end{equation*}
It suffices to check this indivisibility after assigning values to the variables. Next, we will use carefully chosen assignments to turn both $f$ and $g$ into univariate polynomials for which the indivisibility is easy to check. 

First, to eliminate the dependence on $\alpha$, we set $a_m = -(a_1 + \cdots + a_{m-1})$. This turns the polynomials $f$ and $g$ into
\[
f_2 = - \sum_{i \leq j} a_i a_j \in \Fb_p[a_1, \dots, a_{m-1}]
\]
and 
\[
g_2 = - (a_1^{p+1} + \cdots + a_{m-1}^{p+1}) - (a_1 + \cdots + a_{m-1})^{p+1} \in \Fb_p[a_1, \dots, a_{m-1}],
\]
respectively. Finally, we assign $a_1 = a$ and $a_i = (-1)^i$ for $i \in \{ 2, \dots, m-1 \}$. The polynomials $f_2$ and $g_2$ turn into 
\[
f_3 = -(a^2 + \tfrac{m-2}{2})  
\]
and 
\[
g_3 = -2(a^{p+1} + \tfrac{m-2}{2}).
\]
To show that $f_3$ does not divide the square of $g_3$, we check that neither of the roots of $f_3$ is a root of $g_3$. The roots of $f_3$ are 
\[
\xi_\pm  = \pm \sqrt{-\tfrac{m-2}{2}}.
\]
By the assumption that $-(m-2)/2$ is not a square in $\Fb_p$, we have that $\xi_\pm^{p-1} = -1$, and therefore $\xi^{p+1}_\pm = \tfrac{m-2}{2}$. In particular $\xi_\pm$ are not roots of $g_3$, and thus $f_3$ does not divide the square of $g_3$, finishing the proof that a cohomology class of form $\alpha c_1 + c_2$ is not quasi-smoothable. \qedhere
\end{enumerate}
\end{proof}

\begin{rem}\label{rem:HRTcounterexample}
Hartshorne--Rees--Thomas show that $c_2(\Ec) \in \CH^2(\Gr_\Cb(3,6))$ is not smoothable  using action of the mod-2 power operations. We suspect that their argument can be used to improve $m$ in the proof of Theorem~\ref{thm:non-qs} when $p=2$. However,  we have chosen to use a simplified proof, and the price we pay is the non-optimality of $m$.
\end{rem}

\begin{rem}\label{rem:primozic} We remark that the operations in \cite{primozic:2020} could have been used to prove this result since it involves only the study of power operations along the Chow diagonal. For the next applications, however, it is crucial that we work with our version of the power operations. 
\end{rem}

\subsection{Obstructions against lifting to $\MGL$ and motivic Steenrod problem at the characteristic}\label{sec:againstmgl}

This subsection is a mod-$p$ counterpart to the work of the first author and Shin \cite{annala-shin}. Unlike the previous application, we do use behavior of the motivic Steenrod operations away from the Chow diagonal. It also involves non-$\Ab^1$-invariant motivic cohomology as we are primarily concerned with singular schemes over a field $k$. 

Motivated by the isomorphism of graded rings in the case of smooth $k$-schemes: $H^{2\star}_{\mot}(X; \Zb(\star)) \cong \mathrm{CH}^{\star}(X)$ (see, for example, \cite{voevodsky-chow}), one defines 
\[
X \mapsto H^{2\star}_{\mot}(X; \Zb(\star)) =: H^{2\star}(\Zb(\star)^{\mot}(X))
\] for any qcqs $k$-scheme $X$ as an extension of the theory of Chow rings to singular schemes. Here $\Zb(\star)^{\mot}$ are the motivic complexes of \cite{elmanto-morrow}; we will also consider its $p$-localization $\Zb_{(p)}(\star)^{\mot}(X):= (\Zb(\star)^{\mot}(X))_{(p)}$ and mod-$p$ reduction $\Fb_p(\star)^{\mot}(X):= \Zb(\star)^{\mot}(X)/p$. This perspective is explored further in upcoming work of the second author and Morrow where this definition verifies the expectations of Srinivas \cite{srinivas} for an intersection theory on singular varieties.

Based on this premise, the first author and Shin addressed a version of Steenrod's problem in algebraic geometry using $p$-localized motivic cohomology: suppose that $X$ is a $k$-variety, and $n \geq 0$, is the abelian group $H^{2n}_{\mot}(X; \Zb_{(p)}(n))$ generated by the projective pushforward of fundamental classes of quasi-smooth, derived $X$-schemes? This question was inspired by the independent suggestions by the first author and Adeel Khan that it might be possible to construct a good theory of Chow rings of singular schemes as a quotient of the ring of equivalence classes of quasi-smooth derived $X$-schemes \cite{annala-cob, annala-thesis, khan:2020}. See \cite[Question 1.1]{annala-shin} for a precise formulation of the question. If $p$ is invertible in $k$, and $X$ is smooth, then the answer is yes; if resolution of singularites holds in characteristic $p > 0$ then the result holds when $p$ is not invertible as well (but $X$ is still smooth).

The main theorem of \emph{op. cit.} is a negative answer to the question for a singular $X$ when $p$ is invertible in $k$. We now address this problem when $p$ is the characteristic of $k$. To begin, we prove the following result which is of independent interest in the $\Ab^1$-invariant world and is new in characteristic $p > 0$. Define the \emph{motivic Milnor operations} (also called \emph{motivic Milnor primitives}) acting the mod-$p$ motivic cohomology as
\begin{equation}\label{eq:milnor}
\Q_n= q_n \beta - \beta q_n
\end{equation}
where $\beta$ is the Bockstein and $q_n = \P^{p^{n-1}} \cdots \P^p \P^1$ (see e.g. \cite[Bottom of p.199]{hoyois:2013}).

\begin{thm}\label{thm:MGLObstructions}
Let $k$ be a field of characteristic $p>0$ and $\eta \colon \MGL \to H\Fb_p$ be the map in $\SH_k$ that classifies the additive formal group law with coefficients in $\Fb_p$. Then 
\[
\Q_n \circ \eta \simeq 0 \in \SH_k
\]
for all $n$. In particular, if a class in motivic cohomology has non-vanishing of motivic Milnor operations, then it is not liftable to an $\MGL$-cohomology class.
\end{thm}
\begin{proof}
It suffices to prove this in the case $k = \Fb_p$. Over $\Kb$, this follows from \cite[Lemma~6.13]{hoyois:2013}. By Theorem~\ref{thm:Detroit}, $\underline{\maps}_{\SH_{\Oc}}(\MGL, H\Fb_p) \simeq \underline{\maps}_{\SH_{\Kb}}(\MGL, H\Fb_p)$, so we have that $\Q_n \circ \eta \simeq 0 \in \SH_{\Oc}$. Finally, as the motivic power operations on $\Fb_p$ are pulled back from $\Oc$, $\Q_n \circ \eta \simeq 0 \in \SH_{\Fb_p}$, as desired. 
\end{proof}

\begin{rem}[Total obstruction to lifting against $\eta$] Even if all the $\Q_n$-operations vanish for $\alpha \in H^{i}_{\mot}(X; \Fb_p(j))$, it does not mean that the $\alpha$ lifts to a class in $\MGL$. 

An important case is when $i = 2j$ and $X$ is smooth. For degree reasons, the $\Q_n$'s vanish and we expect that the map $\MGL^{2n,n}(X) \rightarrow H_{\mot}^{2n}(X; \Fb_p(n)) \cong \CH^n(X)/p$ is surjective. Indeed, this is the case if $k$ admits resolution of singularities. This story is related to a characteristic $p$ version of Totaro's seminal work on obstruction of algebraizability of cohomological classes via complex bordism \cite{totaro-bordism} (see also Quick's work for the story away from the characteristic \cite{quick-etale}).

\end{rem}


\begin{ex}[Non liftable cycles on $\mathrm{B} \mu_p$] \label{ex:CohOfBmup}
Let $k$ be a characteristic $p > 0$ field and let $\mathrm{B} \mu_p$ be the classifying stack of $\mu_p$-torsors in the fppf-topology on $\Sm_{k}$, and $\Xfr \in \SH_k$. Recall that \cite[Theorem 6.10]{voevodsky:2003a} records a calculation mod-$p$ motivic cohomology of $\Xfr \times \mathrm{B}\mu_p$ at any characteristic. In our case the answer is as follows:
\[
H^{\star}_{\mot}(\Xfr \times \mathrm{B} \mu_p; \Fb_p(\star)) \cong H_{\mot}^{\star}(\Xfr; \Fb_p(\star))[[u, v]]/(u^2 = 0, \beta(u)= v) \qquad |u| = (1,1), |v| = (2,1).
\]
Here, $\beta$ is the Bockstein operator. We remark that $v$ is the pullback of the first Chern class of the universal line bundle along the map $\mathrm{B}\mu_p \to \mathrm{B}\Gb_m$ and $u$ is the unique class such that $\beta(u) = v$.

Let $\Xfr = \mathrm{B} \mu_p$ and let $u_i$ and $v_i$ be element in its motivic cohomology given by the pullbacks of $u$ and $v$ from the respective factors. By the basic properties of motivic cohomology operations Theorem~\ref{cor:ModPSteenrod} (in particular, Lemmas~\ref{lem:milnor} and~\ref{lem:Q_n-bmup}), we see that
\begin{align*}
\Q_1(u_1 \cdot u_2) &= (\P^1 \beta - \beta \P^1)(u_1 \cdot u_2) \\
&= v_1^p \cdot u_2 - u_1 \cdot v_2^p \\
&\not = 0.
\end{align*}
Thus, $u_1 \cdot u_2$ does not lift to an $\MGL$-cohomology class.

We note that the calculation \cite[Proposition~6.4]{AHI} shows that the answer for syntomic cohomology (an oriented theory with elementary blowup excision, the $\Pb^1$-bundle formula and the additive formal group law) is similar: for any scheme/stack $\Xfr$ we have
\begin{equation}\label{eq:syn}
H^{\star}_{\syn}(\Xfr \times \mathrm{B} \mu_p; \Fb_p(\star)) \cong H_{\syn}^{\star}(\Xfr; \Fb_p(\star))[[u, v]]/(u^2 = 0, \beta(u)= v) \qquad |u| = (1,1), |v| = (2,1).
\end{equation}
We will use this calculation in what follows.
%
\end{ex}

Following \cite{annala-shin}, we solve the motivic Steenrod problem at the characteristic for singular varieties.

\begin{thm}
There exists a singular $\Fb_p$-variety $X$ such that its $p$-local motivic cohomology does not have the Steenrod property. In other words, there exists $n \in \Nb$ such that the group $H^{2n}_{\mathrm{mot}}(X; \mathbb{Z}_{(p)}(n))$ is not generated, as a $\Zb_{(p)}$-module, by pushforwards of fundamental classes of quasi-smooth projective derived $X$-schemes.
\end{thm}
\begin{proof} We prove this result in steps, following \cite{annala-shin} with appropriate modifications.

\begin{enumerate}
\item[(Step 1)] Naturally on any $\Fb_p$-scheme $X$ we have a commutative diagram for all $n \geq 0$:
\[
\begin{tikzcd}
\underline{\Omega}^n(X)_{(p)} \ar{d} \ar{r} & \pi_{-2n}\MGL(n)(X)_{(p)} \ar{d}\\
H^{2n}_{\mot}(X; \Zb_{(p)}(n)) \ar{r} & H^{2n}_{\Ab^1}(X; \Zb_{(p)}(n)), \\
\end{tikzcd}
\]
as explained in \cite[\S 4.2]{annala-shin}. Above, $\PCob^*(X)$ is the universal precobordism ring introduced in \cite{annala-yokura, annala-chern}. By \cite[Theorem 2.4]{annala-shin}. The claim is then equivalent to constructing a class $\alpha \in H^{2n}_{\mot}(X; \Zb_{(p)}(n))$ which does not lift along the left vertical map. To do so, we will construct a class $\alpha_{\Ab^1} \in H^{2n}_{\Ab^1}(X; \Zb_{(p)}(n))$ which \emph{does not} lift along the right vertical map but \emph{does} lift along the lower horizontal map. 

\item[(Step 2)] The calculation of \cite[\S 4.2]{annala-shin} produces a singular $k$-variety $X'$ and a $p$-torsion class $\beta_{\Ab^1} \in H^{4}_{\Ab^1}(X'; \Zb(2))$ which fails to lift to $\pi_{-2n}\MGL(n)(X)$. The same construction there works here once we know Theorem~\ref{thm:MGLObstructions}. We sketch this: one constructs a smooth $k$-variety $X''$ with a rational point and a class $u \in H^{2}_{\Ab^1}(X''; \Zb/p(2))$ such that $Q_1\beta u \not=0$ \cite[Proposition 4.7]{annala-shin} hence does not lift to $\MGL$ by Theorem~\ref{thm:MGLObstructions}. The construction of \cite[Construction 3.2]{annala-shin} then produces variety $X'$ such that it has the $\Ab^1$-homotopy type of the reduced suspension $\Sigma(X'')$. Taking the suspension  of the integral Bockstein of $u$ then produces a class $u' \in H^{4}_{\Ab^1}(X'; \Zb(2))$ which is $p$-torsion and does not lift to $\pi_{-4}\MGL(2)(X)$.

\item[(Step 3)] Next, we need to modify $X'$ and $u'$ so that it lifts to motivic cohomology. Observe that if $y$ is a motivic cohomology class in $H^{2n}_{\mot}(Y; \Zb(n)) = H^{2n}_{\Ab^1}(Y; \Zb(n))$ on a smooth variety $Y$ such that the image of $y$ vanishes in $H^{2n}_{\syn}(Y; \Zb_p(n))$, then the class of the external product $ u' \times y$ lifts to a class in $H^{2n}_{\mot}(X' \times Y; \Zb(n))$. Indeed: under the map $H^{2n}_{\Ab^1}(X' \times Y; \Zb(n))\rightarrow H^{2n}(L_{\cdh}\Zb_p^{\syn}(n)(X' \times Y))$, the image of the class $u' \times y$ maps to zero since it is given by a product with image of $y$ under cdh-sheafification, which is assumed to be zero. Hence it lifts to $H^{2n}_{\mot}(X' \times Y; \Zb(n))$ thanks to the long exact sequence induced by the cartesian square in \cite[Theorem~1.5]{elmanto-morrow}.

Let $Y$ be a smooth affine variety that is $\Ab^1$-homotopy equivalent to $\Pb^2$ (e.g. a Jouanolou device), and let $y$ be the generator of $H^{4}_{\mot}(Y;\Zb(2)) = \Zb\{ y \}$. As $Y$ is affine, $H^{4}_{\syn}(Y; \Zb_p(2)) = H^{2}_{\et}(Y; W\Omega^2_{\mathrm{log}})$ by \cite[Corollary 8.21]{BMS2} and the latter is zero since \'etale cohomology of the logarithmic Hodge-Witt sheaves are concentrated in degrees zero and one by the (higher) Artin-Schreier sequence \cite[Th\'eor\`eme I.5.7.2]{illusie-drw}. In the final step we explain that $X = X' \times Y$ and the class 
\[
u' \times y \in H^{8}_{\Ab^1}(X' \times Y; \Zb(4))
\]
lifts to 
\[
\alpha \in H^{8}_{\mot}(X' \times Y; \Zb_{(p)}(4)),
\]
and presents the desired solution. 

\item[(Step 4)] As stated in step 1, it is sufficient that $u' \times y$ is a $p$-torsion class, that does not lift to $\pi_{-8}\MGL(4)(X)$ but does lift to $H^{8}_{\mot}(X' \times Y; \Zb_{(p)}(4))$. The last requirement is assured by the discussion of Step 3. That it does not lift to $\pi_{-8}\MGL(4)(X)$ follows because $\Q_1$ acts nontrivially on its mod-$p$ reduction by the Cartan formula and the fact that both $\P^1$ and $\beta$ kill the mod-$p$ reduction of $y$, first for degree reasons and the second because $y$ lifts to an integral class. Finally, it is clearly $p$-torsion and hence defines a class in $p$-local motivic cohomology. \qedhere
\end{enumerate}

\end{proof}

\subsection{The integral Tate conjecture at the characteristic}

In this section, we formulate optimistic versions of the crystalline Tate conjecture. Actually, we will immediately supply a counterexample to this conjecture so we should actually call it the integral Tate \emph{condition}: the veracity of the statement is then an interesting, geometric condition on the variety in question. This version of the Tate conjecture does not seem to be as popular as its $\ell$-adic counterpart. 

\begin{rem}
To the authors' knowledge, the first paper that discusses the rational version extensively is Morrow's variational take on it \cite{morrow-variational}. An integral variant for divisors, using flat cohomology for the target, was first studied by D'Addezio \cite{daddezio} who proved it for abelian varieties over finite fields. \end{rem}

To formulate this conjecture, we recast the cycle class map via \'etale realization in the environment of non-$\Ab^1$-invariant motivic spectra $\MS_S$. For the rest of the paper, $(\Sm^?_X, \MS^?_X)$ denotes either
\begin{enumerate}
\item $\Sm_X$ and $\MS_X$ as defined in \cite{AHI:atiyah}, or
\item the category of almost-finite-presentation derived $S$-schemes $\dSch^\afp_X$ and the category $\MS^\dbe_X$ of motivic spectra satisfying derived blowup excision, as defined in \cite[Appendix~B]{annala-shin}.
\end{enumerate}

For a scheme $X$, we denote by
\[
M_S(-): \Sm^?_X \rightarrow \MS^?_X
\]
the motivic spectrum associated to a smooth $S$-scheme $X$; in the notation  of \cite{AHI}, this would be denoted by $(-)_+$. We have the endofunctor
\[
L_{\et}: \MS_X \rightarrow \MS_X,
\]
that is the localization functor onto $\MS^{\et}_X \subset \MS_X$ the full subcategory of \'etale-local objects. If $E \in \MS^?_X$, we define $E^{\star,\star}(X)$ by the same convention as in the $\Ab^1$-invariant setting defined in~\eqref{eq:e-cohomology}. We will denote by $H \Zb^\syn_p \in \MS_X$ the pullback from $\Spec(\Zb)$ of the motivic spectrum $(H \Zb_p^\syn)_{\Spec(\Zb)}$ representing syntomic cohomology of schemes in the sense of \cite{bhatt-lurie:apc}; see also \cite[Section 4]{bem}. For our purposes, the reader should note that for a scheme $X$ which is smooth over a perfect field $k$ of characteristic $p > 0$, we have an isomorphisms:
\[
H \Zb^\syn_p(n)(X) \simeq \lim_r R\Gamma_{\et}(X; W_r\Omega^n_{\log})[-n]. 
\]

\begin{prop}\label{prop:etale}
Let $S$ be a spectrum of a Dedekind domain or a field. Then there is an equivalence of $\mathbb E_{\infty}$-algebras in $\MS_S$:
\[
(L_\mathrm{et} H\Zb^{\Ab^1})^\comp_p \xrightarrow{\simeq} H \Zb^\syn_p \in \MS_S.
\]
\end{prop}

\begin{proof} This result is essentially due to Geisser \cite{geisser-dedekind}; we explain this using the language developed in \cite{bem}. On smooth $S$-schemes, \cite[Theorem 5.8]{bem} constructs a multiplicative equivalence of graded presheaves (compatible with mod $p$-reduction as $m$ varies) $\Zb^{\Ab^1}(\star)/p^m \rightarrow L_{\Nis}\tau^{\leq \star}\Zb_p^{\syn}(\star)/p^m$, compatible with the first chern class. Therefore, by post-composing along the map $L_{\Nis}\tau^{\leq \star}\Zb_p^{\syn}(\star)/p^m \rightarrow \Zb_p^{\syn}(\star)/p^m$ and noting that the target has \'etale descent we obtain a multiplicative map of graded presheaves (again compatible with mod $p$-reduction as $m$ varies) which are compatible with the first chern class:
\[
L_{\et}\Zb^{\Ab^1}(\star)/p^m \rightarrow \Zb_p^{\syn}(\star)/p^m.
\]
We now observe that this map is 1) is an equivalence (essentially \cite[Theorem 1.2(2)]{geisser-dedekind}), 2) the first chern class (induced from the one coming from $H\Zb^{\Ab^1}$) induces a $\Pb^1$-bundle formula on mod-$p^r$ syntomic cohomology  and 3) mod-$p^r$ syntomic cohomology enjoys elementary blowup excision as verified in, say, \cite[Section 4]{bem}. A combination of these three facts proves that we have a multiplicative equivalence in $\MS_S$:
\[
L_\mathrm{et} H\Zb^{\Ab^1}/p^m \simeq H \Zb^\syn_p/p^m;
\]
and thus taking inverse limit as $m \rightarrow \infty$ gives us the desired result.
\end{proof}

\begin{rem}\label{rem:hz} For an equicharacteristic scheme $X$, the motivic cohomology of \cite{elmanto-morrow} defines a motivic spectrum $H\Zb \in \MS_X$ since the multiplicative graded collection of presheaves $\{\Zb(\star)^{\mot} \}$ has Nisnevich descent, has the $\Pb^1$-bundle formula (after a construction of a $c_1$) and enjoys elementary blowup excision \cite[Theorem 1.1]{elmanto-morrow}. If $X$ is a field or, more generally, a regular noetherian equicharacteristic scheme then $H\Zb \simeq H\Zb^{\Ab^1}$ by \cite[Theorem 6.1]{elmanto-morrow}. In this case, Proposition~\ref{prop:etale} reads: 
\begin{equation}\label{eq:let-syn}
(L_\mathrm{et} H\Zb)^\comp_p = H \Zb^\syn_p.
\end{equation} For a more general scheme $S$, Bouis \cite{bouis-mixed} has constructed $H\Zb \in \MS_X$ which extends the construction of \cite{elmanto-morrow}. In upcoming work of Bouis and Kundu, they have also established $H\Zb \simeq H\Zb^{\Ab^1} \in \MS_X$ for any Dedekind domain $S$. We expect (but cannot yet prove) that $H\Zb \simeq H\Zb^{\Ab^1}$ for a regular noetherian scheme $S$.
\end{rem}

\begin{rem}\label{rem:nona1-inv}
Having defined the environment $\MS_X$ and the motivic spectrum $H\Zb$ we discuss an extension of the power operations to non-$\Ab^1$-invariant motivic cohomology. One can extend the motivic sheaves $\Zb(j)^{\Ab^1} \colon \Sm_{\Fb_p}^\op \to D(\Zb)$ to sheaves on all $\Fb_p$-schemes by applying the procdh local left Kan extension defined by Kelly--Saito \cite{kelly-saito}. As $L_\procdh$ is symmetric monoidal localization\footnote{By \cite[Proposition~A.5]{nikolaus-scholze}, it suffices to check that the symmetric monoidal strucutre on presheaves preserves procdh-equivalences separately in both variables. This is evident by \cite[Definition~2.1]{kelly-saito}, because proabstract blowup squares are stable under pullbacks.}, it induces a functor on lax $\Tb$-spectra, and one obtains the action of the power operations on the lax $\Tb$-spectrum\footnote{See \cite[\S 1]{annala-iwasa:MotSp} for the definition of a lax spectrum.}
\[
	H\Zb^{\procdh} := L_\procdh H\Zb^{\Ab^1} \in \Sp^\lax_\Tb(\Pc_\Nis(\Sch_{\Fb_p}; \Sp)) 
\] 
satisfying the Adem relations, the Cartan formula, and the instability relation. 

If $X$ is a noetherian $\Fb_p$-scheme of finite Krull dimension, then by \cite[Corollary~1.11]{kelly-saito}, $H\Zb^{\procdh}$ coincides with the motivic cohomology defined by Elmanto--Morrow on smooth $X$-schemes. As this theory is known to satisfy smooth blowup excision and projective bundle formula, we observe that the action of the $\P^i$ and $\B^i$ on the lax $\Tb$-spectrum $H\Zb^{\procdh}$ restricts to an action on the $\Tb$-spectrum
\[
H\Zb \in \MS_X = \Sp_\Tb(\Pc_{\Nis,\sbe}(\Sm_X; \Sp))
\]
representing the non-$\Ab^1$-invariant motivic cohomology. In particular, we get power operation endomorphisms for the non-$\Ab^1$-invariant motivic cohomology. We record this in the following result. 
\end{rem}

\begin{cor}\label{cor:ms-extension}
Let $X$ be a noetherian $\Fb_p$-scheme of finite Krull dimension. There exist reduced mod-$p$ power operation endomorphisms acting on $H\Fb_p \in \MS_X$ that satisfy all the properties stated in Theorem~\ref{thm:ModPSteenrodSpectrumLevel}. \qed
\end{cor}

The analogue of Corollary~\ref{cor:ModPSteenrod} for these more general operations is also true, and is proven similarly to the discussion of Section~\ref{subsect:allPops}.  

The \'etale sheafification functor lets us construct various cycle class maps. 

\begin{defn}\label{def:mod-p} Let $S$ be a Dedekind domain or a field and $X$ a smooth $S$-scheme. The \emph{syntomic cycle class map in degree $n$} is the map
\begin{equation}\label{eq:mot-syn-n}
\mathrm{cyc}_{\syn}^n: H^{2n}_{\mot}(X; \Zb(n)) \cong \mathrm{CH}^n(X) \rightarrow H^{2n}_{\syn}(X; \Zb_p(n)),
\end{equation}
induced by the map in $\MS_S$
\[
H\Zb \rightarrow (L_{\et}H\Zb)^\comp_p \simeq H \Zb^\syn_p,
\]
where the equivalence is~\eqref{eq:let-syn}.

\end{defn}

We will be mostly interested in the case of $S = \Spec(k)$, a perfect field of characteristic $p > 0$. 

\begin{rem}[Comparison with cycle class maps in the literature]\label{compare-gros} Let $X$ be a smooth scheme over a perfect field $k$ of characteristic $p> 0$. Without the language of motivic spectra, we can define the cycle class map by considering the ``\'etale-sheaficiation'' map:
\[
R\Gamma_{\Zar}(X; W_r\Omega^n_{\log,X}) \rightarrow R\Gamma_{\et}(X; W_r\Omega^n_{\log,X}).
\] 
By the isomorphism of Gros-Suwa \cite[Th\'eor\`eme 4.13]{gros-suwa}, we then get a map
\[
\CH^n(X)/p^r \xrightarrow{\cong} H^{n}_{\Zar}(X; W_r\Omega^n_{\log,X}) \rightarrow H^{n}_{\et}(X; W_r\Omega^n_{\log,X}) \cong H^{2n}_{\syn}(X; \Zb/p^r(n)). 
\]
Thanks to the Geisser-Levine theorem which identifies mod-$p^r$-motivic cohomology of $X$ with the complex $R\Gamma_{\Zar}(X; W_r\Omega^n_{\log,X})[-j]$ \cite{geisser-levine}, the above map evidently agrees with the mod-$p^r$ reduction of Definition~\ref{def:mod-p}. 

Furthermore, we note that the Gros-Suwa theorem gives a quasi-isomorphism of complexes:
\[
\CH^n(X)/p^r \cong [\bigoplus_{x \in X^{(n-1)}} H^{n-1}_x(W_r\Omega^{n}_{\log, x}) \rightarrow \bigoplus_{x \in X^{(n)}} H^n_x(W_r\Omega^{n}_{\log, x}) \rightarrow 0]
\]
But, by the purity isomorphism of \cite[1.2.3]{gros-deux}, the complex on the right-hand-side is equivalent to 
\[
[\bigoplus_{x \in X^{(n-1)}} W_r\Omega^{1}_{\log, \kappa(x)} \rightarrow \bigoplus_{x \in X^{(n)}} W_r\Omega^{0}_{\log, \kappa(x)} \rightarrow 0]
\]
Unwinding the definition of \cite[(1.2.8)]{gros-deux}, we see that our construction agrees with Gros' construction of his cycle class map. 
\end{rem}

\begin{rem}[The crystalline cycle class map]  Let $k$ be a perfect field of characteristic $p > 0$ and $n \geq 0$. The \emph{crystalline cycle class map in degree $n$} is the composite
\[
\mathrm{cyc}_{\mathrm{crys}}^n:\CH^n(X) \rightarrow H^{2n}_{\syn}(X; \Zb_p(n)) \rightarrow H^{2n}_{\mathrm{crys}}(X/W(k));
\]
where the second map is defined by noting that syntomic cohomology can be defined via the limit diagram 
\[
 \Zb_p(n) \rightarrow \mathrm{Fil}^n_{\mathrm{Nyg}}R\Gamma_{\mathrm{crys}}(X/W(k)) \rightrightarrows R\Gamma_{\mathrm{crys}}(X/W(k));
\]
here the $\mathrm{Fil}^n_{\mathrm{Nyg}}R\Gamma_{\mathrm{crys}}(X/W(k))$ is the Nygaard filtration and the two maps are respectively the divided Frobenius $\phi_n$ and $\mathrm{can}$ the canonical ``inclusion'' of the Nygaard filtration; see \cite[\S 8]{BMS2} for details. 
\end{rem}

\begin{rem}[The rigid cycle class map]
In the language of motivic spectra, the crystalline cycle class map is induced by the map in $\MS_k$:
\[
H\Zb \rightarrow H\Zb_p^{\syn} \rightarrow HW(k)^{\mathrm{crys}},
\]
where $H\Zb_p^{\syn}$ (resp. $HW(k)^{\mathrm{crys}}$) is the motivic spectrum representing syntomic (resp. crystalline cohomology). From this view point, the cycle class map admits a refinement, namely the above composite factors as
\begin{equation}\label{eq:hz-dagger}
H\Zb \rightarrow (HW(k)^{\mathrm{crys}})^{\dagger},
\end{equation}
where $(-)^{\dagger}$ is the $\Ab^1$-colocalization functor of \cite[Definition 6.3]{AHI:atiyah}.

The motivic spectrum $(HW(k)^{\mathrm{crys}})^{\dagger}$ represents an integral refinement of Berthelot's rigid cohomology \cite{berthelot} in the sense that for any smooth $k$-scheme $X$, we have that \cite[Theorem 6.27(i)]{AHI:atiyah}:
\[
(HW(k)^{\mathrm{crys}})^{\dagger}(X)[\tfrac{1}{p}] \simeq R\Gamma_{\mathrm{rig}}(X)[\tfrac{1}{p}]. 
\]
On the other hand, if $U$ is a smooth $k$-scheme which admits a compactification $X$ whose boundary divisor $\partial X$ is a strict normal crossing divisor then \cite[Theorem 6.27(ii)]{AHI:atiyah} shows that we have an equivalence with Kato's logarithmic crystalline cohomology \cite{kato-logarithmic}:
\[
R\Gamma_{\mathrm{logcrys}}((X, \partial X)/W(k)) \simeq (HW(k)^{\mathrm{crys}})^{\dagger}(U).
\]
Let us define \emph{integral rigid cohomology} as
\[
R\Gamma_{\mathrm{crys}}(X/W(k))^{\dagger} := (HW(k)^{\mathrm{crys}})^{\dagger}(X) \qquad H^i_{\mathrm{crys}}(X/W(k))^{\dagger} := H^i(R\Gamma_{\mathrm{crys}}(X)^{\dagger}).
\]
The projective bundle formula\footnote{Let us be more precise, the first chern class on cyrstalline cohomology induces an equivalence....} in this case then shows that
\[
((HW(k)^{\mathrm{crys}})^{\dagger})^{2n,n}(X) \simeq H^{2n}_{\mathrm{crys}}(X/W(k))^{\dagger} \qquad \forall n \geq 0;
\]
therefore for $n \geq 0$, we have \emph{rigid cycle class map} as the map induced by~\eqref{eq:hz-dagger}:
\[
\mathrm{cyc}_{\mathrm{logcrys}}^n:  \CH^n(X) \rightarrow H^{2n}_{\mathrm{crys}}(X/W(k))^{\dagger}.
\]
Integral rigid cohomology will appear again when we consider the two (in fact, three) coniveau filtrations in the sequel. 
\end{rem}

\begin{defn}[Tate cycles]\label{def:tate-cycles} Let $k$ be a perfect field of characteristic $p > 0$ and $X$ a smooth, projective $k$-variety. For an extension $k'/k$ we write $G_{k'/k} := \mathrm{Gal}(\overline{k}/k)$; if $k'$ is an algebraic closure of $k$ then write $G_k:=G_{k'/k}$
\begin{enumerate}
\item The abelian group of \emph{integral $p$-adic Tate cycles of codimension $n$} on $X$ is given by 
\[
\mathrm{Tate}^n_p(X):= \mathrm{Im}(H^{2n}_{\mathrm{syn}}(X, \Zb_p(n)) \rightarrow H^{2n}_{\mathrm{crys}}(X/W(k))))
\]
\item Let $k'/k$ be an extension of $k$ (possibly infinite, like an algebraic closure). We define:
\[
\mathrm{Tate}^n_p(X_{k'})^{G_{k'/k}} =: \mathrm{Im}(H^{2n}_{\mathrm{syn}}(X_{k'}, \Zb_p(n))^{G_{k'/k}} \rightarrow H^{2n}_{\mathrm{crys}}(X_{k'}/W(k))^{G_{k'/k}})
\]
\end{enumerate}
\end{defn}

\begin{rem}[The case of finite fields] Let $k$ be a finite field and $\overline{k}$ be an algebraic closure.  Then the Hochschild--Serre spectral sequence for $\widehat{\Zb} \xrightarrow{\cong} G_{k}$:
\[
H^r(\widehat{\Zb}, H_{\mathrm{syn}}^s(X_{\overline{k}}, \Zb_p(j))) \Rightarrow H^{r+s}_{\mathrm{syn}}(X, \Zb_p(j)) \qquad \forall j \geq 0,
\]
degenerates (since $\widehat{\Zb}$ has $p$-cohomological dimension one) to produce a short exact sequence
\[
0 \rightarrow H^1(\widehat{\Zb}, H^{i-1}_{\mathrm{syn}}(X_{\overline{k}}, \Zb_p(j))) \rightarrow H_{\mathrm{syn}}^i(X, \Zb_p(j)) \rightarrow H_{\mathrm{syn}}^i(X_{\overline{k}}, \Zb_p(j))^{\widehat{\Zb}} \rightarrow 0
\]
In particular, unwinding Definition~\ref{def:tate-cycles}, we get a surjection
\begin{equation}\label{eq:surject}
\mathrm{Tate}^j_p(X) \twoheadrightarrow   \mathrm{Tate}^j_p(X_{\overline{k}})^{G_{k}}
\end{equation}
\end{rem}


%

\begin{rem}[Relationship with rationalized Tate cycles]\label{rem:tatecycles} $k$ is either a finite or algebraically closed field, then we have isomorphisms

\begin{equation}\label{eq:rational-tate}
H^{2n}_{\syn}(X; \Zb_p(n))\otimes{\Qb} \xrightarrow{\cong} \mathrm{Tate}^n_p(X)\otimes{\Qb} \cong (H^{2n}_{\mathrm{crys}}(X/W(k))[\tfrac{1}{p}])^{\phi = p^n};
\end{equation}
by \cite[Proposition 3.3(ii)]{morrow-variational}. This is the sense in which we are justified in calling these cycles to be an integral version of Tate cycles since the group $(H^{2n}_{\mathrm{crys}}(X/W(k))[\tfrac{1}{p}])^{\phi = p^n}$ is what is usually called the crystalline version of Tate cycles in the literature. 
\end{rem}

The crystalline cycle class map allows us to formulate the integral crystalline Tate conjecture (more rightly called, condition). Just as discussed by Schoen in \cite{schoen} there are a couple of different variants of this statement that one must be somewhat careful about; we stick to formulating these statements over finite fields in this paper. The integral crystalline Tate conjecture should be thought of as the $p$-adic analog of the integral $\ell$-adic Tate conjecture. 

For the rest of the section, by the cycle class map we mean the composition of the first two maps in the sequence
\begin{equation}
\CH^n(X) \otimes \Zb_p \to H^{2n}_\syn(X,\Zb_p(n)) \tto \mathrm{Tate}^n_p(X) \subset H^{2n}_\mathrm{crys}(X/W(k)),
\end{equation}
and it is denoted by $\mathrm{cyc}^n_X$.

\begin{defn}\label{def:conditions-tate} Let $k$ be a finite field and let $X$ be a smooth, projective, geometrically connected $k$-variety. For any $j \geq 0$, we define the following conditions: 

\begin{enumerate}
\item[$(A^n_{X_k})$] The cycle class map
\[
\mathrm{cyc}_{X}^n: \mathrm{CH}^n(X) \otimes \Zb_p \rightarrow \mathrm{Tate}^n_p(X)
\]
is surjective.
\item[$(B^n_{X_k})$] The cycle class map composed with~\eqref{eq:surject}
\[
\overline{\mathrm{cyc}}_{X}^n: \mathrm{CH}^n(X) \otimes \Zb_p \rightarrow \mathrm{Tate}_p^n(X_{\overline{k}})^{G_k}
\] 
is surjective.
\item[$(C^n_{X_k})$] The colimit of the system of maps
\[
\overline{\mathrm{cyc}}_{X_{k'}}^n: \mathrm{CH}^n(X_{k'})\otimes \Zb_p \rightarrow \mathrm{Tate}^n_p(X_{\overline{k}})^{G_{k'}} \qquad \text{$k'/k$ finite extension},
\]
which defines the \emph{continuous cycle class map}
\[
\mathrm{cyc}_{\mathrm{cts},X}^n: \mathrm{CH}^n(X_{\overline{k}}) \otimes \Zb_p \rightarrow \mathrm{Im}(\colim_{k \subset k' \subset \overline{k}} \mathrm{Tate}_p^n(X)^{G_{k'}} \rightarrow H^{2n}_\mathrm{crys}(X_{\overline{k}}/W(k)))
\]
is surjective.
\end{enumerate}

\end{defn}

%
%

\begin{rem}[Relationship with the crystalline Tate conjecture] First we note that, after tensoring with $\Qb_p$, the statements $(A^n_X), (B^n_X)$ and $(C^n_X)$ are equivalent by a standard transfer argument. These are in turn, thanks to the isomorphisms of Eq.~\eqref{eq:rational-tate}, equivalent to the \emph{crystalline Tate conjecture}, which is the statement that:
\begin{enumerate}
\item[($\dagger^n_X$)] The cycle class map $\CH^n(X) \otimes \Qb_p \rightarrow H^{2n}_{\mathrm{crys}}(X/W(k))[\tfrac{1}{p}])^{\phi = p^n}$ is surjective. 
\end{enumerate}

The reader should also note the following relationship between the crystalline Tate conjecture and the usual $\ell$-adic Tate conjecture: a theorem of Morrow (also proved by de Jong) \cite[Proposition 4.1, Theorem 4.3]{morrow-variational} shows that $\dagger^1_X$ is equivalent to the usual Tate conjecture in $\ell$-adic cohomology for divisors, and furthermore that the veracity of $\dagger^1_X$ for surfaces implies it for all other smooth projective, geometrically connected varieties. 
\end{rem}

\begin{rem}[$(A^n_X) \Rightarrow (B^n_X)$] Thanks to the surjection~\eqref{eq:surject}, evidently $(A^n_X) \Rightarrow (B_n^X)$. However, we do not expect the converse to hold as in the case of the $\ell$-adic integral Tate conjecture thanks to the examples of Scavia--Suzuki \cite[Theorem 1.4]{scavia-suzuki-non-alg}.  
\end{rem}

\begin{rem}[$(B^n_{X_{k'}}) \Rightarrow (C^n_X), k' \gg k$]  We note that if $B^n_{X_{k'}}$ is true for all finite extensions $k'$ of $k$ , then $C^n_{X}$ is true. In fact, we only need the veracity of $B^n_{X_{k'}}$ for all $k'$ with enough elements. The statement $C^n_{X}$ is the most optimistic one. This story is similar to the $\ell$-adic, integral counterpart of the statement as discussed by Schoen in \cite{schoen}; the target in $C^n_X$ is the $p$-adic analog of the group:
\[
\bigcup_{U \leq G, \text{open}} H^{2n}(X_{\overline{k}}, \Zb_{\ell}(n))^{U},
\]
which is a target of the cycle class map studied in op. cit.. We also make a remark that the formation of syntomic cohomology does not preserve filtered colimits of schemes even along affine transition maps. Therefore, there it is not true that the map $\colim_{k'} \mathrm{Tate}_p^n(X_{k'})  \rightarrow \mathrm{Tate}_p^n(X_{\overline{k}})$ is an isomorphism. Nonetheless by Remark~\ref{rem:tatecycles}, the map $\colim_{k'} \mathrm{Tate}_p^n(X_{k'}) \rightarrow (H^{2i}_{\mathrm{crys}}(X_{\overline{k}}/W(k))[\tfrac{1}{p}])^{\phi = p^i}$ is rationally surjective. 
\end{rem}

\begin{rem}[Variant with the syntomic cycle class map]\label{rem:syn-version} We can formulate the following version of $A^n_X$:
\begin{enumerate}
\item[($A^n_{X,\mathrm{syn}}$)] The syntomic cycle class map $\mathrm{CH}^n(X) \xrightarrow{\mathrm{cyc}^n_{\mathrm{syn}}} H^{2n}_{\mathrm{syn}}(X, \Zb_p(n))$ is surjective.
\end{enumerate}
Evidently, $A^n_{X,\mathrm{syn}} \Rightarrow A^n_{X}$. However, $A^n_{X}$ can still hold without $A^n_{X,\mathrm{syn}}$; non-algebraizable classes in syntomic cohomology could potentially vanish in crystalline cohomology. While we do not have an example of this phenomenon, we still propose $A^n_X$ as a better variant of the integral Tate conjecture which has a better chance of being true (and therefore harder to give counterexamples to) than $A^n_{X,\mathrm{syn}}$, though we anticipate that whenever $A^n_X$ holds it is only because one can prove $A^n_{X,\mathrm{syn}}$. Similar, ``syntomic versions'' of $B^n_X$ and $C^n_X$ can also be formulated and left to the reader. 
\end{rem}

\begin{ex}[$A^1_X$ for smooth projective curves] \label{exam:sm-proj} The first interesting case of the integral Tate conjecture is for smooth projective, geometrically connected curves over a finite field, $(A^1_X)$. In this case the conjecture asks that the map
\[
\mathrm{Pic}(X) \otimes \Zb_p \rightarrow \mathrm{Tate}_p^1(X)
\]
is surjective. We claim a stronger statement, namely that
\[
\mathrm{Pic}(X) \otimes \Zb_p \simeq H^2_{\syn}(X; \Zb_p(1))
\]
which verifies $A^1_X$. Indeed (for all qcqs scheme $X$) weight one syntomic cohomology is calculated as
\[
R\Gamma_{\et}(X; \Gb_m)^\comp_p[-1] \xrightarrow{c_1^{\syn},\simeq} \Zb_p(1)(X); 
\]
Therefore, we get a diagram:
\[
\begin{tikzcd}
& \mathrm{Pic}(X)\ar{d}  \ar{r}{\mathrm{cyc}^1_{\syn}} &  H^2_{\syn}(X; \Zb_p(1)) \ar{d} & \\
0 \ar{r} &  \mathrm{Pic}(X)/p^r \ar{r} \ar{d} & H^2_{\syn}(X; \Zb/p^r(1)) \ar{r} \ar{d}  & H^2_{\et}(X; \Gb_m)[p^r] \ar{r} & 0\\
& 0 & 0 &.
\end{tikzcd}
\]
As explained in \cite[Example 4.5.7]{bhatt-gauges}, $H^2_{\et}(X; \Gb_m)[p^r] = 0$. On the other hand, $\mathrm{Pic}(X)$ of a finitely generated group over a finite field is finitely generated therefore $\mathrm{Pic}(X) \otimes \Zb_p \cong \mathrm{Pic}(X)^\comp_p$. Hence, to prove the claim, it suffices to prove that $ H^2_{\syn}(X; \Zb_p(1)) \cong \lim_r H^2_{\syn}(X; \Zb/p^r(1))$, in other words that $\lim^1 H^1_{\syn}(X; \Zb/p^r(1))$ is zero. In fact, we note that $H^i_{\syn}(X; \Zb/p^r(j))$ are finite groups for all smooth projective scheme $X$: indeed via Illusie's exact sequence \cite[Th\'{e}orem\`{e} I.5.7.2]{illusie-drw} on $X_{\et}$:
\[
0 \rightarrow \Zb/p^r(j)[-j] \rightarrow W_r\Omega^j_X \xrightarrow{1-F} W_r\Omega^j_X \rightarrow 0,
\]
the long exact sequence in cohomology sandwiches syntomic cohomology in between the coherent cohomology of $W_r\Omega^j_X$. Since $X$ is proper, the latter groups are finitely generated $W_r(k)$-modules, hence finite. 
\end{ex}

\begin{ex}[$A_X^{\dim(X)}$] Let $X$ be a smooth projective, geometrically connected variety over a finite field of dimension $d$; we sketch the veracity of $A^d_X$. This result can be deduced using Kato--Saito's unramified class field theory for smooth projective varieties over a finite field \cite{kato-saito}; we also note that Gros \cite[Th\'eor\`eme 3.2.0]{gros-deux} has a direct proof of this result using cohomological methods which are closed to this paper. This result also generalizes Example~\ref{exam:sm-proj}. To begin with, we have Milne's duality \cite{milne-values}:
\[
\Zb_p(j)^{\syn}(X) \simeq \Zb_p(j-d)^{\syn}(X)^{\vee}[-2d-1];
\]
whence we have isomorphisms:
\[
H^{2d}_{\syn}(X;\Zb/p^r(d)) \cong H^1_{\et}(X; \Zb/p^r)^{\vee} \cong \pi_1^{\mathrm{ab}}(X)/p^r.
\]
The cycle class map fits into a diagram 
\[
\begin{tikzcd}
\mathrm{CH}_0(X) \otimes \Zb_p \ar{d}  \ar{r}{\mathrm{cyc}^d_{\syn}} &  H^{2d}_{\syn}(X; \Zb_p(d)) \ar{d} \\ 
\lim \mathrm{CH}_0(X)/p^r \ar{r} & \lim \pi_1^{\mathrm{ab}}(X)/p^r  \\
\end{tikzcd}
\]
Since the $\lim^1$ of $H^{2d-1}_{\syn}(X; \Zb/p^r(d))$ vanishes by the argument in Example~\ref{exam:sm-proj} the right vertical map is an isomorphism. On the other hand \cite[\S 10]{kato-saito} shows that $\mathrm{CH}_0(X)$ is finitely generated, hence the left vertical map is an isomorphism. Finally, we note that $\mathrm{CH}_0(X)/p^r \rightarrow  \pi_1^{\mathrm{ab}}(X)/p^r $ agrees with the reciprocity map from class field theory; this was verified by Gros for his cycle class map in \cite[Proof of Th\'eor\`eme 2.2.2]{gros-deux} and we note that it agrees with ours in Remark~\ref{compare-gros}. Hence we can appeal to  \cite[Theorem 1]{kato-saito} which verifies that we have an isomorphism mod-$p^r$ for all $r$ between $\mathrm{CH}_0(X)/p^r \xrightarrow{\cong}  \pi_1^{\mathrm{ab}}(X)/p^r $, and hence the bottom horizontal map is also an isomorphism. 
\end{ex}

Having verified some general cases of $A^n_X$, we now turn to counterexamples. The mechanism by which this happens is via Steenrod operations as made precise by the following.

\begin{lem}\label{lem:CycleObs}
Any syntomic cohomology class $H_\syn^{2n}(X ; \Fb_p(n))$ that is in the image of the mod-$p$ cycle class map vanishes under the action of $Q_i$ for all $i \geq 1$.
\end{lem}
\begin{proof}
The cohomology operations on mod-$p$ motivic and syntomic cohomologies are compatible, and the vanishing is true for motivic cohomology for degree reasons.
\end{proof}

The first systematic study of this conjecture and its counterexamples was first produced carefully in \cite{ct-szamuely}. Other counterexamples to the $\ell$-adic integral Tate conjecture have been studied by various authors including (but not limited to): \cite{pirutka-yagita,antieau-tate,quick-etale,benoist}. Our counterexample is in codimension two but is a little different from the ones produced in \cite{atiyah-hirzebruch, ct-szamuely} and it involves the product of an approximation to classifying space with an ordinary elliptic curve. To this end, we review the calculation of mod-$p$ syntomic cohomology of elliptic curves over $\overline \Fb_p$.

\begin{ex}[Mod-$p$ syntomic cohomology of elliptic curves over $\overline \Fb_p$]\label{ex:HsynOfE}
Let $E$ be an elliptic curve over $\overline \Fb_p$. We will compute the syntomic cohomology groups
\[
H_\syn^i(E; \Fb_p(j)) = \colim_{n \to \infty} H_\syn^i(E_n; \Fb_p(j))
\]
where $E_n$ is a model of $E$ over $\Fb_{p^n}$, which exists for all $n \gg 0$, and where the equality follows from the fact that mod-$p$ syntomic cohomology sends cofiltered limits of qcqs schemes to filtered colimits \cite[Corollary~8.4.11]{bhatt-lurie:apc}.

We begin by calculating the syntomic cohomology of the elliptic curves $E_n$. The syntomic cohomology is concentrated in weights $j=0,1$, and the syntomic cohomology can be determined from the exact sequences
\begin{equation}\label{eq:Syn0ArtinSchreier}
\begin{tikzcd}
0 \arrow[r] & H^0_\syn(E_n; \Fb_p(0)) \arrow[r] & H^0(E_n; \Oc_{E_n}) \arrow[r]{}{1-F} & H^0(E_n; \Oc_{E_n}) \arrow[lld]{}{\delta_1} \\
&  H^1_\syn(E_n; \Fb_p(0)) \arrow[r] & H^1(E_n; \Oc_{E_n}) \arrow[r]{}{1-F}  & H^1(E_n; \Oc_{E_n}) \arrow[lld]{}{\delta_2} \\
&  H^2_\syn(E_n; \Fb_p(0)) \arrow[r] &0,&
\end{tikzcd}
\end{equation}
and
\begin{equation}\label{eq:Syn1ArtinSchreier}
\begin{tikzcd}
0 \arrow[r] & H^1_\syn(E_n; \Fb_p(1)) \arrow[r] & H^0(E_n; \Omega^1_{E_n}) \arrow[r]{}{1-C} & H^0(E_n; \Omega^1_{E_n}) \arrow[lld]{}{\delta_3} \\
&  H^2_\syn(E_n; \Fb_p(1)) \arrow[r] & H^1(E_n; \Omega^1_{E_n}) \arrow[r]{}{1-C}  & H^1(E_n; \Omega^1_{E_n}) \arrow[lld]{}{\delta_4} \\
&  H^3_\syn(E_n; \Fb_p(1)) \arrow[r] &0.&
\end{tikzcd}
\end{equation} 
As $H^0_\syn(E_n; \Fb_p(0)) \cong \Fb_p$, the image of $\delta_1$ is one-dimensional. Let us denote the generator of the image by $\epsilon_n$. Similarly, by Milne's duality for syntomic cohomology over finite fields \cite[Theorem~1.9]{milne:1976}, $H^3_\syn(E_n; \Fb_p(1)) \cong \Fb_p\{\tau_n\}$, and therefore there exists a class $y \in H^2_\syn(E_n, \Fb_p(1))$ with a non-zero image in $H^1(E_n;\Omega^1_{E_n})$. We do not decorate $y$ with a subscript $n$ because if $m > n$, then $y$ pulls back to a class in $H^2_\syn(E_m, \Fb_p(1))$ that also has a non-zero image in $H^1(E_m;\Omega^1_{E_m})$. Thus, it suffices to find $y$ for one $n$, and from there it can be pulled back to $E_m$ and $E$. The same convention applies to the classes $x$ and $z$ considered below.
 
To analyze the rest of the cohomology, we need to consider two cases:
\begin{enumerate}
\item If $E$ is \textit{supersingular}, i.e., the action of $F$ on $H^1(E; \Oc_E)$ is zero, then $1-F$ acts as the identity on $H^1(E_n; \Oc_{E_n})$. Combining this observation with Milne's duality result, we observe that the syntomic cohomology of $E_n$ is generated by the classes $1, \epsilon_n, y,$ and $\tau_n$.
\item If $E$ is \textit{ordinary}, i.e., the action of $F$ on $H^1(E; \Oc_E)$ is non-trivial, then for a generator $t$ of $H^1(E_n;\Oc_{E_n}) \cong \Fb_{p^n}$, there exists $\lambda \in \Fb_{p^n}^\times$ such that, for all $a \in \Fb_{p^n}$ 
\begin{equation}
F(at) = a^p \lambda t.
\end{equation}
For $n \gg 0$ there exists $a$ such that $a^p \lambda = a$, and therefore we obtain a class $x \in H^1(E_n; \Fb_p(0))$ with a non-zero image in $H^1(E_n; \Oc_{E_n})$. Moreover, the image of $\delta_2$ is generated by a class $\eta_n \in H^2(E_n; \Fb_p(0))$. 

By Milne's duality result, we observe that the operator $1-C$ on $H^0(E_n;\Omega^1_{E_n})$ has a one-dimensional kernel and cokernel. Let $z \in H^1(E_n; \Fb_p(1))$ be the generator of the kernel, and let $\rho_n$ be the generator of the image of $\delta_3$.

We have observed that the syntomic cohomology of the ordinary elliptic curve $E_n$ has the four classes $x, \eta_n, z$, and $\rho_n$ in addition to the classes that exist also in the supersingular case.  
\end{enumerate}
To finish our analysis, we observe that the connecting maps $\delta_i$ vanish in the colimit because the operators $1-F$ and $1-C$ induce surjective maps on cohomology in the colimit. This can be seen by a computation using the Frobenius linearity of $F$ and $C$ on the respective coherent cohomology groups, and the fact that these groups are one-dimensional over $\overline \Fb_p$. Thus, the generators of the syntomic cohomology of $E$, i.e., the classes that survive the colimit, are
\begin{equation}\label{eq:HsynOfE}
H^\star_\syn(E;\Fb_p(\star)) = 
\begin{cases}
\Fb_p\{1,y\} & \text{if $E$ is supersingular;} \\
\Fb_p\{1,x,y,z\} & \text{if $E$ is ordinary.} 
\end{cases}
\end{equation}
In either case, the kernel of the map from the syntomic cohomology to de Rham cohomology is trivial because the above classes map non-trivially to Hodge cohomology and the Hodge-de Rham spectral sequence collapses for curves. We will use this observation in what follows.
\end{ex}

\begin{thm}\label{thm:counterexample}
There exists a smooth, projective $\Fb_p$-variety $X$ such that $(C^2_X)$ does not hold.
\end{thm}

\begin{proof} 
Let $E_1$ be an ordinary elliptic curve over $\Fb_p$. Using the notation of Example~\ref{ex:HsynOfE}, our counterexample is the image $\alpha'$ of the class
\[
\alpha = \tilde\beta(u_1 u_2 x) \in H^4_\syn(X_n;\Zb_p(2))
\]
in $\colim_k H^4_{\mathrm{crys}}(X_k/\Zb_p)$, where $\tilde \beta$ is the integral Bockstein, and $X_n$ is an approximation of $X'_n:={\rm{B}} \mu_p \times {\rm{B}} \mu_p \times E_n$ where $n$ is large enough so that the class $x \in H^1_{\syn}(E_k; \Fb_p(0))$ exists. 

To check that $\alpha'$ is not in the image of the cycle class map, it suffices to check this after reducing modulo $p$. Furthermore, in Example~\ref{ex:HsynOfE} we saw that the map from the mod-$p$ syntomic cohomology to the de Rham cohomology of $E_n$ is injective in the colimit. The same is true for $X'_n$, because both $H^\star_\syn(X'_n; \Fb_p(\star))$ and $H^\star_\dR(X'_n)$ are free modules over the corresponding cohomology of $E_n$ with basis given by $u_1^a u_2^b v_1^c v_2^d$, where $a,b \in \{0,1\}$ and $c,d \in \Nb$, and where $u_i$ and $v_i$ are the images of the similarly named classes in the motivic cohomology of the two copies of ${\rm B}\mu_p$ \cite[Proposition~6.4]{AHI}. Thus, it suffices to check that $\overline{\alpha} \in H^4_\syn(X';\Fb_p(2))$ is not in the image of the cycle class map.

By Lemma~\ref{lem:CycleObs}, we are reduced to showing that $\Q_1(\overline \alpha) \not = 0$. As the mod-$p$ reduction of the integral Bockstein is the usual Bockstein $\Q_0$, we compute that
\[
\overline \alpha = v_1 u_2 x - u_1 v_2 x + u_1 u_2 \Q_0(x)\footnote{Though we do not need it, we note that $\beta(x)$ is also nonzero because $H^1_{\syn}(X; \Zb_p)$ must be zero for any smooth projective variety over a finite field by the Weil conjectures. We thank Alexander Petrov for pointing this out.}.
\]
Thus,
\begin{align*}
\Q_1(\overline \alpha) &= \Q_1( v_1u_2x) - \Q_1(u_1v_2x) + \Q_1(u_1u_2\Q_0(x)) \\
&=  v_1v_2^px - v_1^pv_2x + v_1^pu_2\Q_0(x) - u_1v_2^p\Q_0(x) \\
&\not = 0
\end{align*}
Here we have used Lemmas~\ref{lem:milnor} (that $\Q_1$ is a derivation) and \ref{lem:Q_n-bmup} (the action of $\Q_1$ on $u_i$ and $v_i$). We also note that we have used that $\Q_1$ of $x$ and $\Q_0(x)$ is zero for degree reasons, see Eq.~\eqref{eq:HsynOfE}.

Finally, we can use Theorem~\ref{thm:approximate} to produce a smooth projective variety $Y$ over $\Fb_p$ of  dimension $2p$ equipped with a map $Y \rightarrow B\mu_p$ such that the map on syntomic cohomology is injective up to cohomological degrees $\leq 2p$ at all weights. The numerology ensures that the map is injective in the bidegree of the classes $u$ and $Q_1(u)$ in syntomic cohomology. So letting $X_n := Y \times Y \times E_n$ does the job.
\end{proof}

\begin{rem}[Relation to the counterexample from \cite{atiyah-hirzebruch} and \cite{ct-szamuely}] Let us discuss the relationship between the counterexample from the previous theorem and the one constructed by Atiyah and Hirzebruch. The analog of the Atiyah-Hirzebruch class in syntomic cohomology is the element
\[
\tilde \beta(u_1u_2u_3) \in H_{\syn}^4(\rm{B} \mu_p \times \rm{B} \mu_p \times \rm{B} \mu_p; \Zb_p(3)).
\]
By the calculation in Lemma~\ref{prop:qi} while this class is nonzero, it is not a crystalline Tate class since it is of degree $(4,3)$ and has no chance to be even in the image of the cycle class map (in the language of this paper, it is not a crystalline Tate class). 

The situation here is quite different from the class in singular and $\ell$-adic \'etale cohomology. In the former: we know that any torsion class must be Hodge. In the latter: one is allowed to implicitly considers a $\tau$-twist of the class to get it in the right degree: $\beta(\tau^{-1}u_1u_2u_3) \in H_{\syn}^4((\rm{B} \mu_p \times \rm{B} \mu_p \times \rm{B} \mu_p)_{\overline{k}}; \Zb_{\ell}(2))$ and then one shows that it is actual fixed by the Galois group; this shows that $\beta(\tau^{-1}u_1u_2u_3)$ is an $\ell$-adic Tate class and could be in the image of the cycle class map. It was proved in \cite{ct-szamuely} that this is not the case using the Milnor operations on \'etale cohomology. We note that one of the salient differences between the $\ell$-adic and $p$-adic situation is that $\tau$ is zero in the latter.

Another class which we do know has a nontrivial action of $Q_1$ is the class
\[
\beta(\epsilon u_1u_2) \in H_{\syn}^4(\rm{B} \mu_p \times \rm{B} \mu_p \times \rm{B} \mu_p; \Zb_p(2)),
\]
where $\epsilon \in H^1_{\syn}(\Spec(\Fb_p);\Fb_p(0))$. This follows immediately from Proposition~\ref{prop:qi}; hence we conclude that this class cannot come from a cycle. While its image in crystalline cohomology is in fact a Tate class, it is zero on crystalline cohomology since $\epsilon$ is zero on de Rham cohomology because there is no $H^1_{\mathrm{dR}}$ of the base field. 
\end{rem}

\begin{rem} At the prime $p = 2$, the counterexample of Atiyah--Hirzebruch \cite{atiyah-hirzebruch} is of dimension $7$. Ours is slightly of lower dimension since it is of dimension $2(2) +1=5$. More generally, for any prime $p$ if it is of dimension $2p+1$. We expect something like \cite[Theorem 1]{soule-voisin} to also hold in characteristic $p > 0$ where the minimal dimension counterexample detected via Steenrod operations is at least $p$. It would be interesting to find counterexamples of lower dimension. 

Another interesting direction would be to find non-torsion counterexamples, following the work of Pirutka-Yagita \cite{pirutka-yagita} for the $\ell$-adic integral Tate conjecture.
\end{rem}

\subsection{The Wu formula for motivic and syntomic cohomology}\label{sec:riemann-roch}

In the final section, we will establish a Wu formula for the power operations in both the syntomic and motivic contexts. Our proof is quite robust and follows immediately from a general Riemann--Roch type statements in $\MS$, which is of independent interest and is established in Theorem~\ref{thm:RRInMS}. Crucial to this method of proof is the fact that our operations are defined on the level of motivic spectra. We will first prove a general Riemann--Roch formulas for motivic spectra, showing that a map of homotopy commutative orientable ring spectra commutes with Gysin pushforwards up to twist by a Todd class. This follows the outline of \cite{deglise-orientations} in $\Ab^1$-homotopy theory. 

We begin with a review of Gysin maps in the non-$\Ab^1$-invariant context; we have already seen this in the work of Khan in the $\Ab^1$-invariant context that we used in \S\ref{ssec:nonsmoothable}. Our discussion is based on upcoming work by L. Tang \cite{tang-gysin}; though the reader interested in only statements for $\Ab^1$-invariant motivic cohomology is free to use Khan's formalism. The end-product of this work gives us Gysin maps along quasi-smooth closed immersions for those cohomology theories that are defined on derived schemes and satisfy derived blowup excision; often we only need it in the more restricted setting where the cohomology theory is defined on all schemes that are smooth over a base.

 If $E \in \MS^?_S$ and is oriented, then we have the \emph{Thom isomorphism}
\begin{equation}\label{eq:thom-iso}
t(\Ec): E(\Tb^{\otimes r} \otimes X) \xrightarrow{\simeq} E(\Th_X(\Ec));
\end{equation}
where $\Ec$ is a locally free rank $r$ sheaf on $X \in \Sm^?_S$. If $E \in \CAlg(\h\MS^?_S)$, then $t(\Ec)$ is furthermore implemented by multiplication with the \emph{Thom class} $t(\Ec) \in E^{2r, r}(\Th_X(\Ec))$, which is the image of $1 \in E^{0,0}(X)$ under the isomorphism~\eqref{eq:thom-iso}.

For $i \colon Z \hook X$ a quasi-smooth closed embedding of virtual codimension $r$ in $\Sm^?_S$, L. Tang has constructed the \emph{Gysin map}\footnote{In the $\Ab^1$-invariant setting and for $Z \hook X$ smooth over $S$, this is implemented as the map in motivic spaces given by $M_S(X) \rightarrow \tfrac{M_S(X)}{M_S(X \setminus Z)} \xleftarrow{\simeq} \Th_Z(\Nc_{Z/X})$ where the equivalence is the relative purity isomorphism of Morel-Voevodsky \cite{morel:1999}.}
\begin{equation}\label{eq:GysDef}
\gys_i \colon M_S(X) \to \Th_Z(\Nc_{Z/X})
\end{equation}
in $\MS^?_S$. If $E \in \CAlg(\h\MS^?_S)$ is oriented, then we can compose this with the Thom isomorphism which takes the form
\[
t(\Nc_E) \cdot \dash \colon E(\Tb^{\otimes n} \otimes M(Z)) \xto{\sim} E(\Th_Z(\Nc_{Z/X})),
\]
to define the \emph{Gysin pushforward}
\[
i_! \colon E(\Tb^{\otimes n} \otimes M(Z)) \simeq E(-n)[-2n](Z) \to E(X).
\]
Concretely, this incudes pushforward maps $i_! \colon E^{\star, \star}(Z) \to E^{\star + 2r, \star + r}(X)$ on bigraded $E$-cohomology groups, where $r$ is the codimension. 

\begin{rem}[Pullback along Gysin vs Gysin pushforward] \label{rem:gysin-thom} Let $E$ be a homotopy commutative ring spectrum in $\MS_X$. The value of $E$ on a Thom space is given as the fibre of the map of a multiplicative map $E(\Pb_X(\Ec \oplus \Oc)) \rightarrow E(\Pb_X(\Ec))$; hence the object $E(\Th_Z(\Nc_{Z/X}))$ is naturally a homotopy commutative, nonunital ring in spectra. The pullback along $\gys_i$ thus induces a map of non-unital rings:
\[
E^{\star,\star}(\Th_Z(\Nc_{Z/X})) \rightarrow E^{\star,\star}(X)
\] that is independent of the choice of orientation. On the other hand the Gysin pushforward map $i_!$ is orientation-dependent and is only $E^{\star,\star}(X)$-linear. This multiplicative property of the former was already used in the proof of Theorem~\ref{thm:non-qs}. 
\end{rem}

\begin{rem}[Pullback along Gysin vs trace of the purity transformation]
Any absolute $\Ab^1$-invariant motivic spectrum $E$ enhances into an object of $\MS^\dbe_S$ for any derived scheme $S$ \cite[Appendix~B]{annala-shin}. For such a $E$, we expect $\gys^*_i$ to coincide with the map $\mathrm{tr}_{i*}$ from Eq.~\eqref{eq:ThomGysin}. We do not use this assertion for what follows. 
\end{rem}

Next, assume that we have a projective quasi-smooth morphism $f: X \rightarrow Y$; by definition we may choose a factorization 
\[
X \xrightarrow{i} \Pb^n_Y \xrightarrow{p} Y
\]
where $i$ is a quasi-smooth closed immersion and $p$ is the projection map. The first author and Shin has defined, for $E \in \CAlg(\h\MS^?_S)$ a pushforward map
\[
p_!: E(\Pb^n_X) \rightarrow E(\Tb^{\otimes n} \otimes M_S(X)) \simeq E(X)(-n)[-2n]
\]
given by the formula in \cite[Definition 2.7]{annala-shin}. Setting 
\[
f_!:=p_! \circ i_!,
\]
they also proved in \cite[Lemma 2.10]{annala-shin} that $f_!$ does not depend on the choice of $i$ and $p$. The following are the properties of the Gysin pushforwards established in \cite{tang-gysin} and \cite[Theorem 2.14]{annala-shin}  that will be needed in this paper.
\begin{thm}
Let $E \in \CAlg(\h\MS^?_S)$ be oriented. Then, the Gysin pushforwards satisfy the following properties:
\begin{enumerate}
    \item \emph{Functoriality:} $\Id_! = \Id$ and if $f \colon X \hook Y$ and $g \colon Y \hook Z$ are projective quasi-smooth morphisms of constant  virtual codimension, then $(g \circ f)_! = g_! \circ f_!$.
    \item \emph{Base change:} If
    \[
    \begin{tikzcd}
        Y' \arrow[d]{}{p'} \arrow[r]{}{f'} & X' \arrow[d]{}{p} \\
        Y \arrow[r]{}{f} & X
    \end{tikzcd}
    \]
    is a commutative square in $\Sm_S^?$ that is also a cartesian square of derived schemes, and if $f$ is projective quasi-smooth, then $p^* \circ f_! = f'_! \circ p'^*$.
    
    \item \emph{Projection formula:} The formula $f_!(f^*(\alpha) \cdot \beta) = \alpha \cdot f_!(\beta)$ holds for all $\alpha \in E^{\star,\star}(X)$ and $\beta \in E^{\star,\star}(Z)$.
    
    \item \emph{Naturality:} If $\varphi: E \to F$ is an orientation-preserving map of oriented theories in $\CAlg(\h\MS^?_S)$, then the squares
    \[
        \begin{tikzcd}
             E^{a,b}(Z) \arrow[r]{}{f_!} \arrow[d]{}{\varphi} & E^{a+2r,b+r}(X) \arrow[d]{}{\varphi} \\
             F^{a,b}(Z) \arrow[r]{}{f_!} & F^{a+2r,b+r}(X)
        \end{tikzcd}
    \]
    commute.
\end{enumerate}
\end{thm}

With now proceed to prove the Riemann-Roch formula. Let 
\begin{equation}\label{eq:phiMap}
\varphi \colon E \to F
\end{equation}
be a map in $\CAlg(\h\MS^?_S)$, and let $e$ and $f$ be orientations of $E$, and $F$ respectively. Evaluating $\varphi$ on $\Pb^\infty_{S,+}$ induces a map of bigraded rings
\[
E^{\star,\star}(S)[[e]] \to F^{\star,\star}(S)[[f]].
\]
We denote by $\Psi(f)$ the image of $e$ under the above map. As $\Psi(f)$ is an orientation of $F$, it has no constant coefficient, and therefore $\Psi(f) = \Td^{-1}_\varphi(f) \cdot f$ for a homogeneous power series $\Td_\varphi^{-1}$ of degree 0. This allows us to define the \emph{inverse Todd class} of virtual vector bundles.

\begin{lem}
For every $\xi \in K_0(X)$, there exists a unique class
\[
\Td^{-1}_\varphi(\xi) \in F^{0,0}(X)
\]
that is invertible, natural in pullbacks, and satisfies
\begin{enumerate}
\item if $\Ls$ is a line bundle, $\Td^{-1}_\varphi([\Ls]) = \Td_\varphi^{-1}(c_1(\Ls))$ as defined above;
\item $\Td^{-1}_\varphi(\xi + \eta) = \Td^{-1}_\varphi(\xi) \Td^{-1}_\varphi(\eta)$.
\end{enumerate}
\end{lem}
\begin{proof}
Setting $\Td^{-1}_\varphi(\xi)$ to be as in (1) for a $\xi = \Ls$ a line bundle, part (2) is a standard consequence of splitting principle.
\end{proof}

In particular, if $\Ec$ is a locally free sheaf of rank $r$, then
\begin{equation}\label{eq:TwistedTopChern}
\varphi(c_r(\Ec)) = \Td_\varphi^{-1}(\Ec) \cdot c_r(\Ec).
\end{equation}
To take advantage of this, we record the following useful formula for the Thom class.

\begin{lem}\label{lem:ChernClassFormula}
Let $\Ec$ be a locally free sheaf of rank $r$ on $X$ and let $E \in \MS^?_S$ be an oriented homotopy-commutative ring spectrum. Then, the pullback along $\Pb_X(\Ec \oplus \Oc) \to \Th_X(\Ec)$ sends 
\[
t(\Ec) \mapsto  c_r(\Ec(-1)) \in E^{2r,r}(\Pb_X(\Ec \oplus \Oc)).
\]
\end{lem}
\begin{proof}
By \cite[Eq.~(6.1)]{AHI}\footnote{We have changed the sign convention in the definition of the Thom class to get rid of the sign in front of the Euler class that is present in e.g. \cite[Eq.~(6.2)]{AHI}.}, 
\[
t(\Ec) \mapsto \sum_{i=0}^r (-1)^{r-i} c_i(\Ec) \cdot x^{r-i},
\]
where $x = c_1(\Oc(1))$. We want to identify this class with $c_r(\Ec(-1))$. As both classes are multiplicative with respect to summation in $K^0(X)$, it suffices to treat the case of a line bundle.

Using the formal group law of $E$ to compute $c_1(\Ls(-1))$ in terms of $c = c_1(\Ls)$ and $x$, we observe that
\[
c_1(\Ls(-1)) = (c-x) (1 + Ax),
\]
where $A = G(c,x)$ for some formal power series $G(a,b) \in E^{\star, \star}[[a,b]]$. As $c-x$ annihilates $x$, we see that $c_1(\Ls(-1)) = c-x$, which is exactly the desired formula for a line bundle.
\end{proof}

\begin{rem}
The convention of Chern classes used here is manifestly dual to that traditionally used in intersection theory \cite{fulton:1998}. Indeed, the embedding $\Pb(\Oc) \hook \Pb(\Ec \oplus \Oc)$ is the vanishing locus of the induced map $\Ec(-1) \to \Oc$, so the top Chern class of $\Ec$ is the Gysin pushforward along the zero locus of a cosection, not a section. For example, $c_1(\Oc(-1))$ is the hyperplane class in the cohomology of $\Pb^n$.

This reflects the convention of \cite{AHI} according to which $\Vb(\Ec)$ and $\Pb(\Ec)$ classify cosections and quotient line bundles of $\Ec$ and $\Ec$, in contrast to \cite{fulton:1998} where $\Vb(\Ec)$ and $\Pb(\Ec)$ classify of sections of $\Ec$ and sub line bundles sheaves of $\Ec$, respectively.
\end{rem}

\begin{lem}\label{lem:ThomClassFormula}
Let $\Ec$ be a vector bundle on $X$. Then
\[
\varphi(t(\Ec)) = \Td_\varphi^{-1}(\Ec) \cdot t(\Ec).
\]
\end{lem}
\begin{proof}
According to Lemma~\ref{lem:ChernClassFormula}, $t(\Ec)$ is a top Chern class. Thus, by Eq.~\eqref{eq:TwistedTopChern},
\[
\varphi(t(\Ec)) = \Td_\varphi^{-1}(\Ec(-1)) \cdot t(\Ec).
\]
Moreover, as $t(\Ec)$ annihilates $c_1(\Oc(-1))$, the claim follows from the splitting principle by computing the Chern classes of tensor products using the formal group law.
\end{proof}

\begin{lem}\label{lem:RRForEmbedding}
Let $i \colon Z \hook X$ be a quasi-smooth closed immersion in $\Sm^?_S$. Then, for all $z \in E^{\star, \star}(Z)$, 
\[
\varphi(i_!(z)) = i_!(\Td^{-1}_\varphi(\Nc_{Z/X}) \cdot \varphi(z)).
\]
\end{lem}
\begin{proof}
Let $\gys_i \colon X \to \Th_Z(\Nc_{Z/X})$ be the Gysin map in $\MS^?_S$ \cite{tang-gysin}. By definition,
\[
i_!(z) = \gys^*_i(t(\Nc_{Z/X}) \cdot \pi^*(z)),
\]
where $\pi \colon \Pb(\Ec \oplus \Oc) \to X$ is the structure map. As $\varphi$ commutes with pullbacks, the claim follows from the formula for $\varphi(t(\Nc_{Z/X}))$, which is Lemma~\ref{lem:ThomClassFormula}. 
\end{proof}
%

We can now state and prove the main result of this section, which is a general Riemann--Roch formula in $\MS^?_S$.

\begin{thm}[Riemann--Roch for motivic spectra]\label{thm:RRInMS}
Let $\varphi \colon E \to F$ be a possibly non-orientation-preserving map in $\CAlg(\rm{h}\MS^?_S)$. Let $f \colon X \to Y$ be a projective\footnote{Here, a \textit{projective} morphism is a map that factors through a closed embedding $i \colon X \hook \Pb^n_Y$ for some $n\geq 0$.} quasi-smooth morphism in $\Sm^?_S$, Then, for all $x \in E^{\star,\star}(X)$, 
\[
\varphi(f_!(x)) = f_!(\Td_\varphi(\Lbf_{X/Y}) \cdot \varphi(x)).
\]
\end{thm}
\begin{proof}
If $f$ is a closed embedding, then $\Lbf_{X/Y} \simeq \Nc_{X/Y}[1]$, so in this case the result is a reformulation of Lemma~\ref{lem:RRForEmbedding} in terms of the cotangent complex. As the formula is stable under compositions, it suffices to prove the formula in the case where $f$ is the structure map $\pi \colon \Pb^n_Y \to Y$. This can be done using the argument of \cite[Theorem~4.3.2]{deglise-orientations}, which only uses projective bundle formula and the basic properites of Gysin maps along derived regular embeddings.
\end{proof}

%

We now apply Theorem~\ref{thm:RRInMS} to the total power operations. Let $S = \Fb_p$ and consider the following object in $\MS_{\Fb_p}$
\[
H\Fb_p^{\mathrm{tot}} := \prod_{n \in \Zb} (\Tb^{\otimes n(p-1)} \otimes H\Fb_p) \oplus (\Tb^{\otimes n(p-1)} \otimes H\Fb_p[1]).
\]
This product is in fact a direct sum by \cite[Lemma~5.2]{HKO}. Thus, we can use the direct sum decomposition on the source to define a symmetric bilinear map $\mu^\tot \colon H\Fb_p^{\mathrm{tot}} \otimes H\Fb_p^{\mathrm{tot}} \to H\Fb_p^{\mathrm{tot}}$ degree-wisely by positing that
\[
(\Tb^{\otimes n(p-1)} \otimes H\Fb_p) \otimes (\Tb^{\otimes m(p-1)} \otimes H\Fb_p) \to \Tb^{\otimes (m+n)(p-1)} \otimes H\Fb_p
\]
and
\[
(\Tb^{\otimes n(p-1)} \otimes H\Fb_p) \otimes (\Tb^{\otimes m(p-1)} \otimes H\Fb_p [1]) \to \Tb^{\otimes (m+n)(p-1)} \otimes H\Fb_p[1]
\]
coincide with the usual multiplication on $H\Fb_p$, and that $\mu^\tot$ is null-homotopic on summands of form
\[
(\Tb^{\otimes n(p-1)} \otimes H\Fb_p[1]) \otimes (\Tb^{\otimes m(p-1)} \otimes H\Fb_p [1]). 
\]
The bilinear pairing $\mu^\tot$ equips $H\Fb_p^{\mathrm{tot}}$ with structure of a commutative algebra in ${\h}\SH_S$.

\begin{cons} The \textit{total power operation} is the map in $\SH_{\Fb_p}$
\begin{equation}\label{eq:total-p}
\P: H\Fb_p \rightarrow H\Fb_p^{\mathrm{tot}}
\end{equation}
given degree-wise by
\[
\P^n: H\Fb_p \rightarrow \Tb^{\otimes n(p-1)} \otimes H\Fb_p
\]
and
\[
\B^n: H\Fb_p \rightarrow \Tb^{\otimes n(p-1)} \otimes H\Fb_p[1].
\]
By the Cartan formula (Theorem~\ref{cor:ModPSteenrod}(2)), $P$ is a map in $\CAlg(\h\SH_{\Fb_p})$.

Applying \'etale sheafification on Eq.~\eqref{eq:total-p} in $\MS_{\Fb_p}$ yields a map
\begin{equation}\label{eq:total-p-syn}
\P: H\Fb_p^{\syn} \rightarrow (H\Fb^{\syn}_p)^{\mathrm{tot}}
\end{equation}
of commutative algebras in $\h\MS_{\Fb_p}$. We call this the \emph{total syntomic power operation}.
\end{cons}

Now let $S$ be an $\Fb_p$-scheme. Either by pullback in $\SH$ or by pullback in $\MS$ and \'etale sheafification we obtain the total power operations in
\[
\P:(H\Fb_p^{\Ab^1})_S \rightarrow (H\Fb_p^{\Ab^1})_S^\mathrm{tot},
\]
or
\[
\P: (H\Fb_p^{\syn})_S \rightarrow (H\Fb^{\syn}_p)_S^{\mathrm{tot}}
\]

\begin{cor} \label{cor:total-power-ops} Let $X \rightarrow Y$ be projective, quasi-smooth morphism over $S$. the formula
\[
\P(f_!(x)) = f_!(\Td_P(\Lbf_{X/Y}) \cdot \P(x)).
\]
holds in both mod-$p$ syntomic and $\Ab^1$-motivic cohomology. 
\end{cor}

\begin{rem}\label{rem:todd-vs-wu} Let $\Ec$ be a vector bundle on $\Sm_S^?$. Following definitions from topology, the \emph{$j$-th Stiefel-Whtiney class} of $\Ec$ is the class in motivic or syntomic cohomology given by
\[
w_j(\Ec) := t(\Ec)^{-1}(\P^j(t(\Ec)).
\]
The \emph{total Stiefel-Whitney class} of $\Ec$ is then defined by
\[
w(\Ec) := \sum_{j \geq 0} w_j(\Ec).
\]
By taking alternating sums, these classes are extended to any perfect complex on $X$. Applying Lemma~\ref{lem:ThomClassFormula} then the formula from Corollary~\ref{cor:total-power-ops} reads as
\begin{equation}\label{eq:sw}
\P(f_!(x)) = f_!(w(\Lbf_{X/Y}) \cdot \P(x)).
\end{equation}
\end{rem}

\appendix 

\section{Approximation of syntomic cohomology of classifying stacks}

In this appendix, we prove some approximation results for syntomic cohomology of classifying stacks. One of the main differences between syntomic and motivic cohomology is that the former is not $\mathbb{A}^1$-invariant, even on affine schemes. Hence, the methods of Totaro \cite{totaro-classifying} and Morel-Voevodsky \cite{morel:1999} of approximating the Chow groups and motivic cohomology groups of a stack by a scheme does not work on the nose. Instead, for each $i, j \geq 0$, we expect to build a map $X \rightarrow BG$ such that the induced pullback map on syntomic cohomology $H^i_{\mathrm{syn}}(-;\mathbb{F}_p(j))$ is injective. The following definition is found in \cite[Definition 5.1]{antieau-bhatt-mathew}:

\begin{defn}\label{def:Hodge-d} Let $R$ be a base ring. A morphism of $R$-algebraic stacks $\mathcal{X} \rightarrow \mathcal{Y}$ is a Hodge $d$-equivalence if for all $j \geq 0$, the have
\[
\mathrm{cofib}(R\Gamma(\mathcal{Y}; \mathbf{L}^j_{-/R}) \rightarrow R\Gamma(\mathcal{X}; \mathbf{L}^j_{-/R}) ) \in \mathrm{D}(R)^{\geq d-j}.
\]
\end{defn}

In this paper, the mod-$p$ syntomic cohomology of a stack is defined via right Kan extension, i.e. 
\[
R\Gamma_{\mathrm{syn}}(\mathcal{X}; \mathbb{F}_p(j)) \simeq \lim_{\Spec(R) \rightarrow \mathcal{X}} R\Gamma_{\mathrm{syn}}(\Spec R; \mathbb{F}_p(j)). 
\]
This is a relatively easy procedure and the expected properties of syntomic cohomology continues to hold, primarily because syntomic cohomology is an an fpqc sheaf, whence its right Kan extension is of universal descent with respect to fpqc surjections (see, for example, \cite[Proposition 4.2.2]{eks-1}). For example, if $\mathcal{X}$ is an algebraic stack (quasi-compact for simplicity), then it admits an atlas (i.e. an \'{e}tale surjection) $\Spec(R) \rightarrow \mathcal{X}$ and syntomic cohomology of $\mathcal{X}$ is calculated by the totalization of the syntomic cohomology of the \v{C}ech nerve:
\begin{equation}\label{eq:cech}
R\Gamma_{\mathrm{syn}}(\mathcal{X}; \mathbb{F}_p(j)) \simeq \lim_{\Delta} R\Gamma_{\mathrm{syn}}((\Spec R)^{\times_{\mathcal{X}} \bullet}; \mathbb{F}_p(j)). 
\end{equation}

\begin{lem}\label{lem:cofib-conn} Let $k$ be a perfect field of characteristic $p > 0$. Let $\mathcal{X} \rightarrow \mathcal{Y}$ be a morphism of $k$-algebraic stacks which is a Hodge-$d$-equivalence. Then we have that
\[
\mathrm{cofib}(R\Gamma_{\mathrm{syn}}(\mathcal{Y}; \Fb_p(j)) \rightarrow R\Gamma_{\mathrm{syn}}(\mathcal{X}; \mathbb{F}_p(j)) \in \mathrm{D}(\Fb_p)^{\geq d}.
\]
\end{lem}

\begin{proof} The result follows from the filtration on mod-$p$ syntomic cohomology in characteristic $p > 0$ in \cite[Lemma 4.16]{elmanto-morrow}. Indeed, write $C(j)$ for the cofibre of the map $\mathrm{cofib}(R\Gamma(\mathcal{Y}; \mathbb{L}^j_{-/R}) \rightarrow R\Gamma(\mathcal{X}; \mathbb{L}^j_{-/R}) )$. Then the filtration in \emph{loc. cit.} induces a finite filtration on the cofibre of interest  whose graded pieces are

\begin{align*}
&C(j)[-j-1], C(j-1)[-j],C(j-2)[-j+1],\dots, C(0)[-1],\\
& C(0)[0], C(1)[-1], C(2)[-2],\dots, C(j)[-j].
\end{align*}

By the assumption Hodge $d$-equivalence, we see that the graded pieces on the top row are concentrated in cohomological degrees $\geq d+1$, while the graded pieces on the bottom row are concentrate in cohomological degrees $\geq d$, from which we conclude the statement of the lemma.
\end{proof}

\begin{thm}\label{thm:approximate} Let $k$ be a perfect field $k$ of positive characteristics. For any finite group scheme $G$ and any $d \geq 0$, there exists smooth projective scheme $d$-fold $X$ and a map $X \rightarrow \mathrm{B}G$ such that for all $j \geq 0$, and all $i \leq d$ the map
\[
H^{i}_{\mathrm{syn}}(\mathrm{B}G; \mathbb{F}_p(j)) \rightarrow H^{i}_{\mathrm{syn}}(X;\mathbb{F}_p(j)),
\]
is injective.
\end{thm}

\begin{proof} As in the proof of \cite[Theorem 1.2]{antieau-bhatt-mathew}, we can find a $k$-linear representation $V$, a $d$-dimensional complete intersection $Z \subset \mathbb{P}(V)$ such that $Z$ is $G$-stable, $G$ acts freely on $Z$ and $Z/G \simeq [Z/G]$ is smooth and projective. For such a $Z$, \cite[Proposition 5.3 and 5.10]{antieau-bhatt-mathew} implies that the induced map $[Z/G]\simeq Z/G \rightarrow [\mathbb{P}(V)/G]$ is a Hodge $d$-equivalence. Therefore, by Lemma~\ref{lem:cofib-conn}, the induced map on syntomic cohomology has cofibres in degrees $\geq d$. The claim then follows from the projective bundle formula in syntomic cohomology; we remark that this result for schemes is \cite[Theorem 9.1.1]{bhatt-lurie:apc} and the result follows by Kan extension for stacks which, in particular, shows that $H^{i}_{\mathrm{syn}}(\mathrm{B}G;\mathbb{F}_p(j))$ is a summand of $H^{i}_{\mathrm{syn}}([\mathbb{P}(V)/G];\mathbb{F}_p(j))$. Now, noting that $\Spec(k) \rightarrow \mathrm{B}G$ is an fpqc surjection, which pulls-back along $[\mathbb{P}(V)/G] \rightarrow \mathrm{B}G$ to the fpqc surjection $\mathbb{P}(V) \rightarrow [\mathbb{P}(V)/G]$ the projective bundle formula follows from the one for schemes by the formula~\eqref{eq:cech}.
\end{proof}

\section{Calculations in syntomic and motivic cohomology of certain classifying stacks} In this appendix, we study the action of the Milnor operations~\eqref{eq:milnor} on the motivic and syntomic cohomology of $B\mu_p^{\times s}$. Our discussion is based on certain well-known formulas in topology and we refer to \cite{quick-bp} for details where these formulas are derived in the topological setting.

Throughout this section, we fix an perfect field $k$ of characteristic $p > 0$. For an algebraic stack $\Xfr$ we define $\Fb_p(j)^{?}(\Xfr)$ where $? = \mot$ or $\syn$ via right Kan extension from schemes. We write 
\[
H^{\star,\star}:= H^{\star}_{?}(k; \Fb_p(\star)),
\]
the bigraded cohomology ring of $k$ where $? = \syn$ or $\mot$. As already recorded in Example~\ref{ex:CohOfBmup} we know how the cohomology of $B\mu_p^{\times s}$ takes the form:
\[
H^{\star}_{?}(B\mu_p^{\times s}; \Fb_p(\star)) \cong H^{\star,\star}[[u_1,v_1, \cdots, u_s, v_s]]/(u_i^2 = 0) \qquad |u_i| = (1,1), |v_i| = (2,1).
\]
Furthermore, we have that $\beta(u_i) = v_i$. 

\begin{rem}[Oddities at non-odd prime] We remark that the motivic and syntomic cohomology of $B\mu_p^{\times s}$ looks the same across all primes $p$. This contrasts against the situation in topology where the case of $p =2$ is the odd one out where we have the relation that $u_i^2 = v_i$. 
\end{rem}

We had defined 
\[
\Q_i = q_i \beta - \beta q_i,
\]
where $\beta$ is the Bockstein and 
\[
q_i = \P^{p^{i-1}} \cdots \P^p \P^1.
\]
The equations equally make sense as operations on syntomic cohomology. We note:

\begin{lem}\label{lem:milnor} We have the following equalities that hold for both motivic and syntomic cohomology:
\begin{enumerate}
\item $\Q_{n+1} = \P^{p^n}\Q_n - \Q_n\P^{p^n}$ for $n \geq 0$;
\item $\Q_n(xy) = \Q_n(x)y + (-1)^{|x||\Q_n|}x\Q_n(y)$.
\end{enumerate}
Here, $|-|$ refers to the topological/simplicial degree of a class/operation.
\end{lem}

\begin{proof} We first give references for these equalities in the characteristic zero field $\Kb$. If $p$ is odd, then both equalities is \cite[Lemma 13.11]{haesemyer-weibel} (note that there should actually be a sign in the formulas because it is the case in topology and \emph{op. cit.} obtained these formulas from the topological Steenrod algebra); note that in \cite[Definition 13.10]{haesemyer-weibel} the $\Q_n$'s are defined inductively by the right hand side of (1) and the cited reference proves the equality with our definition of $\Q_n$. If $p = 2$, then we note that \cite[Example 13.7]{voevodsky:2003a} cautions that the commutator formula might not hold; it might also not be a derivation by \cite[Remark 13.4.1]{haesemyer-weibel}. However, the failure of both equalities are due to the presence of the element $\rho$ in motivic cohomology; such an element is zero if $-1$ is a sum of squares in the base field. We note that by examining \cite[Corollary 13.14, Remark 13.14.1]{haesemyer-weibel} both equalities hold in the case of $p =2$ modulo $\rho$. By the construction of our operations in characteristic $p > 0$, both equalities now hold for all primes $p$ since $\rho$ is zero in characteristic $p > 0$. The result for syntomic cohomology follows immediately. 
\end{proof}

We now study the effect of $\Q_n$ on the generators $v_i$ and $u_i$.

\begin{lem}\label{lem:Q_n-bmup} The following equalities hold:
\begin{enumerate}
\item $\Q_n(u_i) = v_i^{p^n}$
\item $\Q_n(v_i) = 0$.
\end{enumerate}
\end{lem}

\begin{proof} At all primes, including the prime $p =2$, the derivation in \cite[Lemma 3.5]{quick-bp} works and only uses basic properties of the Milnor operations. Note that Lemma~\ref{lem:milnor} guarantees the definition of $\Q_n$ used by Quick (via induction) agrees with ours. 
\end{proof}

In the main text, the following proposition was used at various places.

\begin{prop}\label{prop:qi} Let $\lambda \in H^{\star,\star}$ then
\[
\Q_n\Q_{n-1} \cdots \Q_0(\lambda u_1 \cdots u_m) \not= 0 \qquad m \geq n+1.
\]
\end{prop}

\begin{proof} Since the operations are linear over the $H^{\star,\star}$, it sufices to prove that 
\[
\Q_n\Q_{n-1} \cdots \Q_0(u_1 \cdots u_m) \not=0.
\] But this is \cite[Lemma 3.6]{quick-bp}.
\end{proof}

\section{Deglise--Jin--Khan fundamental class in higher Chow groups}\label{sec:djk-higher}

Given a regular embedding $i \colon Z \hook X$ of codimension $r$, Deglise--Jin--Khan \cite{DJKFundamental} construct a fundamental class $\eta_i \in E^\mathrm{BM}(Z/X, -\Nc_{i})$ in the bivariant $E$-theory group of $i$, where $E \in \SH_X$. Here, we explicitly compute this class for the motivic cohomology spectrum $H\Zb^{\Ab^1}$ when $X$ is smooth over a Dedekind domain $A$. In this situation, the localization theorem of Levine \cite[Theorem~1.7]{levine:2001} identifies  $H\Zb^{\Ab^1,\mathrm{BM}}(Z/X, -\Nc_{i})\simeq H\Zb^{\Ab^1,\mathrm{BM}}(-r)(Z/X)$ with the cycle complex of Bloch $z^0(Z,\bullet) \in D(\Zb)$. The homology groups of this complex are the \textit{Higher chow groups} of $Z$, $\CH^r(Z,n)$, which are isomorphic to the motivic cohomology of $Z$ if $Z$ is either regular and equicharacteristic, or smooth over $A$. We refer to \cite{geisser:2005} for an introduction to motivic cohomology from the cycle-theoretic point of view.

We will repeat the construction of the fundamental class from \cite[\S 3.2]{DJKFundamental} for higher Chow groups; in what follows the base is $\Spec(\Zb)$. The construction of the fundamental class is based on the specialization to the normal cone map \cite[Definition~3.2.4]{DJKFundamental} $\sigma_i$, which is by definition the composition
\begin{equation}\label{eq:specializationtoN}
z^0(X,\bullet) \xto{\gamma_t} z^1(\Gb_m \times X, \bullet)[-1] \xto{\partial} z^0(\Nc_i, \bullet),
\end{equation}
where $\partial$ is the connecting map of the localization cofibre sequence 
\[
z^n(\Nc_i, \bullet) \to z^{n+1}(D_iX, \bullet) \to z^{n+1}(\Gb_m \times X, \bullet)
\]   
associated to the affine deformation to the normal cone space $D_iX$ \cite[\S 3.2.3]{DJKFundamental}. Next, we identify the other map in this composition.

\begin{figure}[h!]
\begin{center}
\includegraphics[scale=1.4]{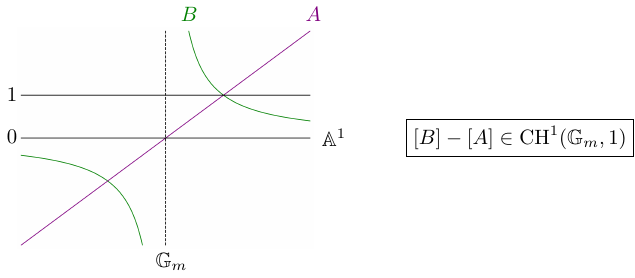}
\end{center}
\caption{
\textbf{Illustration of the class $\{t\} \in \CH^1(\Gb_m, 1) = H^1_{\Ab^1}(\Gb_m, \Zb(1))$.} The cycles $[A]$ and $[B]$ are the graphs of the two functions $\lambda \to \lambda$ and $\lambda \to \lambda^{-1}$ from $\Gb_m$ to $\Ab^1$, respectively. The two cycles are elements of $z^1(\Gb_m,1)$ because each of them meets the two faces, represented by the horizontal lines at $0$ and $1$, transversely. The difference $[B]-[A]$ is a cycle in the cycle complex, and therefore it gives rise to a class $\{t\} := [B]-[A] \in \CH^1(\Gb_m, 1)$.
}\label{fig1}
\end{figure}

\begin{lem}
The map $\gamma_t$ from Eq.~\eqref{eq:specializationtoN} is corresponds to multiplication by the class $\{t\} \colon \Zb[1] \to z^1(\Gb_m,\bullet)$ depicted in Fig.~\ref{fig1}.
\end{lem}
\begin{proof}
By \cite[\S 3.2.2]{DJKFundamental}, we have to check two conditions:
\begin{enumerate}
\item $\{t\}$ is sent to the cycle $[\Spec(\Zb)] \in z^0(\Spec(\Zb),\bullet)$ by the connecting map $\partial'$ of the localization sequence 
\[
z^0(\Spec(\Zb), \bullet) \to z^1(\Ab^1, \bullet) \to z^1(\Gb_m, \bullet);
\]
\item $\{t\}$ is pulled back to a nullhomotopic map $\Zb[1] \to z^1(\Spec(\Zb),\bullet)$ along the map $1 \colon \Spec(\Zb) \to \Gb_m$.
\end{enumerate}
The validity of the second condition is evident, as $B-A$ pulls back to 0 as an algebraic cycle. The validity of the second condition follows from the fact that the algebraic cycle $B - \bar A \in z(\Ab^1,1)$ intersects the vertical axis at a single point---the origin. Above, $\bar A$ is the closure of the algebraic cycle $A \subset \Gb_m \times \Ab^1$ inside $\Ab^1 \times \Ab^1$. Thus, by the definition of $\partial'$, we have that $\partial'\{t\} = [\Zb]$ as desired.
\end{proof}

Now that we have identified the two maps in the definition of the specialization to the normal cone map Eq.~\eqref{eq:specializationtoN}, we are ready to identify the fundamental class. By definition, $\eta_i \in z^0(Z, \bullet)$ is the pullback of $\sigma_i([X]) \in z^0(\Nc_i, \bullet)$ along the zero-section $Z \hook \Nc_i$. We are now ready to prove the main result of this section.

\begin{prop}\label{prop:DJKFundamentalClassForHigherCH}
Let $i\colon Z \hook X$ be a regular embedding, where $X$ is a smooth scheme over a Dedekind domain and $Z$ is integral. Then, the fundamental class $\eta_i \in z^0(Z, \bullet)$ coincides with $[Z]$.
\end{prop} 
\begin{proof}
By the discussion of the paragraph preceding the statement, it suffices to show that $\sigma_i([X]) = [\Nc_i]$. To see this, we notice that the closure of the cycle $B \times X - A \times X$ inside $\Ab^1 \times D_iX$ pulls backs to 0 along the face at $1$, and to $-[\Nc_i] \in z^1(D_iX;\bullet)$. Thus, by the definition of the connecting map $\partial$ in Eq.~\ref{eq:specializationtoN}, we see that $\sigma_i([X]) = [\Nc_i]$.
\end{proof}

\clearpage

\bibliographystyle{alphamod}
\bibliography{references}

\end{document}